\documentclass{amsart}
\usepackage{amsmath,amssymb,amsthm,graphicx,epsfig,float,url}
\usepackage[colorlinks=true,citecolor={Plum},linkcolor={Periwinkle}]{hyperref}
\usepackage{pdfsync}
\usepackage[usenames, dvipsnames]{xcolor}
\usepackage{tikz}
\usepackage{subfig}
\usepackage{stmaryrd}
\usepackage{amssymb}
\usepackage{mathrsfs}  
\usepackage{float}
\usepackage{esint}
\usepackage{verbatim}
\usepackage{geometry}
\usetikzlibrary{decorations.pathmorphing}
\usetikzlibrary{patterns.meta, plotmarks,calc}

\usepackage{mathtools}

\usepackage{dsfont}
\usepackage[makeroom]{cancel}

\usepackage{todonotes}
\newcommand{\R}{\textnormal{I\kern-0.21emR}}
\newcommand{\N}{\textnormal{I\kern-0.21emN}}

\def\B{{\mathbb B}}
\def\A{{\mathbb{A}}}

\def\O{{\Omega}}
\def\n{{\nabla}}
\def\p{{\varphi}}

\def\T{{\mathbb T^d}}
\def\e{{\varepsilon}}
\newtheorem{theoremal}{Theorem}

\newtheorem{theorem}{Theorem}

\newtheorem{material}{material}[section]
\newtheorem{proposition}[material]{Proposition}
\newtheorem{corollary}[material]{Corollary}
\newtheorem{definition}[material]{Definition}
\newtheorem{lemma}[material]{Lemma}

\newtheorem{remark}[material]{Remark}

\numberwithin{equation}{section}
\numberwithin{section}{part}

\title[Free boundary problems in optimal control]{Unstable free boundary problems in optimal control theory: existence and regularity}
\author{Lorenzo Ferreri}
\address{Scuola Normale Superiore, Piazza dei Cavalieri, 7, 56126 Pisa, Italy}\email{lorenzo.ferreri@sns.it}
\author{Idriss Mazari-Fouquer}
\address{CEREMADE, UMR CNRS 7534, Universit\'e Paris-Dauphine, Universit\'e PSL, Place du Mar\'echal De Lattre De Tassigny, 75775 Paris cedex 16, France.}
\email{mazari@ceremade.dauphine.fr}%, 
\author{Rapha\"{e}l Prunier}
\address{CEREMADE, UMR CNRS 7534, Universit\'e Paris-Dauphine, Universit\'e PSL, Place du Mar\'echal De Lattre De Tassigny, 75775 Paris cedex 16, France.}
\date{\today}

\begin{document}

\maketitle

\begin{abstract}
We establish the first general regularity result for constrained optimal control problems arising naturally in mathematical physics and mathematical biology. Namely, we prove that for a large class of problems of the form ``maximise $\int \psi(\Theta_m)-c\int m$ where $-\Delta \Theta_m=m\Theta_m+B(x,\Theta_m)$, under the constraint $0\leq m\leq 1$ a.e.", the solution $m^*$ is bang-bang, in the sense that $m^*=\mathds 1_{E^*}$, and that $\partial E^*$ is smooth up to a $(d-2)$-dimensional subset. Moreover, we prove that the solutions to the volume constrained problem ``maximise $\int \psi(\Theta_m)$ where $-\Delta \Theta_m=m\Theta_m+B(x,\Theta_m)$, under the constraint $0\leq m\leq 1$ a.e and $\int m=m_0$" are bang-bang in the sense that $m^*=\mathds 1_{E^*}$ and that, in the two-dimensional case, $\partial E^*$ is a finite union of smooth curves. This is done via reduction to an unstable free boundary problem, the regularity analysis of which was pioneered by Monneau \& Weiss and Chanillo, Kenig \& To. In our case, the free boundary is not minimising, and the laplacian of the state function is sign-changing, which creates significant difficulties, in particular regarding the non-degeneracy of blow-ups. This requires a new approach blending tools from optimal control theory, free boundary and measure theory to establish the regularity of the free boundary. 

\end{abstract}

\textbf{Keywords:}  Free boundary problems, optimal control, regularity theory, monotonicity formulas.

\textbf{AMS-Subject Classification:} 35R35, 35J57, 35Q92, 49J30.

\textbf{Acknowledgement:}  I. M.-F. and R. P. are supported by the PSL Young Researcher Starting Grant (P.I: I. M.-F.) ``Optimal control of ecological systems". L.F. is partially supported by the INdAM - GNAMPA Project ``Fine structure and regularity in nonlinear variational problems" CUP E53C25002010001.

The authors thank  R. Monneau and P. Pegon for numerous fruitful conversations, as well as A. Aussedat and F. Paiano for many comments and remarks on a preliminary version of this article.

\tableofcontents
\part{Introduction}
\section{Main results}
\subsection{Scope of the paper and rough statement of the main results}\label{Se:Goal}
Our main objective is twofold: first, showing that a large class of bilinear optimal control problems can be recast as free boundary problems of unstable (or inverse) obstacle type. Second, studying the regularity properties associated with these free boundary problems. Let us note that, as these free boundary problems do not arise from an energy minimisation, several of the tools and approaches used in the standard theory of inverse obstacle problems fail to apply.  We propose what is to the best of our knowledge the first general regularity results valid for optimal control problems. As the precise statements require some background we first give a rough statement of our results: we consider
\begin{equation}\label{Eq:PvGeneric} \text{Maximise over $m$ the functional }J(m)=\int_{\T}\psi(\Theta_m)\text{ where }-\Delta \Theta_m=m\Theta_m+B(x,\Theta_m).\end{equation} In the PDE above, the function $m\in L^\infty(\T)$ is the control and $B$ is a fixed non-linearity, $\Theta_m$ is the state function and $\psi$ is the cost function. The interaction  between control and state is given by $m\cdot \Theta_m$, which is why this setting is referred to as ``bilinear" (see Section \ref{Ap:Necessity} for other control-state interactions). Such couplings appear naturally in mathematical physics, where $m$ should be thought of as a potential (leading to the so-called ``composite membrane problem" \cite{Cox1990}), and in biology, where $m$ models the heterogeneity of the domain \cite{zbMATH07668634}.  We defer a detailed discussion of the relevant literature to Section \ref{Se:Bibliography}. Both in biology and in physics it is reasonable to assume that $m$ belongs to the admissible set 
\begin{equation}\label{Eq:Admissible0} \mathcal M_0(\T):=\left\{m\in L^\infty(\T):\, 0\leq m\leq 1\text{ a.e.}\right\}.\end{equation} 
This first set of (pointwise) constraints is supplemented with either a (global) integral constraint, in which case the admissible class reads 
\begin{equation}\label{Eq:Admissible} m\in	\mathcal M(\T):=\left\{m\in L^\infty(\T):\, 0\leq m\leq 1\text{ a.e.},\, \int_\T m=m_0\right\},\end{equation}and the problem writes 
\[ \tag{Constrained problem} \max_{m\in \mathcal M}J(m),\] or with an $L^1$ penalisation, meaning that we focus on
\[\tag{Penalised problem} \max_{m\in \mathcal M_0}  J(m)-c\int_\T m  \] for some parameter $c$.

The main results of this article are the following (see Theorem \ref{Th:Main} for a precise statement): if the functional $J$ is monotone with respect to the control (meaning that $m\leq m'$ a.e. implies $J(m)\leq J(m')$) then the following hold:
\begin{enumerate}
\item \underline{Bang-bang property:} for either penalised or constrained optimisation, any optimiser $m^*$ writes $m^*=\mathds 1_{E^*}$ for some subset $E^*$ of $\T$.
\item\underline{Regularity for the problem in the two-dimensional case:} for either penalised or constrained optimisation, in two dimensions, $\partial E^*$ (in a sense to be made precise) is composed of a finite number of $\mathscr C^{2,\gamma}$ curves, for any $\gamma\in (0;1)$. 
\item \underline{Regularity in all dimensions:} in the penalised case, in any dimension $d$, $\partial E^*$ is $\mathscr C^{2,\gamma}$ up to a $(d-2)$-dimensional set.  We comment on the constrained case in Section \ref{Se:Open}.
\end{enumerate}
\begin{remark}
In general, there is no equivalence between the constrained and the penalised version
 of the problem. Indeed, the monotonicity of $J$ implies some small scale convexity of $J$ (see Proposition \ref{Pr:LocalMinimality}), and it is fairly well-known that penalised problems are not equivalent to constrained ones for convex maximisation. In one dimension for instance, one can construct convex polynomials $f$ and $g$ such that the constrained problem $\max_{\{g\leq 0\}}f$ is not equivalent to the penalised problem $\max (f+\lambda g)$, $\lambda$ being the associated Lagrange multiplier, in the sense that the maximisers for the two problems do not coincide.
\end{remark}
In this paper we establish a link between such optimal control problems and so-called ``inverse obstacle problems", which is done through the use of the switching function in optimal control theory. In particular, half of the paper is devoted to boiling the problems down to this framework and ties to several recent contributions \cite{zbMATH07812267}, while the second half has more to do with unstable free boundary  problems \cite{zbMATH05505659,zbMATH05129536,zbMATH06912136,zbMATH07826622}, wherein one studies solutions to 
\[-\Delta \eta=f(x)\mathds 1_{\{\eta>0\}}-g(x)\mathds 1_{\{\eta\leq 0\}},\] where $f$ and $g$ are two given functions. We refer to Section \ref{Se:Bibliography} for a discussion of the references.

 \subsection{Mathematical context and main results}
 \subsubsection{The main bilinear control problem} $\mathcal M_0,\, \mathcal M$ are defined in \eqref{Eq:Admissible0}--\eqref{Eq:Admissible}. Consider a non-linearity $B=B(x,u)$ satisfying the following assumption:
\begin{equation}\label{Eq:AssB}\tag{$A_1$}
\begin{cases}
B\in L^\infty(\T;\mathscr C^2(\R))\,, 
\\ B(x,0)=0\,, 
\\ \forall x \in \T\,, u\mapsto \frac{B(x,u)}{u} \text{ is decreasing on }(0;+\infty)\,, 
\\ \lim_{u\to \infty}\sup_{x\in \T} \frac{B(x,u)}{u}=-\infty.
\end{cases}
\end{equation}
For any $m\in \mathcal M_0$, let $\Theta_m$ be defined as the unique solution of 
\begin{equation}\label{Eq:LDE2}
\begin{cases}
-\Delta \Theta_m=m\Theta_m+B(x,\Theta_m)&\text{ in }\T\,, 
\\ \Theta_m\geq 0\,, \not\equiv 0.\end{cases}\end{equation}
\begin{remark}
We could work in any smooth domain $\O$ with Neumann, Robin or Dirichlet boundary conditions. Similarly, as will be clear in the proof, our results also apply to non-symmetric, non homogeneous operators with smooth coefficients.
\end{remark}
The well-posedness of \eqref{Eq:LDE2} follows from standard sub- and super-solution techniques \cite{zbMATH02194918}; we summarise the main findings in the following:
\begin{theoremal}\cite{zbMATH02194918}\label{Th:LDE2}
Assume that \eqref{Eq:AssB} is satisfied. For any $m\in \mathcal M_0$, there exists a unique solution $\Theta_m$ to \eqref{Eq:LDE2}. For any $p\in [1;+\infty)$, $\Theta_m\in W^{2,p}(\T)$.
\end{theoremal}
In the present paper, the bilinear optimal control problems that we consider are
\begin{equation}\label{Eq:PvNonEnergetic12}\tag{$Q_{\mathrm{constr}}$}
\max_{m\in \mathcal M}\int_{\T} \psi(x,\Theta_m)\quad  \text{  (constrained problem) }
\end{equation}  and 
\begin{equation}\label{Eq:PvNonEnergetic1}\tag{$Q_{\mathrm{penal}}$}
\max_{m\in \mathcal M_0}\int_{\T} \psi(x,\Theta_m)-c\int_\T m\quad\text{ (penalised problem)}
\end{equation} where $\Theta_m$ solves \eqref{Eq:LDE2} and we assume that \eqref{Eq:AssB} holds. If $\psi$ is continuous in the $\Theta$ variable (which will be one of the working assumptions in the paper, see \eqref{Eq:AssJ1} below), the existence of solutions $m^*$ to \eqref{Eq:PvNonEnergetic1} or \eqref{Eq:PvNonEnergetic12} follows from the direct method in the calculus of variations. In order to characterise these optimisers, we make the following further assumption on $B$ (essentially ensuring the differentiability of the solution mapping $m\mapsto \Theta_m$):
 for any $m\in \mathcal M$, the lowest eigenvalue $\omega(m)$ of the operator 
\[ -\Delta \Theta_m-m-\partial_\Theta B(x,\Theta_m)\] satisfies
\begin{equation}\label{Eq:AssB2}\tag{$A_2$}
\omega(m)>0.\end{equation}

\begin{remark}The prototypical example of such a problem is the optimisation of the total population size in population dynamics, see Section \ref{Se:Bibliography}, where $\Theta_m$ is interpreted as a local density of individuals, and $m$ as a resources distribution. The non-linearity $B$ is $B(x,\theta)=-\theta^2$, and $\psi(x,u)=u$.\end{remark}

Regarding the cost function $\psi$, we make the following monotonicity assumptions:
\begin{equation}\label{Eq:AssJ1}\tag{$A_3$}
\begin{cases}
\psi\in L^\infty(\T;\mathscr C^2((0;+\infty)))\,, 
\\ \partial_\Theta \psi(x,\Theta)\geq 0\text{ for a.e. $x$, for any $\Theta\geq 0$,}
\\\forall \Theta>0\,, |\{x \in \T,\, \partial_\Theta \psi(x,\Theta)>0\}|>0.
\end{cases}
\end{equation}
Before stating our main result, we recall the definition of measure theoretic boundary $\partial^e A$ of a set $A \subset \T$, that is
\begin{equation}\label{De:Boundary} \partial^e A:=\left\{x \in \T:\, \forall r>0,\, |A\cap \B(x;r)|\cdot |\left( \T \setminus A \right) \cap \B(x;r)|>0\right\}.
\end{equation}
The main result of this paper is the following:
\begin{theorem}[First formulation of the regularity result]\label{Th:Main1} Assume \eqref{Eq:AssB}--\eqref{Eq:AssB2}--\eqref{Eq:AssJ1} hold. Then:
\begin{enumerate}
\item \underline{Regularity for the constrained problem:} any solution $m^*$ to \eqref{Eq:PvNonEnergetic12} is a characteristic function, meaning that $m^*=\mathds 1_{E^*}$. Furthermore,
the measure theoretic boundary $\partial^e E^{\ast}$ of $E^*$ is a finite union of $\mathscr C^{2,\gamma}$ curves if $d=2$.
\item\underline{Regularity for the penalised problem:} let $c>0$. Any solution $m^*$ to \eqref{Eq:PvNonEnergetic1} is a characteristic function, meaning that $m^*=\mathds 1_{E^*}$. Furthermore, defining $\partial^e E^*$ as in \eqref{De:Boundary}, the following holds:\begin{enumerate}
\item If $d=2$, $\partial^e E^*$ is a finite union of $\mathscr C^{2,\gamma}$ curves.
\item In dimension $d\geq 3$, $\partial^e E^*$ is $\mathscr C^{2,\gamma}$ up to a $(d-2)$-dimensional set.
\end{enumerate}

\end{enumerate}
\end{theorem}

\begin{remark}[On the difference between the penalised and the constrained problems]
Our analysis relies on blow-up procedures.
One of the starting points in deriving regularity properties for the free boundary is to establish non-degeneracy of the blow-up limits. For us, this is in fact one of the main difficulties: as shall be made clearer in the rest of the paper, we can unfortunately not rely on the usual structural assumptions (\emph{e.g.} minimality, super-harmonicity) to obtain this property. Here, our approach is non-standard and blends tools from optimal control, free boundary, Fourier analysis and measure theory. In particular, the set of perturbations we are allowed to make has a strong influence on the optimality conditions we derive. As there are far less admissible perturbations for the constrained problem \eqref{Eq:PvNonEnergetic12} than for the penalised one \eqref{Eq:PvNonEnergetic1}, it is not surprising \emph{a priori} that we can get stronger results for the latter. 
Indeed, as regards \eqref{Eq:PvNonEnergetic12} at the moment, our strategy is briefly the following: we first establish the non-degeneracy of blow-up limits at specific points of $\partial^eE^*$, namely points of intermediate density (see Definition \ref{De:IntermediateDensity}), and we then devise a strategy to propagate the regularity at these points to the entire free boundary $\partial^eE^*$ (see Section \ref{Sec:MainIdeasProof} for more details on the strategy). At the moment, this last step only works in dimension $d=2$ (see Section \ref{Se:2DCKT}, and in particular Section \ref{Se:Why2D} for further details on the the restriction to the two-dimensional case).
   \end{remark}

Structurally, the only two features we rely on to derive Theorem \ref{Th:Main1} are the bilinearity of the control-state interaction, and the monotonicity of the functional. In fact, the bilinearity, combined with \eqref{Eq:AssJ1} entail the following monotonicity property: for any $m,\, m'\in \mathcal M_0$, 
\[ m\leq m'\text{ a.e. }\Rightarrow J(m)\leq J(m').\] Without either monotonicity or bilinearity, the solutions to \eqref{Eq:PvNonEnergetic1}--\eqref{Eq:PvNonEnergetic12} are not even characteristic functions, see Section \ref{Ap:Necessity}. 

 \subsubsection{Reformulation as a Hamilton-Jacobi equation}

In order to introduce the free boundary system associated with \eqref{Eq:PvNonEnergetic1} it is convenient to rewrite the state equation \eqref{Eq:LDE2} as a Hamilton-Jacobi equation. Observe that for any $m\in \mathcal M_0$, setting $\theta_m:=\ln(\Theta_m)$, this new function satisfies \begin{equation}
\label{Eq:LDE3}
-\Delta \theta_m-|\n \theta_m|^2=m+Q(x,\theta_m)\text{ in }\T,\end{equation} where $Q(x,\theta_m):=\frac{B(x,\Theta_m)}{\Theta_m}$. In fact, we will work with general Hamilton-Jacobi equations of the form \eqref{Eq:LDE3} rather than with \eqref{Eq:LDE2}. The assumptions \eqref{Eq:AssB}--\eqref{Eq:AssB2} on $B$ are replaced with the following ones on $Q$: $Q=Q(x,u)\in L^\infty(\T; \mathscr C^2( \R))$ satisfies
\begin{equation}\label{Eq:AssQ}\tag{$H_1$}
\begin{cases}
\forall u\in \R,\, \max_{x\in \T}\partial_u Q(x,u)<0, 
\\  \lim_{u\to +\infty}\sup_{x\in \T} Q(x,u)=-\infty,
\\ \lim_{u\to -\infty}\sup_{x\in \T} \left| Q(x,u)\right|=0.
\end{cases}
\end{equation}
Similar to Theorem \ref{Th:LDE2}, adapting the proofs of \cite{zbMATH02194918},  we have the following result:
\begin{theoremal}\label{Th:ExLDE}
Under Assumption \eqref{Eq:AssQ}, for any $m\in \mathcal M_0$, there exists a unique solution $\theta_m$ to \eqref{Eq:LDE3}. Furthermore, for any $p\in [1;+\infty)$, $\theta_m\in W^{2,p}(\T)$.
\end{theoremal}
The optimal control problems \eqref{Eq:PvNonEnergetic12}--\eqref{Eq:PvNonEnergetic1} are replaced 
 with 
\begin{equation}\label{Eq:PvNonEnergetic}\tag{$P_{\mathrm{constr}}$}
\max_{m\in \mathcal M}\left(J(m)=\int_{\T} j(x,\theta_m)\right)
\end{equation} and
\begin{equation}\label{Eq:PvNonEnergetic2}\tag{$P_{\mathrm{penal}}$}
\max_{m\in \mathcal M_0}\left(J(m)=\int_{\T} j(x,\theta_m)-c\int_\T m\right).
\end{equation}

 Assumption \eqref{Eq:AssB2}, ensuring differentiability of the solution mapping, is replaced with: \begin{multline}\label{Eq:AssQ2}\tag{$H_2$}
\text{The principal eigenvalue $\Lambda(m)$ of }L_m:=-\Delta+2\n\cdot(\cdot\n\theta_m)-\partial_{\theta}Q(x,\theta_m)\text{ satisfies }\Lambda(m)>0. 
 \end{multline}
 The existence of $\Lambda(m)$ follows from \cite{zbMATH06458598}.
Finally, we assume that $j$ satisfies \begin{equation}\label{Eq:AssJ}\tag{$H_3$}\begin{cases}\forall x \in \T,\, \partial_\theta j(x,\theta)\geq 0\text{ in }(-\infty,\infty),\\\text{ for any $u\in (-\infty;+\infty)$, }|\{x:\, \partial_uj(x,u)>0\}|>0.\end{cases}\end{equation}

We establish the following:
 \begin{theorem}[Second formulation of the regularity result]\label{Th:Main}Assume \eqref{Eq:AssQ}--\eqref{Eq:AssQ2}--\eqref{Eq:AssJ} hold. Then:
\begin{enumerate}
\item \underline{Regularity for the constrained problem:} any solution $m^*$ to \eqref{Eq:PvNonEnergetic} is a characteristic function, meaning that $m^*=\mathds 1_{E^*}$. Furthermore, the measure theoretic boundary $\partial^e E^*$ of $E^*$ (see \eqref{De:Boundary})
is a finite union of $\mathscr C^{2,\gamma}$ curves if $d=2$.\item\underline{Regularity for the penalised problem:} let $c>0$. Any solution $m^*$ to \eqref{Eq:PvNonEnergetic2} is a characteristic function, meaning that $m^*=\mathds 1_{E^*}$. Furthermore, defining $\partial^e E^*$ as in \eqref{De:Boundary}, the following holds:\begin{enumerate}
\item If $d=2$, $\partial^e E^*$ is a finite union of $\mathscr C^{2,\gamma}$ curves.
\item In dimension $d\geq 3$, $\partial^e E^*$ is $\mathscr C^{2,\gamma}$ up to a $(d-2)$-dimensional set.
\end{enumerate}
\end{enumerate}
\end{theorem}
\begin{remark}
    Notice that one of the unexpected facts about Theorem \ref{Th:Main1} is the general bang-bang property, namely, that any solution is a characteristic function. We refer to Section \ref{Se:Bibliography} for a discussion of this aspect, as well as to Section \ref{Ap:Necessity}.
\end{remark}
One might wonder whether we could handle Hamilton-Jacobi equations with non-quadratic hamiltonians. We address this question in Section \ref{Re:HJB}.

\subsubsection{Structure of the rest of the introduction}
In the remainder of this introduction, we  describe the link between \eqref{Eq:PvNonEnergetic}--\eqref{Eq:PvNonEnergetic2} and unstable free boundary problems, see Section \ref{Se:Link}, while in Section \ref{Sec:MainIdeasProof} we give the general idea of the proof of the regularity of the free boundary. In Section \ref{Se:Plan} we give the roadmap to the paper. We then present the main references for the study of such problems, see Section \ref{Se:Bibliography}. Finally, we make a series of comments (some of which are expanded on in the appendices) and list some open problems in Section \ref{Se:Open}.

\subsection{From optimal control to unstable free boundary problems}\label{Se:Link}
In this section we describe how \eqref{Eq:PvNonEnergetic}--\eqref{Eq:PvNonEnergetic2} tie to unstable free boundary problems, and what the main connections and differences with the existing free boundary literature are. We focus on \eqref{Eq:PvNonEnergetic}.  Part \ref{Se:Basics} actually consists in boiling down all bilinear optimal control problems to the free boundary problem setting. Specifically, we prove the following things: 
\begin{enumerate}
\item First, any solution to \eqref{Eq:PvNonEnergetic} is a characteristic function: $m^*=\mathds 1_{E^*}$.
\item Second, $E^*$ can be characterised as follows: 
\[ E^*=\{\eta_{m^*}>c^*\}\text{ or }\{\eta_{m^*}\geq c^*\}\] for some constant $c^*$, where the switching function $\eta_{m^*}$ is the (unique) solution to 
\[-\Delta \eta_{m^*}+2\n\cdot( \eta_{m^*}\n\theta_{m^*})-\partial_\theta Q(x,\theta_{m^*})=\partial_\theta j(x,\theta_{m^*}).\] 
\item Third, we can rewrite the equation on $ \eta_{m^*}$ as a free boundary problem of the type 
\begin{equation}\label{Eq:CMPG}-\Delta  \eta_{m^*}=f\mathds 1_{E^*}-g\mathds 1_{\T \setminus E^*}, \qquad f+g > 0,
\end{equation}

alongside some optimality conditions on the shape $E^*$ (see \eqref{Eq:QBT1}), and where $f$ and $g$ are smooth enough functions. This opens the way to blow-up analysis.
\end{enumerate}

From Section \ref{Se:BU} onward, our aim is to prove the regularity of $\eta_{m^{\ast}}$ and of the free boundary $\partial^e E^*$; to do so, we prove that the fact that $E^*$ solves an optimal control problem endows it with particular shape optimality conditions (see \eqref{Eq:QBT1} below), and we reduce our analysis to the regularity of weak solutions of \eqref{Eq:CMPG}, under such shape optimality conditions.

First, let us describe the main contributions and links with the existing literature. Problem \eqref{Eq:CMPG} has been studied by several authors in the past twenty years \cite{MR2606634,zbMATH06021980, Andersson2006CrossshapedAD,  zbMATH06179931, zbMATH05311233, zbMATH05505659, zbMATH05129536,zbMATH05204068,zbMATH06912136, zbMATH07826622} and falls within the theory of unstable obstacle problems. Up to replacing $\eta_{m^*}$ with $-\eta_{m^*}$,  we can assume that (locally at least) $f > 0$, so that the different contributions can be classified in three different regimes, according to the sign of $g$.
\begin{enumerate}
    \item\label{item:IntroLiteraturePt1} $g \equiv 0$. This is the \emph{unstable obstacle problem} \cite{  MR2606634, zbMATH06021980,Andersson2006CrossshapedAD, zbMATH06179931,zbMATH05129536}. The regularity was pioneered by Monneau \& Weiss \cite{zbMATH05129536}, where the authors proved that, in dimension $d=2$, the variational solutions of
    \begin{equation*}
        \min_{u - u_0  \in W_0^{1,2}(\Omega)} \int_\Omega |\n u|^2-\int_\Omega u_+^2
    \end{equation*}
    are $\mathscr C^{1, 1}$, and the free boundary $\partial \{ u>0 \} \cap \Omega$ is locally analytic. In particular, by dimension reduction, in any dimension $d \ge 2$ the free boundary of a minimiser is smooth up to a Hausdorff $(d-2)-$dimensional set. Concerning non-minimising solutions, on the other hand, in a series of papers \cite{ MR2606634, zbMATH06021980, Andersson2006CrossshapedAD, zbMATH06179931}, Andersson, Shahgholian \& Weiss exhibited examples which fail to be $\mathscr C^{1, 1}$ and displaying singular cross-shaped free boundaries, even in dimension $d = 2$. They proceed to prove that, in any dimension $d \ge 2$, the set of points where non-minimising solutions fail to be $\mathscr C^{1, 1}$ has Hausdorff dimension at most $(d-2)$, together with finiteness of the measure and rectifiability properties. To the best of our knowledge, not much is known about the regularity of the free boundary for non-minimising solutions.
    \item\label{item:IntroLiteraturePt2} $g < 0$. This is the \emph{composite membrane problem} \cite{zbMATH01588525, zbMATH05311233, zbMATH05505659, zbMATH05204068}. It arises in the context of the optimisation problem
    \begin{multline}\label{Eq:CMP2}
    \hspace{5.5cm} \min_{m\in \mathcal M}\lambda(m)\\\text{ where $\lambda(m)$ is the lowest eigenvalue of }\begin{cases}
    -\Delta \phi_m=\lambda(m)\phi_m+m\phi_m&\text{ in }\Omega,\,,\\ \phi_m=0&\text{ on }\partial \Omega,
    \end{cases}
    \end{multline}
    and can be reformulated as the volume-constrained unstable free boundary problem 
    \begin{equation}\label{Eq:VC} \min_{u\in W^{1,2}_0(\Omega),\, \int_\Omega u^2=1,\, |\{u>0\}|\leq m_0}\int_\Omega |\n u|^2-\int_\Omega u_+^2.
    \end{equation}
    The first general regularity result for \eqref{Eq:CMP2} and \eqref{Eq:VC} is contained in the work by Chanillo \& Kenig \cite{zbMATH05311233} where, relying on the results from \cite{zbMATH05129536, zbMATH05204068}, the authors prove that any critical point (\emph{i.e.} just a weak solution of \eqref{Eq:CMPG}) is $\mathscr C^{1, 1}$ regular, up to a Hausdorff $(d-2)-$dimensional set, and that the free boundary $\partial\{ u>0 \}$  is locally analytic, up to a Hausdorff $(d-1)-$dimensional set. Furthermore, for variational solutions, the work by Chanillo, Kenig \& To \cite{zbMATH05505659} improves the previous results in dimension $d = 2$, proving that the $\mathscr C^{1, 1}$ regularity and the local analyticity of the free boundary actually hold everywhere in $\Omega$. 
    \item\label{item:IntroLiteraturePt3} $g > 0$. This is the \emph{unstable two-phase membrane problem} \cite{zbMATH06912136, zbMATH07826622}. In this case, the first general regularity result was  proved by Soave \& Terracini \cite{zbMATH06912136}, where the authors show that any weak solution to \eqref{Eq:CMPG}  is $\mathscr C^{1, 1}$ and the free boundary $\partial\{ u > 0\}$ is analytic, both properties holding true outside of a Hausdorff $(d-2)-$dimensional set. These results were extended to the context of Alt-Phillips problems by Soave \& Tortone in \cite{zbMATH07826622}.
\end{enumerate}
The \emph{main differences} in our analysis with respect to the above results is that \emph{we cannot guarantee \emph{a priori} neither that $\eta_{m^{\ast}}$ is a local minimizer of}
\begin{equation*}
    \eta\mapsto \frac12\int_\T |\n \eta|^2-\int_\T (f\eta_++g\eta_-)
\end{equation*}
\emph{nor that $g$ has a fixed constant sign}, which creates several difficulties in term of non-degeneracy of the solutions and regularity of the free boundary (as it turns out, the optimality conditions stemming from optimal control allow to bypass such strong sign or minimality assumptions). Let us discuss this in connection with the literature.  

Concerning regime \eqref{item:IntroLiteraturePt3}, in \cite{zbMATH06912136, zbMATH07826622} the non-positivity of $g$ is fundamental to prove the (almost-)monotonicity of a family of Weiss monotonicity formulas, which in turn implies the non-degeneracy of the solutions at some finite order. The strict negativity of $g$, on the other hand, is required to classify the orders of vanishing of the solutions, allowing for blow-up analysis and dimensional estimates on the singularities of the free boundary. Unfortunately, the monotonicity properties and the analysis of the free boundary break down if $g$ changes sign. Moreover, it is not possible to obtain a satisfactory regularity result in the generality that we require, that is independent of the particular regime \eqref{item:IntroLiteraturePt1}, \eqref{item:IntroLiteraturePt2} or \eqref{item:IntroLiteraturePt3}, using only the PDE \eqref{Eq:CMPG}. Indeed, in regime \eqref{item:IntroLiteraturePt2}, the function $-\frac{1}{2} x_2^2$ is a solution in dimension $d = 2$ with a whole line of singular points (which would prevent dimension reduction arguments), while in regime \eqref{item:IntroLiteraturePt1} we recall the existence of non $\mathscr C^{1,1}$, non-minimising solutions  with singular free boundaries even in dimension $d = 2$. Thus, compensating \eqref{Eq:CMPG} with some second order optimality conditions (as we do) is natural and not restrictive.

Concerning regime \eqref{item:IntroLiteraturePt2}, in \cite{ zbMATH05311233, zbMATH05505659,zbMATH05204068} the strict positivity of $g$ is crucial throughout the analysis as it guarantees strict superharmonicity of the solution, which in turn implies non-degeneracy at order two at singular free boundary points (namely where $\{ u = \vert \nabla u \vert = 0 \}$), and thus allows for blow-up analysis and dimension reduction arguments.

Regarding regime \eqref{item:IntroLiteraturePt1}, the variational formulation, and in particular minimality, is a starting point of the analysis and is used to prove non-degeneracy of order two for the solutions. This, once again, allows for blow-up analysis and to study the regularity of the free boundary.

To conclude, from the free boundary regularity perspective, the main difficulty of our work relies in \emph{devising a strategy which is valid independently of the regime \ref{item:IntroLiteraturePt1}, \ref{item:IntroLiteraturePt2} or \ref{item:IntroLiteraturePt3}, and allows to prove the regularity of the free boundary without a priori non-degeneracy}. To this aim, we introduce new ideas based on optimal control and Fourier analysis to establish non-degeneracy at points of intermediate density, and use these points together with arguments from measure theory to reconstruct the entire free boundary in two dimensions (we refer to Section \ref{Sec:MainIdeasProof} for a more detailed discussion). Incidentally, this is also where the difference between the constrained problem \eqref{Eq:PvNonEnergetic} and the penalised one \eqref{Eq:PvNonEnergetic2} comes into light: for \eqref{Eq:PvNonEnergetic2}, the local optimality conditions are much stronger, and allow to derive non-degeneracy of the blow-ups at \emph{any} critical free boundary point, once again relying on optimal control tools.

\subsection{Main ideas of the proof: the regularity of the free boundary}\label{Sec:MainIdeasProof}
In the following discussion, we consider $m^*=\mathds 1_{E^*}$ (justified by Theorem \ref{Th:Main} (1)) and we briefly discuss the main ideas for the regularity of $\partial^e E^*$ (see also \eqref{De:Boundary}). For the sake of readability, we focus on the constrained case \eqref{Eq:PvNonEnergetic}--the strategy for the penalised case \eqref{Eq:PvNonEnergetic2} is similar-- and, more specifically, on point (1) of Theorem \ref{Th:Main}. 

Let  $E^{\ast} \subset \T$ be an optimal set, and consider the full free boundary is $\partial E^{\ast}$. We say that a point $x_0 \in \partial E^{\ast}$ is regular if $\nabla \eta_{m^{\ast}}(x_0) \neq 0$ and we write $x_0 \in \mathcal{R}$. The implicit function theorem guarantees the $\mathscr C^{2,\gamma}$ regularity of the free boundary near $x_0$. Then, we say that a point $x_0 \in \partial E^{\ast}$ is singular if $x_0 \in \mathcal S:= \partial E^{\ast} \setminus \mathcal{R}$, leading to a first decomposition of the free boundary as \begin{equation*}
    \partial E^{\ast} = \mathcal{R} \sqcup \mathcal{S}.
\end{equation*}
The main difficulty in the study of the regularity of the free boundary for \eqref{Eq:PvNonEnergetic} relies in the several possible scenarios that might occur for points belonging to $\mathcal S$, and on the fact that, as noted in the previous paragraph, we can not rely on energy minimisation properties (this is a connection with the recent work \cite{zbMATH07971611}). To be more explicit, let us describe the possible behaviours at singular points: we decompose $\mathcal{S}$ into a measure theoretic part $\mathcal{S}^e \coloneqq \mathcal{S} \cap \partial^e E^{\ast}$ (where $\partial^e E^{\ast}$ is defined in \eqref{De:Boundary}) and a degenerate part $\mathcal{S}^d \coloneqq \mathcal{S} \setminus \mathcal{S}^e$, so that
\begin{equation*}
    \partial E^{\ast} = \mathcal{R} \sqcup \mathcal{S}^e \sqcup \mathcal{S}^d.
\end{equation*}
The part $\mathcal{S}^d$ is the most singular portion of $\partial E^{\ast}$ as it is singular both analytically  (for the function) and geometrically (for the set). In general, $\mathcal{S}^d$ need not be empty \cite{zbMATH01631054}, although its presence is physically and biologically less important as it does not describe a transition from optimal to non-optimal regions. For these reasons, the study of $\mathcal{S}^d$ will not be addressed in the present paper. Let us focus on the remaining singularities $\mathcal{S}^e$. Given a point $x_0 \in \mathcal{S}^e$, we can classify $x_0$ according to three mutually exclusive scenarios described in detail in Theorem \ref{Th:MW1} but that can be summarised as follows:
\begin{enumerate}
    \item The decay of $\eta_{m^{\ast}}$ at $x_0$ is less than quadratic (or, in terms of the Weiss functional introduced in Section \ref{Se:Weiss}, $\Psi(0^+)=-\infty$). We call these harmonic singular points, as the renormalised quadratic blow-ups of $\eta_{m^{\ast}}$ at $x_0$ are nondegenerate 2-homogeneous harmonic functions (see Theorem \ref{Th:MW1}). We write $x_0 \in \mathcal{S}_{h}^e$.

    \item The decay of $\eta_{m^{\ast}}$ at $x_0$ is exactly quadratic (in terms of the Weiss functional $\Psi(0^+)>-\infty$, see Theorem \ref{Th:MW1}). We call these nondegenerate singular points, and the quadratic blow-ups of $\eta_{m^{\ast}}$ at $x_0$ are nondegenerate 2-homogeneous functions solving a limit equation involving $E^{\ast}$. We write $x_0 \in \mathcal{S}_{nd}^e$.

     \item The decay of $\eta_{m^{\ast}}$ at $x_0$ is more than quadratic (in terms of the Weiss functional $\Psi(0^+) = 0$). We call these degenerate singular points, and the quadratic blow-ups of $\eta_{m^{\ast}}$ at $x_0$ are identically zero. We write $x_0 \in \mathcal{S}_{d}^e$.
\end{enumerate}
The final decomposition of the free boundary reads
\begin{equation*}
    \partial E^{\ast} = \mathcal{R} \sqcup \mathcal{S}_h^e \sqcup \mathcal{S}_{nd}^e \sqcup \mathcal{S}_d^e \sqcup \mathcal{S}^d.
\end{equation*}
Restricting the analysis to the two dimensional case, in Parts \ref{Se:BU} and \ref{Pa:Constrained} we show that $\mathcal{S}_h^e, \mathcal{S}_{nd}^e, \mathcal{S}_d^e = \emptyset$, so that $\mathcal{S}^e = \emptyset$. Let us shortly describe how we deal with each term.
\begin{itemize}
    \item $\mathcal{S}_h^e = \emptyset$. This is a consequence of the strong second order optimality conditions (see Theorem \ref{Th:ASCConstrained}). The result is proved in Section \ref{Se:2dReg}, and follows the lines of \cite{zbMATH05505659,zbMATH05129536}. Actually, it is likely that for \eqref{Eq:PvNonEnergetic}, $\mathcal{H}^{d-2}\left( \mathcal{S}_h^e \right)<+\infty$ in any dimension $d\ge2$, together with $(d-2)-C^1$-rectifiability (see \cite{zbMATH06021980, zbMATH06179931}).

    \item $\mathcal{S}_{nd}^e = \mathcal{S}_{nd, p}^e \sqcup \mathcal{S}_{nd, s}^e $ and $\mathcal{S}_{nd, s}^e =\emptyset$. We can further decompose the set of nondegenerate singular points $\mathcal{S}_{nd}^e$ into two mutually disjoint subsets $\mathcal{S}_{nd, p}^e$ and $\mathcal{S}_{nd, s}^e$, with the following properties. 
    \begin{itemize}
        \item  The set $\mathcal{S}_{nd, p}^e$ consists of points $x_0 \in \mathcal{S}_{nd, p}^e$ at which all blow-ups have constant sign and, consequently, are either one-dimensional parabolas or one-dimensional half-parabolas. We cannot exclude such blow-ups \emph{a priori}, as we cannot guarantee sub/super-harmonicity properties for the blow-ups, and as we can not show that blow-ups satisfy blown-up versions of local optimality conditions for points of $\mathcal{S}_{nd}^e$. The analysis of $\mathcal{S}_{nd, p}^e$ is deferred to the next point.

        \item The set $\mathcal{S}_{nd, s}^e \coloneqq \mathcal{S}_{nd}^e \setminus \mathcal{S}_{nd, p}^e$ consists of points $x_0 \in \mathcal{S}_{nd}^e$ which admit at least one sign-changing blow-up. We notice that the structure of these sign-changing blow-ups is quite peculiar and nontrivial even in dimension $d = 2$ (see Proposition \ref{Pr:Classification}). The zero set is composed of star-shaped lines connected by a unique singular point. These lines separate connected components with constant sign whose angular opening is alternating between two angles depending on $f(x_0)$ and $g(x_0)$ (see Figure \ref{Fig:BU}), which can vary in a finite family (depending on the point). As a consequence of this classification and of the strong second order optimality conditions from Theorem \ref{Th:ASCConstrained}, we have $\mathcal{S}_{nd, s}^e = \emptyset$ (see Proposition \ref{Pr:RegularityIntermediateDensity}).
    \end{itemize}

    \item $\mathcal{S}_{d}^e \cup \mathcal{S}_{nd, p}^e = \emptyset$ in dimension $d=2$. This is the core of our analysis and requires several new ideas. The main difficulty (and also the main difference with respect to the literature \cite{zbMATH05311233, zbMATH05505659,zbMATH05129536, zbMATH06912136}) is that, in order to exclude the degenerate singular points, we can only rely on the second order optimality conditions. To this aim, we first consider the notion of intermediate density point for $E^{\ast}$, see \eqref{Eq:NDAssumption}. Then, we derive a weak version of the stability condition (Proposition \ref{Pr:LocalMinimalityShape}), which is flexible enough to pass to the blow-up limit (unlike the strong version from Theorem \ref{Th:ASCConstrained}), and with which we are able to prove the nondegeneracy of blow-up limits at intermediate density points by means of Fourier analysis (Theorem \ref{Th:NonDegeneracyConstrained}). In particular, since we have already excluded $\mathcal{S}_h^e$ and $\mathcal{S}_{nd, s}^e$ and the blow-ups are sign-changing (see Proposition \ref{Pr:Classification}), the intermediate density points are regular. More precisely, the free boundary passing through them is a simple and closed curve, with a positive contribution bounded from below independently of the particular point or curve to the strong stability condition from Theorem \ref{Th:ASCConstrained}, see Lemma \ref{Le:NonDegenerateConnectedComponentsAreSmooth} (for this last property the arguments are similar to \cite{zbMATH05505659}). As a consequence, only a finite number of such curves can occur, see Proposition \ref{Pr:NonAccumulationToSingularPoints}. Now we can use the nondegeneracy at intermediate density points to deduce some information on $\mathcal{S}_{d}^e$ and $\mathcal{S}_{nd, p}^e$. More precisely, by means of an argument from measure theory (Proposition \ref{Pr:DensityIntermediateDensity}), the intermediate density points are dense in the geometric free boundary $\partial^e E^{\ast}$ (here is crucial to consider the geometric  free boundary and not the whole free boundary, thus discarding $\mathcal{S}^d$). Thus, if $x_0 \in \mathcal{S}_{d}^e \cup \mathcal{S}_{nd, p}^e$ we could find a sequence of intermediate density points $ x_{0, i} $ converging to $x_0$ and in turn, generate a sequence of smooth and disjoint free boundary curves $\Gamma_i \subset \partial^e E^{\ast}$. Since we can only have finitely many of them, we have reached a contradiction and $\mathcal{S}_{d}^e \cup \mathcal{S}_{nd, p}^e = \emptyset$.  
\end{itemize}

\subsection{Plan of the article}\label{Se:Plan}
In Part \ref{Se:Basics}, we study in detail how bilinear optimal control problems can be recast as unstable free boundary problems, with a specific emphasis on local optimality conditions (Section \ref{Se:UFBP}). We perform a blow-up analysis in Section \ref{Se:BU} to obtain some first regularity results and, in Section \ref{Se:NonDegeneracy}, we give some (sharp) non-degeneracy estimates. In Part \ref{Pa:Constrained}, we give the complete regularity of the free boundary in the two-dimensional case as well as the regularity in higher dimensions for the penalised problem.
\subsection{Notational conventions}
\begin{itemize}
\item We write $f\lesssim g$ to mean ``there exists a constant $C$ that does not depend on $f,\, g$ or any of the remaining parameters, such that $f\leq C g$."
%\item The Sobolev spaces (with both positive and negative indices) are denoted (and defined) in the standard way $W^{s,p}(\T)$, for $s\in \R$, $p\in [1;+\infty]$.
\item When $\Sigma$ is a $k$-dimensional object, $d\mathscr H^k\llcorner\Sigma$ denotes the restriction of the $k$-dimensional Hausdorff measure to $\Sigma$.
\item When $\Sigma$ is a smooth oriented hypersurface and $f$ is a function that has a discontinuity across $\Sigma$, $\llbracket f\rrbracket_\Sigma$ (or $\llbracket f\rrbracket$ when no ambiguity is possible) denotes the jump of $f$ across $\Sigma$.
\item The notation $\fint_\O$ denotes the average integral:
\[ \fint_\O f=\frac1{|\O|}\int_\O f.\]
\end{itemize}

 \section{References, comments and open problems}
 \subsection{Bibliographical references and state of the art}\label{Se:Bibliography}
 
\subsubsection{A biological motivation for our study: understanding the fragmentation phenomenon}
Our main motivation to study regularity of optimal controls comes from mathematical biology and the complex geometric features one encounters there. We begin by recalling the most basic optimal control problem in mathematical biology (the optimal survival ability) that has a reasonable geometric behaviour, and give an example of a seemingly simple optimisation problem that has a ``wild" behaviour (the total population size).

 \textbf{The optimal survival ability}
In mathematical biology, the control $m$ should be understood as a resources distribution available to a population, and the basic problem in optimal control for population dynamics is the following: how should we spread resources in order to make it as quick as possible for a very small population to grow \cite{zbMATH02194918,MR1105497,CantrellCosner1,zbMATH05530397,zbMATH06690453}? As it turns out, this so-called ``optimal survival ability problem" amounts to solving the eigenvalue minimisation problem\begin{equation}\label{Eq:CMP} \min_{m\in \mathcal M}\lambda(m)\text{ where }\lambda(m)=\min_{u\in W^{1,2}_0(\O)\,, \int_\O u^2=1}\int_\O |\n u|^2-\int_\O mu^2\end{equation} where $\mathcal M$ is defined in \eqref{Eq:Admissible}.
This problem also appears as the so-called ``composite membrane problem" \cite{zbMATH01631054,zbMATH05505659,Cox1990,zbMATH05204068}.
The following properties for \eqref{Eq:CMP} are by now standard \cite{zbMATH01631054}:
\begin{enumerate}
\item Any solution $m^*$ is a characteristic function: $m^*=\mathds 1_{E^*}$.
\item Furthermore, $E^*$ can be characterised through the eigenfunction $u_m^*$ \emph{i.e.} the solution of
\[-\Delta u_{m^*}=\lambda(m^*)u_{m^*}+m^*u_{m^*}\] as 
\begin{equation}\label{Eq:IntroOC1} E^*=\{u_{m^*}>c\}\end{equation} for some $c$.
\item Finally, because of intimate ties to the isoperimetric inequality, several geometric properties are obtainable:
when $\O$ is a ball, the Schwarz rearrangement shows that the unique optimiser $m^*$ of \eqref{Eq:CMP} is the characteristic function of a centred ball. More generally, and even if various symmetry breaking phenomena might occur, $E^*$ ``often" inherits good properties from the domain $\O$ \emph{e.g.} when $\O$ is convex and has enough symmetries, any optimal set $E^*$  is convex. 
\end{enumerate}
So this problem is, in simple geometries, ``well-behaved". Naturally, there are still various open and difficult questions about \eqref{Eq:CMP}; see \cite{zbMATH06690453,zbMATH07573593} for a survey of some problems and \cite{zbMATH08132443,zbMATH07577014,zbMATH07890709,zbMATH07733788} for recent progress. The main regularity question for \eqref{Eq:CMP} is the following: for a general domain $\O$, and for an optimal $E^*$, how regular is $\partial E^*$? This seemingly simple question was solved in $\R^2$ in the $|E|\ll 1$ regime  by Chanillo, Kenig \& To \cite{zbMATH05505659}: they proved  that if $m^*=\mathds 1_{E^*}$  is such that $\lambda(m^*)\geq 0$ and solves \eqref{Eq:CMP}, then  $\partial E^*$  (in fact, $\partial\{u_{m^*}>c\}$) is analytic.  It should be noted that, because of the optimality conditions \eqref{Eq:IntroOC1}, the problem 
\[ \min_{m\in \mathcal M}\lambda(m)\] 
is similar to the inverse (or unstable) obstacle problem and, because of the Rayleigh quotient formulation of $\lambda(m)$, it can be recast as a minimisation, but only in $u$; indeed, one can see that for this simple problem, it is equivalent to study
\[ \min_{u\in W^{1,2}_0(\Omega),\int_\O u^2=1}\int_\O |\n u|^2-\int_\O \mathds 1_{E_u}u^2\]where $\{u> c_u\}\subset E_u\subset \{E_u\geq c_u\}$ and $|E_u|=m_0$.

 \textbf{The optimisation of the total population size} Our initial motivation was the optimisation of the total population size in logistic-diffusive models, which displays many more complex geometric behaviours. We consider, in a domain $\O$ and for a given  $\mu>0$ (quantifying the speed of dispersal of a population) and $m\in \mathcal M$ (modelling the resources distribution inside a domain) the solution $\Theta_{m,\mu}$ to 
 \begin{equation}\label{Eq:LDEIntro}
 \begin{cases}
 -\mu\Delta\Theta_{m,\mu}=\Theta_{m,\mu}\left(m-\Theta_{m,\mu}\right)&\text{ in }\O,\, 
 \\ \partial_\nu\Theta_{m,\mu}=0&\text{ on }\partial \O,\, 
 \\ \Theta_{m,\mu}>0.
 \end{cases}
 \end{equation} 
Existence and uniqueness of a solution to \eqref{Eq:LDEIntro} are guaranteed by  $\int_\O m=m_0>0$
   \cite{zbMATH02194918,MR1105497,CantrellCosner1}.  Lou \cite{zbMATH05023210} proposed the study of the optimisation problem 
 \begin{equation}\label{Eq:SizeIntro}
 \max_{m\in \mathcal M}\int_\O \Theta_{m,\mu},\end{equation} with two salient questions:
 \begin{enumerate}
 \item The first one was the question of the saturation of constraints; in other words, considering a solution $m^*$ of \eqref{Eq:SizeIntro}, is it true that $m^*=\mathds 1_{E^*}$ for some $E^*$?
 \item The second one is more geometric in nature: if the answer to the previous question is positive, what does the optimal set $E^*$ look like? 
 \end{enumerate}
 Regarding the first question, after several partial results  \cite{zbMATH07155098,NagaharaYanagida}, it was proved by Mazari, Nadin \& Privat \cite{zbMATH07523453} that any solution $m^*$ writes $m^*=\mathds 1_{E^*}$ for some $E^*$. Regarding the geometry of $E^*$ it was observed in \cite{zbMATH07155098} that, as $\mu\to \infty$, the optimal set $E^*$ behaves geometrically like the optimal set for the composite membrane problem \eqref{Eq:CMP}; it was later shown by Mazari \& Ruiz-Balet \cite{zbMATH07310945} (and extended by Heo \& Kim \cite{zbMATH07411877}) that, as $\mu\to 0^+$, there holds 
 \begin{equation}\label{Eq:IntroFragmentation}
\mathrm{Per}(E_\mu^*)\underset{\mu\to 0^+}\rightarrow +\infty
 \end{equation} where, for any $\mu>0$, $\mathds 1_{E_\mu^*}$ solves \eqref{Eq:SizeIntro}. In the one-dimensional case, this means that the number of connected components is unbounded as $\mu\to 0^+$, so that this phenomenon is dubbed \emph{fragmentation}. In higher dimensions, this indicates that the optimal set $E_\mu^*$ either fragments, or develops oscillations on its boundary. \emph{Understanding precisely if, for a fixed diffusivity $\mu$, the optimal set can develop singularities, was our initial motivation for undertaking a general study of bilinear optimal control problems.} In particular, as a byproduct of our analysis, in dimensions 1 and 2, the optimal set has finite perimeter for any $\mu>0$.

 \begin{remark}[The single large \emph{vs} several small debate]
As a side note, we want to remark that such fragmentation phenomena (in the sense of \eqref{Eq:IntroFragmentation}) are not a purely mathematical artefact, and that they resonate with long concerns in conservation ecology; this field emerged in the 70's, in the wake of the works of Diamond \cite{Diamond1975}, and is concerned with the question: is it better to give a population access to small, scattered resources or to concentrated resources? We hope that the present work contributes, from the mathematical perspective, to this endeavour.
 \end{remark}
 
 Note that while \eqref{Eq:CMP} is concerned with the optimisation of the natural energy associated with the PDE constraint (here, the Rayleigh quotient), it is not so for the total population size problem, and this creates several very delicate difficulties--in particular, one can not expect the functional to be convex in general.

 Finally, let us mention that Nadin  proved \cite{arXiv:2112.02876}, for the optimisation of the total population size, that, in the one-dimensional case, optimisers $m^*=\mathds 1_{E^*}$ are $BV$ functions, meaning that $E^*$ has finite perimeter (which is the optimal regularity); his approach relies on the Pontryagin maximum principle and can not be extended to the higher dimensional case.

 \subsubsection{Stable free boundary problems in optimal control theory}
The basic free boundary problem in optimal control is of stable type, by which we mean the following: study the regularity of $\partial\{u>0\}$, where $u$ minimises
\[ \mathcal E_{\mathrm{stable}}:v\mapsto \frac12\int_\O |\n v|^2+\int_\O v_+\] subject to some boundary conditions in a smooth domain $\O$. This problem, called the obstacle problem,
% also dubbed the Alt-Phillips problem since \cite{zbMATH03964583,zbMATH03867795}, 
is a reference case that is now fairly well understood and, in particular, non-degeneracy estimates as well as $\mathscr C^{1,1}$ regularity of $u$ are now standard facts.  The theory was pioneered by Kinderlehrer \& Nirenberg \cite{zbMATH03547784}, where the first results addressed the regularity of $u$. The major breakthrough regarding the regularity of the free boundary is the work of Caffarelli \cite{zbMATH03600481}.  We refer to the monograph \cite{zbMATH06062397} for a comprehensive review. Although this problem finds its roots in economics and physics, it potentially has a number of applications in optimal control problems. To be more specific (and since we could not locate a discussion of this aspect in the existing literature) let us discuss the recent contributions of Buttazzo, Casado-D\'iaz \& Maestre \cite{zbMATH07930852,arXiv:2601.01591}. They study bilinear optimal control problem of the form 
\begin{equation}\label{Eq:BCDM}\max_{m\in \mathcal M} \int_{\T} j(\theta_m)\text{ subject to }\begin{cases}-\Delta \theta_m+m\theta_m=f&\text{ in }\O,\, \\ \theta_m=0&\text{ on }\partial \O.\end{cases}\end{equation} From the point of view of regularity, \cite{zbMATH07930852} establishes that, if $m^*$ is an optimal control, then $m^*$ is a $BV$ function. However, from a mathematical perspective, the optimality system associated with \eqref{Eq:BCDM} writes
\[-\Delta z_{m^*}+\phi(z_{m^*})=j'(\theta_{m^*}),\] with $\phi(\cdot)$ a non-decreasing function, thereby bearing strong resemblance to minimisers of $\mathcal E_{\mathrm{stable}}$. This allows to show that $\phi(z_{m^*})$ is $BV$, which in turn entails that $m^*$ is $BV$. There are two crucial differences with our setting:
\begin{enumerate}
\item First, in general, solutions of \eqref{Eq:BCDM} are not characteristic functions of sets (and one can construct simple examples where the optimiser $m^*$ of \eqref{Eq:BCDM} is the constant control). 
We do however emphasise that if one can prove \emph{a priori} that $m^*=\mathds 1_{E^*}$ then it is likely that one can use the standard techniques of \cite{zbMATH06062397} to prove $\mathscr C^{2,\gamma}$ regularity of $\partial E^*$. As this is not the primary focus of the paper we leave this as a question for future research.
\item Second, as already observed in the case of the composite membrane problem (see Eq. \eqref{Eq:IntroOC1}), the optimality system associated with the class of problems we study  write
\[-\Delta u_{m^*}=\phi(u_{m^*})+\lambda(m^*)u_{m^*}\] where $\phi$ is non-decreasing, see \eqref{Eq:CMPG}; this reverts the monotonicity of the optimality system and prohibits using the tools of \cite{zbMATH07930852}. From an energetic perspective, this is very simply explained by considering the energy associated with \eqref{Eq:CMPG} (in the case $f=1,\, g=0$ for simplicity):
\[\mathcal E_{\mathrm{unstable}}:v\mapsto \frac12\int_\O |\n v|^2-\int_\O v_+.\]
 On the one-hand, $\mathcal E_{\mathrm{stable}}$ is a convex functional, while $\mathcal E_{\mathrm{unstable}}$ is not, leading to non-uniqueness of minimisers, bifurcations etc. We expand on these aspects in the next paragraph.
\end{enumerate}

\subsubsection{Regularity theory for unstable free boundary problems} Naturally, our work has strong ties to the theory of the unstable obstacle problem, for which we refer to the detailed overview we gave in Section \ref{Se:Link}.

\subsubsection{Bilinear control problems as free boundary problems}
A part of the analysis carried in this paper is related to a recent series of contributions devoted to the understanding of the so-called bang-bang property in optimal control problems--namely, is it true that any optimiser $m^*$ of an $L^\infty-L^1$ constrained optimal control problem writes as $m^*=\mathds 1_{E^*}$? In particular, several results \cite{zbMATH07523453,zbMATH07812267} derived this property under monotonicity and bilinearity assumptions for elliptic and parabolic problems. Nevertheless, the approach we present here, based on the reformulation of the problem as a Hamilton-Jacobi problem, is in our opinion more efficient, and sheds a new light on this type of results. Note however that, as far as we are aware, this is the first time generic bilinear optimal control problems are treated as free boundary problems at this level of generality, and we hope that the present paper contributes to the qualitative and quantitative theory of optimal control problems in applied fields. 

\subsubsection{Some ties to thresholding schemes}
We already mentioned the interest of regularity theory in the understanding of surprising fragmentation phenomena in population dynamics. Another interest comes from the numerical approximation of optimal controls through the use of a  thresholding scheme. To give a concrete example, let us consider the composite membrane problem \eqref{Eq:CMP}. A popular iterative scheme is an adaptation of the work of C\'ea, Gioan \& Michel \cite{Cea}, which consists of a fixed-point method on the optimality conditions of \eqref{Eq:CMP}, see \eqref{Eq:IntroOC1}. Namely, a given initialisation $m_1\in \mathcal M$, the algorithm goes as follows:
 \begin{enumerate}
 \item For $k\geq 1$, compute the $L^2$ normalised, non-negative eigenfunction $u_k$ associated with the lowest eigenvalue $\lambda(m_k)$ of $-\Delta -m_k$.
 \item Compute $c_k$ such that $|\{u_k>c_k\}|=m_0$.
 \item Set $m_{k+1}:=\mathds 1_{\{u_k>c_k\}}.$
 \end{enumerate}
Such schemes can be seen as a control version of the seminal \cite{MBO} that rather deals with numerical approximations of the mean curvature flow and which  has been the topic of much research activity over the past decades \cite{Barles_1995,evans1993convergence,ishii1995generalization,ISHII_1999,laux2016convergence,Laux_2017,ruuth2001diffusion,Swartz_2017}. In the context of optimal control, this method and its generalisations to many other classes of optimal control problems is  popular and efficient \cite{BintzLenhart,Quantum,HintermullerKaoLaurain,KaoLouYanagida,KaoMohammadi,zbMATH07301284,zbMATH06690453,Pironneau,KaoMohammadi2}.
  Nevertheless, a general proof of convergence (or the identification of suitable assumptions guaranteeing convergence) is still missing. Chambolle, Mazari-Fouquer \& Privat initiated \cite{zbMATH08079247} a systematic study of such convergence properties, and emphasised the role of local stability of optimal shapes, which is intimately tied to their regularity. We hope that the results we present here can provide new impetus in this research line.

\subsection{Further comments}\label{Se:Open}

\subsubsection{Open problem: volume constrained unstable free boundary problems}
As alluded to in several places, one of the main difficulties is the lack of a systematic theory for volume constrained optimisation problem of the type \eqref{Eq:VC}, wherein one enforces, in the parlance of free boundary problems, a volume constraint on the phase $\{u>0\}$. In contrast, compare this to the case of the standard obstacle problem
\[\min_{u,\, |\{u>0\}|=V_0}\frac12\int_\O |\n u|^2+\int u_+.\] As the underlying functional is convex, it is standard that, up to choosing the right Lagrange multiplier $c$, this is equivalent to studying  the penalised problem
\[\min_{u}\frac12\int_\O |\n u|^2+\int u_++c|\{u>0\}|.\]

\subsubsection{A comment on the optimal control of Hamilton-Jacobi equations}\label{Re:HJB}
As we set ourselves in the framework of optimal control of Hamilton-Jacobi equations, a natural question is the following:
is  there any hope to tackle, with similar techniques, optimal control problems of the form
\[ \max_{m\in \mathcal M}\int_{\T}j(u_m)\text{ subject to }-\Delta u_m-H(x,\nu_m)=m+Q(x,u_m)?\] Here, we do not detail the assumptions on $Q$, and merely assume that for any $m$, this equation is well-posed in $W^{1,2}(\T)$. We still assume that $j'>0$ (for the sake of simplicity, we drop the dependency in $x$). On the one-hand, if
\[ H(x,p)=V(x)|p|^2,\] where $V\in \mathscr C^1(\T)$, $\min_{\T}V>0$, we claim that,  by adapting the proofs of this paper, the following holds: \begin{enumerate}
\item Any solution $m^*$ is a characteristic function, say $m^*=\mathds 1_{E^*}$, and that $E^*=\{\eta_{m^*}>c_{m^*}\}$ or $E^*=\{\eta_{m^*}\geq c_{m^*}\}$ for some Lagrange multiplier $c_{m^*}$, $\eta_{m^*}$ being the switch function defined as the unique solution of the equation 
\[-\Delta \eta_{m^*}+2\n\cdot(V\eta_{m^*}\n u_{m^*})-\partial_uQ(x,u_{m^*})\eta_{m^*}=j'(u_{m^*}),\]
\item $E^*$ satisfies the conclusions of Theorem \ref{Th:Main} (whether one looks at the constrained or penalised case).\end{enumerate} In fact, this extends to any Hamiltonian $H$ satisfying 
\[ \n^2_{pp}H(x,p)\gtrsim \mathrm{Id}\] in the sense of quadratic forms, uniformly in $x\in \T$. The real difficulty comes from the case of non-quadratic Hamiltonians and, in general, we do not expect that such problems can  be reformulated as free boundary problems: the basic problem is that, if $H$ decays at 0 faster than a quadratic polynomial, we can not guarantee that the solution $m^*$ to the optimal control problem is a bang-bang function. More specifically, in Part \ref{Se:Basics} we show that any second-order critical point for quadratic Hamiltonian is bang-bang, which we can guarantee is not true for non-quadratic Hamiltonian. Indeed, consider some $\alpha>2$, the Hamiltonian
\[ H(x,p):=|p|^\alpha\]
and the optimisation problem
\[ \max_{m\in \mathcal M} \mathscr J(m):=\int_{\T} j(u_m)\text{ subject to }-\Delta u_m+|\n u_m|^\alpha=m+Q(u).\]  It is a simple exercise to check that, for $Q(u)=e^{-u}$, $j(u)=e^{\frac{u}2}$, the constant control $\bar m\equiv m_0$ is a second-order critical point of $\mathscr J$.

\subsubsection{On the structural assumptions needed to ensure the bang-bang property}\label{Ap:Necessity}
Both bilinearity and monotonicity are crucial to obtain our results. Let us briefly explain why.
\begin{enumerate}
\item It is quite simple to construct monotone optimal control problem with linear interaction that fail to satisfy the bang-bang property: take for instance the control problem
\[ \max_{m\in \mathcal M}\int_{\T} y_m\text{ subject to }-\Delta y_m-y_m(1-y_m)=m\text{ in }\T.\] Then, by a simple comparison argument, this control problem is monotone but, by an elementary concavity argument, the unique maximiser is the constant $\bar m\equiv m_0$.
\item Likewise, a simple example of a bilinear control problem that fails the monotonicity requirement and whose maximisers are not bang-bang is provided by Lou \cite{zbMATH05023210} where one aims at minimising the total population size in logistic models. Namely, consider
\[ \max_{m\in\mathcal M}\int_{\T}\left(-\theta_m\right)\text{ subject to }-\Delta\theta_{m}=\theta_{m}\left(m-\theta_{m}\right)\text{ in }\T.\]
By a direct comparison, Lou proved that the unique solution to that problem is the constant $\bar m\equiv m_0$.
\end{enumerate}

\part{From optimal control to free boundary problems}\label{Se:Basics}
\section*{Plan of the part }
The goal of this part is to reformulate \eqref{Eq:PvNonEnergetic}--\eqref{Eq:PvNonEnergetic2} as unstable free boundary problems satisfying certain stability conditions. In Section \ref{Se:UFBP} deals with \eqref{Eq:PvNonEnergetic} and Section \ref{Se:UFBP2} with \eqref{Eq:PvNonEnergetic2}. The proofs are nearly identical for both, and we only give all details for \eqref{Eq:PvNonEnergetic}.

\section{ The case of volume constrained problems}\label{Se:UFBP}
\subsection{Main result}
The main result is the following:
\begin{theorem}\label{Th:ReformConstrained}
Let $m^*$ solve \eqref{Eq:PvNonEnergetic}. Let $\eta_{m^*}$ denote the associated switch function, that is, the unique solution of 
\begin{equation}\label{Eq:C2IMNonEnergetic}
-\Delta \eta_{m^*}+2\n\cdot\left(\eta_{m^*}\n \theta_{m^*}\right)-\partial_\theta Q(x,\Theta_{m^*})\eta_{m^*}=\partial_\theta j(x,\theta_{m^*})\text{ in }\T.
\end{equation}
Then:
\begin{enumerate}
\item \underline{Reformulation as a free boundary problem:} $m^*=\mathds 1_{E^*}$ for some $E^*$ and there exists $c_{m^*}>0$ such that either 
\[ E^*={\{\eta_{m^*}>c_{m^*}\}}\]
or 
\[ E^*={\{\eta_{m^*}\geq c_{m^*}\}}.\] 
 In particular, $(\theta_{m^*},\eta_{m^*})$ solves the free boundary system 
\begin{equation}\label{Eq:FBNE}
\begin{cases}
-\Delta \theta_{m^*}-|\n \theta_{m^*}|^2=\mathds 1_{\{\eta_{m^*}> c_{m^*}\}\text{ or }\{\eta_{m^*}\geq  c_{m^*}\}}+Q(x,\theta_{m^*})&\text{ in }\T,\, 
\\ -\Delta \eta_{m^*}+2\n\cdot\left(\eta_{m^*}\n \theta_{m^*}\right)-\partial_\theta Q(x,\theta_{m^*})\eta_{m^*}=\partial_\theta j(x,\theta_{m^*})&\text{ in }\T,
\end{cases}
\end{equation}Setting $\eta:=\eta_{m^*}-c_{m^*}$,  there exist $f,\, g$ such that
\[\forall p\in [1;+\infty),\, f,\, g\in W^{1,p}(\T),\, \min_\T\left(f+g\right)>0\] and 
\[ -\Delta \eta=(f+g)\mathds 1_{\{\eta>0\}\text{ or }\{\eta\geq 0\}}-g.\]
\item\underline{Second-order optimality conditions I:} with the same notation for $\eta$ there exist constants $\alpha>0$, $r_0>0$ such that the following holds: for any $m\in \mathcal M$ satisfying 
\[ \{m\neq m^*\}\subset \B(x_0;r_0)\] for some $x_0\in \T$, there holds
\begin{equation}\label{Eq:QBT1} 0\geq \int_{\T}\eta(m-m^*)+\alpha \Vert m-m^*\Vert_{W^{-1,2}(\T)}^2.\end{equation}
\item\underline{Second-order optimality conditions II:} there exists $r_0,\,\sigma>0$ such that, for any $\Gamma\subset \partial\{\eta>0\}$ of class $\mathscr C^{2,\gamma}$ (for any $\gamma \in (0;1)$), for any $v\in W^{1,2}(\T)$ supported in $\B(x_0;r_0)$ for some $x_0\in \T$, 
\begin{equation}\label{Eq:SOSD1}\int_\Gamma \frac{v}{|\n \eta|}=0\Rightarrow \sigma\int_\Gamma\frac{v^2}{|\n \eta|}\leq \int_\T |\n v|^2.\end{equation}
\end{enumerate}
\end{theorem}

\begin{remark}
Condition \eqref{Eq:QBT1} is a quantitative Hardy-Littlewood inequality. Conditions \eqref{Eq:QBT1}--\eqref{Eq:SOSD1} are related to local minimality for the associated energy, but do not allow to conclude that $\eta$ is an energy minimiser, see Section \ref{Se:LocalMinimality}.
These optimality conditions serve different purposes: \eqref{Eq:QBT1} will give non-degeneracy at the blown-up scale, while \eqref{Eq:SOSD1} will be used to rule out certain blow-up profiles, in the spirit of \cite[Theorem 8.1]{zbMATH05129536}--\cite[Section 3]{zbMATH05505659}, as it roughly states that
\[ \int_{\{\eta=0\}}\frac1{|\n \eta|}<+\infty.\] We refer to  Section \ref{Se:ASCConstrained}, Theorem \ref{Th:ASCConstrained}.

\end{remark}
 
Points $(1)$ and $(2)$ of Theorem  \ref{Th:ReformConstrained} are proved in Section \ref{Se:BangBang}. Point (3) is established in Section \ref{Se:Shape}.
 \subsection{Reformulation as a free boundary problem}\label{Se:BangBang}

\subsubsection{Preliminary comments and results}
The set $\mathcal M$ is convex; its extreme points are given by 
\[ \mathcal E=\left\{\mathds 1_E:\,\text{ $E$ measurable, } E\subset\T\,, |E|=m_0\right\},\] see  \cite[Proposition 7.2.17]{zbMATH06838450}. Our approach relies on first and second-order optimality conditions, which in turn hinge on differentiability of the solution mappings. For the sake of readability, we gather the relevant information here.

\begin{lemma}\label{Le:Differentiability}
The map $m\mapsto \theta_m$ is twice Gateaux differentiable and letting, for any $h\in L^\infty(\T)$, 
\[ \dot\theta_m:=\underset{t\to 0}\lim\frac{\theta_{m+th}-\theta_m}{t},\] where the limit is $W^{1,2}(\T)$ weak, $L^2(\T)$ strong, $\dot\theta_m$ is the unique solution of 
\begin{equation}\label{Eq:DotTheta}
-\Delta \dot\theta_m-2\langle \n \dot\theta_m,\n \theta_m\rangle=h+\partial_\theta Q(x,\theta_m)\dot\theta_m.
\end{equation}
Finally, letting $\ddot\theta_m$ be the second-order Gateaux derivative of $m\mapsto \theta_m$, $\ddot \theta_m$ is the unique solution of 
\begin{equation}\label{Eq:DdotTheta}
-\Delta \ddot\theta_m-2\langle \n \ddot\theta_m,\n \theta_m\rangle=\partial_\theta Q(x,\theta_m)\ddot\theta_m+\partial_{\theta\theta}^2Q(x,\theta_m)(\dot\theta_m)^2+2|\n\dot\theta_m|^2.
\end{equation}
\end{lemma}
We give the (standard) proof in Appendix \ref{Ap:Differentiability}. From Lemma \ref{Le:Differentiability}, $J$ is Gateaux differentiable at any $m\in \mathcal M$, and its Gateaux derivative $\dot J(m)[h]$ at $m$ in the direction $h$ is given by 
\[ \dot J(m)[h]=\int_{\T} \dot\theta_m\partial_\theta j(x,\theta_m).\] Introduce the solution $\eta_m$ to 
\begin{equation}\label{Eq:C2Switch1}
-\Delta \eta_{m}+2\n\cdot\left(\eta_{m}\n \theta_{m}\right)-\partial_\theta Q(x,\theta_{m})\eta_{m}=\partial_\theta j(x,\theta_{m})\text{ in }\T.\end{equation}

\begin{lemma}\label{Le:SwitchFunction}
Equation \eqref{Eq:C2Switch1} is well-posed. Furthermore,  there exists $\delta>0$ such that 
\begin{equation}\label{Eq:MinSwitch}
\min_{m\in \mathcal M}\min_{x\in \T}\eta_m(x)\geq \delta.
\end{equation}
 \end{lemma}
 
 \begin{remark}
\eqref{Eq:MinSwitch} reflects Assumption \eqref{Eq:AssJ} and the monotonicity of $J$ in the sense that 
 \[\forall m,m'\in \mathcal M,\, \max(J(m),J(m'))\leq J(\max(m,m')).\]
 \end{remark}
 \begin{proof}[Proof of Lemma \ref{Le:SwitchFunction}]
 Regarding the well-posedness of \eqref{Eq:C2Switch1}  observe that the differential operator 
\[ L:=-\Delta +2\n\cdot(\cdot\n \theta_m)-\partial_\theta Q(x,\theta_m)\] is the adjoint of the invertible operator (the invertibility being guaranteed by Assumption \eqref{Eq:AssQ})
\[ L^*=-\Delta-2\langle\, \cdot,\n\theta_m\rangle-\partial_\theta Q(x,\theta_m).\] The Fredholm alternative yields the unique solvability of the equation. Regarding \eqref{Eq:MinSwitch}, by elliptic regularity, it suffices to show that for any $m\in \mathcal M$ we have $\min_{\T}\eta_m>0$.  As  \eqref{Eq:AssQ} and  \eqref{Eq:AssJ} imply $\eta_m\geq 0,\, \not\equiv 0$ in $\T$, the strong maximum principle yields the conclusion.
 \end{proof}

Multiplying \eqref{Eq:DotTheta} by $\eta_m$ and integrating by parts gives
\begin{equation}\label{Eq:DotJNu}
\dot J(m)[h]=\int_{\T} \eta_mh.
\end{equation} Taking, for an optimiser $m^*$ of \eqref{Eq:PvNonEnergetic}, $h=m-m^*$ for some competitor $m\in \mathcal M$, we obtain:
\begin{proposition}\label{Pr:OC1}
If $m^*$ solves \eqref{Eq:PvNonEnergetic}, then, 
\[ \forall m\in \mathcal M,\, \int_{\T}\eta_{m^*} m\leq \int_{\T} \eta_{m^*}m^*.\]
\end{proposition}
If we had, for any $c\in \R,\, |\{\eta_{m^*}=c\}|=0$ this would give  
$m^*=\mathds 1_{\{\eta_{m^*}>c\}}$ for some $c$. However, this is not guaranteed \emph{a priori}.

\subsubsection{Lower bound on the second-order derivative} We denote, for $m\in \mathcal M$ and $h\in L^\infty(\T)$  $\ddot J(m)[h,h]$  the second-order Gateaux derivative of $J$ at $m$ in the direction $h$.
\begin{proposition}\label{Pr:LocalMinimality} Let $m\in \mathcal M$.
There exist a constant $\alpha>0$ and $r_0>0$ such that the following holds: for any $m'\in L^\infty(\T)$ such that $\{m\neq m'\}\subset \B(x_0;r_0)$ for some $x_0\in \T$, 
\begin{equation}\label{Eq:DdotJLB}
\ddot J(m)[m'-m,m'-m]\geq \alpha \Vert m'-m\Vert_{W^{-1,2}(\T)}^2.
\end{equation}
\end{proposition}
To prove it we use the following lemma:
\begin{lemma}\label{Pr:OC2} There exist $\alpha\,, \beta>0$ such that, for any $m\in \mathcal M$ and any $h\in L^\infty(\T)$ there holds 
\begin{equation}\label{Eq:LEDdotJ}
\ddot J(m)[h,h]\geq \alpha \int_{\T}|\n \dot\theta_m|^2-\beta\int_{\T}\dot\theta_m^2,\end{equation} where $\dot \theta_m$ solves \eqref{Eq:DotTheta}.
In particular, there exist $\alpha,\beta>0$ such that, for any $h\in L^\infty(\T)$, there holds
\begin{equation}\label{Eq:DdotJ2}
\ddot J(m)[h,h]\geq \alpha\Vert h\Vert_{W^{-1,2}(\T)}^2-\beta\Vert h\Vert_{W^{-2,2}(\T)}^2.
\end{equation}
\end{lemma}

\begin{proof}[Proof of Lemma \ref{Pr:OC2}]
From Lemma \ref{Le:Differentiability}, $J$ is twice Gateaux differentiable and, for any $m\in \mathcal M$, for any $h\in L^\infty(\T)$, we have
\[ \ddot J(m)[h,h]=\int_{\T}(\dot\theta_m)^2 \partial^2_{\theta\theta}j(x,\theta_m)+\int_{\T} \ddot\theta_m \partial_\theta j(x,\theta_m).\]
Using \eqref{Eq:C2Switch1} and \eqref{Eq:DdotTheta} we deduce that 
\begin{equation}\label{Eq:DdotJ}
\ddot J(m)[h,h]=\int_{\T} \eta_m|\n\dot\theta_m|^2+\int_{\T} W \dot\theta_m^2
\end{equation}
with 
\[W:=\partial_{\theta\theta}^2Q(x,\theta_m)\eta_m+\partial^2_{\theta\theta}j(x,\theta_m)\in L^\infty(\T).\]
From \eqref{Eq:MinSwitch}, \eqref{Eq:AssJ} and the fact that $\theta_m\in L^\infty$, we deduce that there exist $\alpha,\, \beta>0$ such that 
\[ \ddot J(m)[h,h]\geq \alpha \int_\T |\n \dot\theta_m|^2-\beta\int_\T \left(\dot\theta_m\right)^2.\]  In order to derive \eqref{Eq:DdotJ2}, it suffice  to show that for any $m\in \mathcal M$ there holds
\begin{equation}\label{Eq:WK2Est}
\text{For $k=0,1$, for any $h\in L^\infty(\T)$,} \Vert \n^k \dot\theta_m\Vert_{L^2(\T)}\lesssim \Vert h\Vert_{W^{k-2,2}(\T)},
\end{equation}
and 
\begin{equation}
\Vert h\Vert_{W^{-1,2}(\T)}\lesssim \Vert \dot\theta_m\Vert_{W^{1,2}(\T)}.
\end{equation}
This is a standard consequence of the stability assumption \eqref{Eq:AssQ2} and of elliptic regularity; we refer to Appendix \ref{Ap:Fredholm} (see Proposition \ref{Th:NegativeSobolev}).
\end{proof}

We now use the following technical proposition:
\begin{proposition}\label{Pr:CompW12W22}
Let $d\geq 2$. We let 
\[ \mathcal Y_r:=\left\{\mu \in W^{-1,2}(\T),\, \mu\text{ measure, } \mathrm{supp}(\mu)\subset \B(0;r)\right\}.\]
Then there holds 
\begin{equation}\label{Eq:CompW12W22}
\lim_{r\to 0^+}\sup_{\mu\in \mathcal Y_r\setminus\{0\}}\frac{\Vert \mu\Vert_{W^{-2,2}(\T)}^2}{\Vert \mu\Vert_{W^{-1,2}(\T)}^2}=0.
\end{equation}
\end{proposition}
To prove it, we use the following lemma which follows from a standard localisation argument.

\begin{lemma}\label{Pr:SobolevComparison}
Let $j\in \N\setminus\{0\}$. The following holds: for any $r>0$ small enough, for any $f\in L^\infty(\T)$ supported in $\B(0;r)$, there holds 
\begin{equation}\label{Eq:SobolevComparison}
 \Vert f\Vert_{W^{-j,2}(\R^d)}^2\lesssim \Vert f\Vert_{W^{-j,2}(\T)}^2\lesssim \Vert f\Vert_{W^{-j,2}(\R^d)}^2.
\end{equation}
\end{lemma}

\begin{proof}[Proof of Proposition \ref{Pr:CompW12W22}]
We need to distinguish between the cases $d=2$ and $d\geq 3$.
\begin{enumerate}
\item When $d\geq 3$ we take for any $r>0$, $\phi_r$  a smooth cut-off:
\begin{equation}\label{Eq:CutOff}
\begin{cases}\text{$\phi_r$ is radially symmetric, decreasing,}
\\\text{ $\phi_r\equiv 1$ in $\B(0;r)$, $\phi\equiv 0$ in $\B(0;2r)^c$,}
\\\text{ $\Vert \n\phi_r\Vert_{L^\infty(\R^d)}\lesssim \frac1r$.}
\end{cases}\end{equation}
Let $u\in W^{2,2}(\R^d)$. As $\mu$ is  supported in $\B(0;r)$, there holds
\begin{align*}
\int_{\R^d}ud\mu&=\int_{\R^d} (\phi_r u)d\mu
\\&\leq \Vert \phi_r u\Vert_{W^{1,2}(\R^d)}\cdot \Vert \mu\Vert_{W^{-1,2}(\R^d)}.
\end{align*}
It thus suffices to show that 
\begin{equation}\label{Eq:CompGoal}
 \Vert \phi_r u\Vert_{W^{1,2}(\R^d)}\lesssim \omega(r) \cdot \Vert  u\Vert_{W^{2,2}(\R^d)},
\end{equation} with $\omega(r)\underset{r\to 0^+}\to 0$.
To prove \eqref{Eq:CompGoal} observe that 
\begin{align*}
\Vert \phi_r u\Vert_{W^{1,2}(\R^d)}^2&\lesssim \int_{\R^d} \left(\phi_r^2+|\n \phi_r|^2\right) u^2+\int_{\R^d} \phi_r^2|\n u|^2
\\&\lesssim\frac1{r^2} \int_{\B(0;2r)} u^2+\int_{\B(0;2r)}|\n u|^2.
\end{align*}
 $\nabla u  \in L^{2^*}(\R^d)$ where $2^* = \frac{2d}{d-2}$.  We obtain \[ \|\nabla u \|_{L^2(\B(0;2r))}^2 \lesssim \left( \int_{\B(0;2r)} |\nabla u |^{2^*} dx \right)^{\frac2{2^*}}\cdot |\B(0;2r)|^{1 - \frac2{2^*}} \] so that
\begin{equation}\|\nabla u \|_{L^2(\B(0;2r))}^2 \leq \|\nabla u \|_{L^{2^*}(\R^d)}^{2} \cdot |\B(0;2r)|^{\frac2d} \lesssim  \|u \|_{W^{2,2}}^2 \cdot r^2 \end{equation} and, finally,
\begin{equation}\label{Eq:o} \|\nabla u \|_{L^2(\B(0;2r))}\lesssim r \| u\|_{W^{2,2}(\R^d)}.\end{equation}
To handle the other term, we use Sobolev embeddings: if $d=3$, we obtain, from the Sobolev embedding $W^{2,2}\hookrightarrow L^\infty$, 
\[ \frac1{r^2}\int_{\B(0;2r)}u^2\lesssim \Vert u\Vert_{W^{2,2}(\T)}^2\cdot r.\]  
If $d=4$, we take any $q\in [1;+\infty)$, and obtain from the Sobolev embedding $W^{2,2}(\T)\hookrightarrow L^{q}(\T)$\[ \frac1{r^2}\int_{\B(0;2r)}u^2\lesssim \Vert u\Vert_{W^{2,2}(\T)}^2\cdot r^{4\left(1-\frac2q\right)-2}\] and it suffices to take $q$ large enough.  

If $d> 4$, we let $q^*:=\frac{2d}{d-4}$ so that $W^{2,2}\hookrightarrow L^{q^*}$ and we obtain, similarly
\[ \frac1{r^2}\int_{\B(0;2r)}u^2\lesssim \Vert u\Vert_{W^{2,2}(\T)}\cdot r^2.
\]
Thus, in any case, we derive 
\begin{equation}\label{Eq:o1}
\frac1{r^2}\int_{\B(0;2r)}u^2\lesssim r^\epsilon\end{equation} for some $\epsilon>0$.  \eqref{Eq:o}--\eqref{Eq:o1} give \eqref{Eq:CompGoal}.
\item When $d=2$, we take another cut-off: introduce the function $\phi_r$ defined as 
\begin{equation}\label{Eq:CutOff2}\phi_r:x\mapsto
\begin{cases}1\text{ if $|x|\leq r$},
\\\frac{ \ln\left(\frac{\sqrt{r}}{\| x\|}\right)}{\ln\left(\frac1{\sqrt{r}}\right)}\text{ if $r\leq \| x\| \leq \sqrt{r}$}
\\ 0\text{ else.}
\end{cases}\end{equation} We claim that 
\begin{equation}\label{Eq:LogEst}\int_{\B(0;\sqrt{r})}u^2|\n \phi_r|^2\lesssim \Vert u\Vert_{W^{2,2}(\T)}^2\cdot \frac{1}{{|\ln(r)|}}, \end{equation} which is sufficient to conclude using once more the estimate
\[\Vert \phi_r u\Vert_{W^{1,2}(\R^d)}^2\lesssim \int_{\R^d} \left(\phi_r^2+|\n \phi_r|^2\right) u^2+\int_{\R^d} \phi_r^2|\n u|^2\] and \eqref{Eq:o}.
 To establish \eqref{Eq:LogEst},  we first observe that 
\[ |\n \phi_r|^2\lesssim \frac1{|x|^2\ln\left(\frac1{\sqrt{r}}\right)^2}\] and simply bound this integral using the Sobolev embedding $W^{2,2}\hookrightarrow L^\infty$ to obtain
\begin{align*}
\int_{\B(0;\sqrt{r})}u^2|\n \phi_r|^2&\lesssim \Vert u\Vert_{W^{2,2}(\T)}^2\cdot \int_r^{\sqrt{r}}\frac{s}{s^2\ln\left(\frac1{\sqrt{r}}\right)^2}ds
\\&\lesssim \frac{ \Vert u\Vert_{W^{2,2}(\T)}^2}{\ln\left(\frac1{\sqrt{r}}\right)^2}\ln\left(\frac1{\sqrt{r}}\right)
\\&\lesssim  \frac{ \Vert u\Vert_{W^{2,2}(\T)}^2}{\ln\left(\frac1{\sqrt{r}}\right)},
\end{align*} which concludes the proof with $\omega(r)=\frac1{|\ln(r)|}$.
\end{enumerate} 

\end{proof}

\begin{proof}[Proof of Proposition \ref{Pr:LocalMinimality}]
This proposition is a direct consequence of Lemma \ref{Pr:OC2} and of Proposition \ref{Pr:CompW12W22}: there holds, for two positive $\alpha,\, \beta$,
\[ \ddot J(m)[h,h]\geq \alpha \Vert h\Vert_{W^{-1,2}(\T)}^2-\beta \Vert h\Vert_{W^{-2,2}(\T)}^2.\] Now, let $h=m'-m$ for some fixed $m$. If $h$ is supported in  $\B(x_0;r)$ for $r>0$ small enough, Proposition \ref{Pr:CompW12W22} guarantees 
\[ \ddot J(m)[h,h]\gtrsim \left(\alpha+\underset{r\to 0^+}o(1)\right)\Vert h\Vert_{W^{-1,2}(\T)}^2,\] where $\underset{r\to 0^+}o(1)$ does not depend on $h$. Taking $r>0$ small enough allows to conclude.
\end{proof}

\subsubsection{Proof of Points $(1)$ and $(2)$ of Theorem \ref{Th:ReformConstrained}}
Proving Point $(2)$ first provides an approach to Point $(1)$.\begin{proof}[Proof of Point $(2)$ of Theorem \ref{Th:ReformConstrained}]
We let $m\in \mathcal M$ and we set $h:=m-m^*$. From the mean-value theorem, there exists $s\in [0;1]$ such that 
\begin{align*}
0&\geq J(m)-J(m^*)
\\& =\dot J(m^*)[h]+\frac12\ddot J(m^*+s h)[h,h].\end{align*}
From Proposition \ref{Pr:LocalMinimality}, if $\{m\neq m^*\}\subset \B(x_0;r_0)$ for $r_0$ small enough, this gives \[ 0\geq \int_{\T}\eta_{m^*}(m-m^*)+\alpha \Vert m-m^*\Vert_{W^{-1,2}(\T)}^2.\]

\end{proof}
\begin{proof}[Proof of Point $(1)$ of Theorem \ref{Th:ReformConstrained}]
We first prove that $m^*=\mathds 1_{E^*}$ for some $E^*$.
Argue by contradiction and assume that $\omega^*:=\{0<m^*<1\}$ satisfies 
\begin{equation}\label{Eq:Abnormal}
|\omega^*|>0\end{equation} for some optimiser $m^*$. It follows from Proposition \ref{Pr:OC1} that for any $h\in L^\infty(\T)$  supported in $\omega^*$ and satisfying $\int_{\omega^*}h=0$ we have 
\[ \int_{\omega^*}\eta_{m^*}h=0,\]  so that there exists $c_{m^*}\in \R$ such that
\[ \omega^*\subset\{\eta_{m^*}=c_{m^*}\}.\] 
Take $\alpha,\, r_0$ given by Proposition \ref{Pr:LocalMinimality} and let $h=m-m^*\not\equiv 0$ be supported in $\B(x_0;r_0)\cap \omega^*$ for some $x_0$. To obtain the existence of such an $h$, it suffices to take $h=\mathds 1_{F}-m^*$ where $F\subset \B(x_0;r_0)\cap \omega^*$ satisfies $|F|=\int_{\B(x_0;r_0)\cap \omega^*}m^*$, so that $m=m^*+h$ satisfies $m\in \mathcal M$.
Proposition \ref{Pr:LocalMinimality} gives 
\[ 0\geq \underbrace{\int_{\omega^*}\eta h}_{=0}+\alpha \Vert h\Vert_{W^{-1,2}(\T)}^2>0,\] a contradiction, and thus $m^*=\mathds 1_{E^*}$.

Now, from Proposition \ref{Pr:OC1}, it follows that there exists $c_{m^*}$ such that
 \[\{\eta_{m^*}>c_{m^*}\}\subset E^*\subset \{\eta_{m^*}\geq c_{m^*}\}.\] 
To conclude, it suffices to show that, if
\[ F:=E^*\cap \{\eta_{m^*}= c_{m^*}\}\] has positive measure, then 
\[ F=\{\eta_{m^*}=c_{m^*}\}.\] Assume that $|F|>0$ and, arguing by contradiction, let $G\subset  \{\eta_{m^*}= c_{m^*}\}\setminus F$ have positive measure. Let $x_F,\, x_G$ be Lebesgue points of $F$ and $G$ respectively. Consider 
\[ h:=\mathds 1_{\B(x_G;r_G)\cap G}-\mathds 1_{\B(x_F;r_F)\cap F}\] where $r_F, r_G$ are chosen to ensure that $\int h=0$, and $r_F,\, r_G\in (0;r_0)$ where $r_0$ is given by Proposition \ref{Pr:LocalMinimality}. As $\eta_{m^*}$ is constant on $F\cup G$, we deduce  
\[ 0\geq \underbrace{\int_{\omega^*}\eta h}_{=0}+\alpha \Vert h\Vert_{W^{-1,2}(\T)}^2>0,\] a contradiction.

 Now, let $m^*=\mathds 1_{E^*}$ be a solution of \eqref{Eq:PvNonEnergetic}. Expanding \eqref{Eq:FBNE},  $\eta_{m^*}$ solves 
\begin{align*}
-\Delta \eta_{m^*}&=-2\n\cdot\left(\eta_{m^*}\n\theta_{m^*} \right)+\partial_\theta Q(x,\theta_{m^*})\eta_{m^*}+\partial_\theta j(x,\theta_{m^*})
\\&={\partial_\theta Q(x,\theta_{m^*})\eta_{m^*}+\partial_\theta j(x,\theta_{m^*})-2\langle \n \eta_{m^*},\n\theta_{m^*}\rangle}-2\eta_{m^*}\Delta\theta_{m^*}
\\&=\left({\partial_\theta Q(x,\theta_{m^*})\eta_{m^*}+\partial_\theta j(x,\theta_{m^*})-2\langle \n \eta_{m^*},\n\theta_{m^*}\rangle} +2\eta_{m^*}(|\n \theta_{m^*}|^2+Q(x,\theta_{m^*}))\right)
\\&+2\eta_{m^*}m^*
\\&=f_{m^*}\mathds 1_{E^*}-g_{m^*}\mathds 1_{(E^*)^c}, 
\end{align*}
with \[\begin{cases}f_{m^*}:={\partial_\theta Q(x,\theta_{m^*})\eta_{m^*}+\partial_\theta j(x,\theta_{m^*})-2\langle \n \eta_{m^*},\n\theta_{m^*}\rangle} 2\eta_{m^*}(|\n \theta_{m^*}|^2+Q(x,\theta_{m^*}))+2\eta_{m^*},\\
g_{m^*}=-f_{m^*}+2\eta_{m^*}.\end{cases}\] Thus, 
\[ f_{m^*}+g_{m^*}=2\eta_{m^*}\geq \delta>0\] for some $\delta>0$ by Lemma \ref{Le:SwitchFunction}. By elliptic regularity 
\[ \forall p\in [1;+\infty), f_{m^*},\, g_{m^*}\in W^{1,p}(\T).\]
\end{proof}

\subsection{Second-order optimality conditions in the sense of shapes}\label{Se:Shape}
The last part of this section is devoted to Point $(3)$ of Theorem \ref{Th:ReformConstrained}; although it is possible to derive Point $(3)$ using a shape derivative framework, we obtain it directly from \eqref{Eq:QBT1}. We use the following definition:
\begin{definition}\label{De:AdmissibleTriplet}
Let $m^*=\mathds 1_{E^*}$ be an optimiser of \eqref{Eq:PvNonEnergetic}. Let $F$ be an open set with smooth boundary and  $\Gamma\subset \partial E^*$ be a smooth $\mathscr C^{2,\gamma}$  hypersurface such that $F\cap \partial E^*=\Gamma$. For $\Phi\in \mathscr C^\infty_c(\T)$, we say that the triplet $(F, \Gamma,\Phi)$ is admissible if $\Phi$ is supported in $F$ and satisfies 
\[ \int_\Gamma \langle \Phi,\nu\rangle=0,\]where $\nu$ denotes the outer normal to $\Gamma.$ 
\end{definition}
We begin with the following result:
\begin{proposition}\label{Pr:LocalMinimalityShape}
There exists a positive $\alpha,\, r_0>0$ such that the following holds: for any admissible triplet $(F,\Gamma,\Phi)$ such that $\mathrm{supp}(\Phi)\subset \B(x_0;r_0)$ for some $x_0\in \T$, 
\begin{equation}\label{Eq:Tada} \alpha \Vert \mathcal H_{\Phi,\Gamma}\Vert_{W^{-1,2}(\T)}^2\leq \int_\Gamma |\n \eta_{m^*}|\langle \Phi,\nu\rangle^2,\end{equation} where $\mathcal H_{\Phi,\Gamma}$ is the measure supported on $\Gamma$ defined as
\[ \mathcal H_{\Phi,\Gamma}=\left.\langle \Phi,\nu\rangle d\mathscr H^{d-1}\right\llcorner \Gamma.\]
\end{proposition}
\begin{proof}[Proof of Proposition \ref{Pr:LocalMinimalityShape}]
As a first step, observe that it suffices to show \eqref{Eq:Tada} for vector fields that satisfy $\Phi\mathds 1_{\langle \Phi,\nu\rangle>0},\, \Phi\mathds 1_{\langle \Phi,\nu\rangle<0}\in  W^{2,p}(\Gamma)$ for any $p$ large enough, as a standard mollification argument gives the density of
\[ \left\{ \Phi\text{ compactly supported in $F$ },\, \int_\Gamma \langle \Phi,\nu\rangle=0,\,  \Phi\mathds 1_{\{\langle \Phi,\nu\rangle<0\}},\,  \Phi\mathds 1_{\{\langle \Phi,\nu\rangle>0\}}\in W^{2,p}(\Gamma)\right\}\] in 
\[ \left\{ \Phi\text{ compactly supported in $F$ },\, \int_\Gamma \langle \Phi,\nu\rangle=0,\, \Phi \in W^{2,p}(\Gamma)\right\}\] for the strong $L^2(\Gamma)$ and $W^{-1,2}(\T)$ topologies. We let $\delta>0$ be fixed, and we choose $\delta_-=\delta_-(\delta)>0$ such that, letting 
\[ \Phi_\pm:=\Phi\mathds 1_{\langle \pm\langle \Phi,\nu\rangle>0\}},\]  setting 
\[ \tilde\Phi:=\delta \Phi_++\delta_-\Phi_-,\,  m_\delta:=m^*\circ \left(\mathrm{Id}+\delta\Phi_++\delta_-\Phi_-\right)^{-1}\] there holds
\[ \int_{\B(0;1)} m_\delta=\int_{\B(0;1)}m^*.\] Observe that
\[ \int_\Gamma \langle \Phi,\nu\rangle=0\Rightarrow \frac{\delta_-(\delta)}{\delta}\underset{\delta\to 0^+}\rightarrow 1.\]
Finally, we set $h_\delta:=m_\delta-m^*$. On the one-hand, we have \begin{equation}\label{Eq:ConvMeasure1}
\frac{h_\delta}{\delta}\underset{\delta \to 0}\rightarrow \mathcal H_{\Phi,\Gamma} \text{ weakly in $W^{-1,2}(\B(0;1))$ and in the sense of measures.}\end{equation}
This follows from a change of variables: for any $v\in W^{1,2}(\B_1)$, 
\begin{align*}
\int_{\B(0;1)} vh_\delta&=\delta \int_{\B(0;1)} m^* \left(\langle \n v, \Phi_+\rangle+v\n\cdot \Phi_+\right)
\\&+\delta_- \int_{\B(0;1)} m^* \left(\langle \n v, \Phi_-\rangle+v\n\cdot \Phi_-\right)+\underset{\delta \to 0}o(\delta)
\\&=\delta \int_\Gamma v\langle \Phi,\nu\rangle+\underset{\delta \to 0}o(\delta).
\end{align*}
On the other hand, it follows from the second-order Hadamard formula  \cite[Proposition 5.4.18]{zbMATH06838450} (applicable here since $\Phi_\pm$ are $W^{2,p}$ for any $p$ large enough) and the fact that $\eta=0$ on $\Gamma$ that 
\begin{equation}\label{Eq:Hadamard2} \int_{\B(0;1)}\eta h_\delta=-\delta^2\int_\Gamma |\n \eta_0|\cdot \langle \Phi,\nu\rangle^2+\underset{\delta \to 0^+}o(\delta^2).\end{equation}
Thus, from \eqref{Eq:QBT1}
\[ \delta ^2\left(\alpha \Vert \mathcal H_{\Phi,\Gamma}\Vert_{W^{-1,2}(\B(0;1))}^2- \int_\Gamma |\n \eta_{0}|\langle \Phi,\nu\rangle^2\right)+\underset{\delta\to 0^+}o(\delta^2)\leq 0\] and we deduce
  \eqref{Eq:SOSD1}. 
\end{proof}

We can now conclude the proof of Theorem \ref{Th:ReformConstrained}.

\begin{proof}[Proof of Point $(3)$ of  Theorem \ref{Th:ReformConstrained}]
Consider a $\mathscr C^{2,\gamma}$ hypersurface $\Gamma\subset \{\partial \eta_{m^*}>c_{m^*}\}$   satisfying $\Gamma\subset \B(x_0;r_0)$ for some $x_0\in \T$ and where $r_0$ is given by Proposition \ref{Pr:LocalMinimalityShape}.
We first observe that for any $\Phi$ such that $(F,\Gamma,\Phi)$ is admissible (see Definition \ref{De:AdmissibleTriplet}) for some $F$,
\[ 
 \Vert \mathcal H_{\Phi,\Gamma}\Vert_{W^{-1,2}(\T)}^2=\int_\T |\n v_\Phi|^2\] where 
 \[\begin{cases}-\Delta v_\Phi=0&\text{ in }\T,\\ \left\llbracket \partial_\nu v_\Phi\right\rrbracket=-\langle \Phi,\nu\rangle&\text{ on }\Gamma, \\ \int_\T v_\Phi=0.\end{cases}\] This relies on the condition $\int_\Gamma \langle \Phi,\nu\rangle=0$.
 Consider now the solution $u_\Phi$ to 
 \begin{equation}\label{Eq:uphi}
 \begin{cases}
 -\Delta u_\Phi=0&\text{ in }\T,\\ \left\llbracket \partial_\nu u_\Phi\right\rrbracket=-\langle \Phi,\nu\rangle&\text{ on }\Gamma, \\ \int_\Gamma\frac{  u_\Phi}{|\n \eta_{m^*}|}=0.
 \end{cases}
 \end{equation}
 As $v_\Phi-u_\Phi$ is constant, we deduce that we also have 
 \[ 
 \Vert \mathcal H_{\Phi,\Gamma}\Vert_{W^{-1,2}(\T)}^2=\int_\T |\n u_\Phi|^2,\] so that Proposition \ref{Pr:LocalMinimalityShape} rewrites: for a constant $\alpha$ (independent of $\Gamma$ and $\Phi$) there holds 
 \begin{equation}\label{Eq:Ki} \alpha \int_\T|\n u_\Phi|^2\leq \int_\Gamma |\n \eta_{m^*}|\cdot \langle \Phi,\nu\rangle^2.\end{equation}
 Now, consider 
 the lowest eigenvalue 
\begin{equation}\label{Eq:SigmaRayleigh} \sigma_1(\Gamma):=\min_{v\in W^{1,2}(\T),\, v\not\equiv 0\text{ on }\Gamma,\,\int_\Gamma \frac{v}{|\n \eta_{m^*}|}=0  }\frac{\int_\T |\n v|^2}{\int_\Gamma \frac{v^2}{|\n\eta_{m^*}|}}\end{equation} and an associated eigenfunction 
$\psi_1$. The associated eigen-equation reads 
\begin{equation}\label{Eq:EigenEquation}
\begin{cases}
-\Delta \psi_1=0&\text{ in }\T,\, 
\\ \llbracket \partial_\nu \psi_1\rrbracket=-\frac{\sigma_1(\Gamma)}{|\n \eta_{m^*}|} \psi_1&\text{ on }\Gamma,\, 
\\\int_\Gamma \frac{\psi_1}{|\n \eta_{m^*}|}=0.
\end{cases}
\end{equation}
Take the vector field
\[ \Phi:=\frac{\psi_1}{|\n \eta_{m^*}|}\nu\text{ on }\Gamma,\] so that 
\[ u_\Phi=\frac1{\sigma_1(\Gamma)}\psi_1.\] It follows from \eqref{Eq:Ki} that 
\[ \frac{\alpha}{\sigma_1(\Gamma)}\int_\Gamma \frac{\psi_1^2}{|\n \eta_{m^*}|}\leq \int_\Gamma |\n \eta_{m^*}|\langle \Phi,\nu\rangle^2=\int_\Gamma \frac{\psi_1^2}{|\n \eta_{m^*}|}\] and, consequently, 
\[ \sigma(\Gamma)\geq \alpha.\] The variational formulation \eqref{Eq:SigmaRayleigh} allows to conclude.
\end{proof}

\begin{remark}[Comments on these stability criteria]\label{Re:Stability}
The eigenvalue problem associated with $\sigma_1(\Gamma)$ was introduced in \cite{zbMATH08079247} in order to prove stability of certain optimal controls.
The stability condition \eqref{Eq:SOSD1}  is similar to the one derived (without volume constraint) by Monneau \& Weiss  \cite[Lemma 8.4]{zbMATH05129536} for the inverse obstacle problem
\[ \min_{u}\int_{\T}|\n u|^2-\int_{\T} u_+.\] However,  \cite[Lemma 8.4]{zbMATH05129536} relies in a crucial capacity on the variational formulation of the problem, and their approach can not be applied here. We note that  Chanillo, Kenig \& To \cite[Lemma 2.2]{zbMATH05505659} recovered the Monneau \& Weiss criterion  using a second variation approach but, likewise, their analysis relies on the fact that the (volume constrained in their case) composite membrane problem admits a variational formulation.
\end{remark}

\subsection{A stability criterion in the two-dimensional case}\label{Se:ASCConstrained}
The main result of this section is a consequence of Point $(3)$ of Theorem \ref{Th:ReformConstrained}, and is a (minor) adaptation of \cite[Lemma 3.2]{zbMATH05505659}, which itself gave a precise formulation to \cite[Proof of Theorem 8.1]{zbMATH05129536}:
\begin{theorem}\label{Th:ASCConstrained}
Assume $d=2$. Let $x_0\in \partial^eE^*$, $r_0>0$ be given by Proposition \ref{Pr:LocalMinimality} and define, for any $k\in \N$, $r_k:=\frac{r_0}{2^k}$. If, for any $k\in \N$, there exists a regular curve $\gamma_k:[0;1]\to \partial^e E^*,\, \gamma_k\in  W^{2,p}$ for any $p\in [1;+\infty)$ such that 
\[\begin{cases}
\gamma_k([0;1])\subset \partial^eE^*\cap\{\n \eta_{m^*}\neq 0\}\cap \B(x_0;r_k)\setminus \B(x_0;r_{k+1}),
\\ \mathscr H^1 \left(\gamma_k([0;1)]\right)\gtrsim \frac1{2^k}.
\end{cases}\]
Then 
\[ |\n \eta_{m^*}|(x_0)>0.\]
\end{theorem} In other words, a point that is surrounded by a large enough amount of a regular part of the free boundary can not be critical. The proof of this result relies on the following lemma (see \cite[Lemma 3.1]{zbMATH05505659}):
\begin{lemma}\label{Le:CKT}
Let $d=2$. Let $x_0\in \partial^eE^*$, $r_0>0$ and $\{r_k\}_{k\in \N}$ be a decreasing sequence converging to 0. If, for any $k\in \N$, there exists a regular curve $\gamma_k:[0;1]\to \partial^e E^*,\, \gamma_k\in  W^{2,p}$ for any $p\in [1;+\infty)$ such that 
\[
\gamma_k([0;1])\subset \partial^eE^*\cap\{\n \eta_{m^*}\neq 0\}\cap \B(x_0;r_k)\setminus \B(x_0;r_{k+1}),
\]
then 
\begin{equation}\label{Eq:ASCConstrained}\sum_{k=0}^\infty \int_{\gamma_k([0;1])}\frac1{|\n \eta_{m^*}|}<+\infty.\end{equation}
\end{lemma}

\begin{remark}[Regarding the dimensional condition]
Lemma \ref{Le:CKT}--Theorem \ref{Th:ASCConstrained} are true in any dimension (up to replacing the condition $\mathscr H^1 \left(\gamma_k([0;1)]\right)\gtrsim \frac1{2^k}$ with $\mathscr H^{d-1} \left(\gamma_k(\B_{\R^{d-1}}(0;1)\right)\gtrsim \frac1{2^{k(d-1)}}$), but their implications are not as clear-- the condition is  too weak, which can be explained by scaling properties of the torsion function, see Section \ref{Se:Why2D}.
\end{remark}

\begin{proof}[Proof of Lemma \ref{Le:CKT}]
We follow \cite[Lemma 3.1]{zbMATH05505659} and rely on the stability criterion \eqref{Eq:SOSD1}. Argue by contradiction and assume that 
\begin{equation}\label{Eq:Testard}\sum_{k=0}^\infty \int_{\gamma_k([0;1])}\frac1{|\n \eta_{m^*}|}=+\infty.\end{equation} We introduce a function $V:\R^d\to \R$ that satisfies 
\begin{enumerate}
\item $V$ is smooth and radially symmetric, 
\item $V(0)=1$, $V\in (0;1)$ in $\B\left(0;\frac12\right)$, 
\item $V\in (-1;0)$ in $\B(0;1)\setminus\B\left(0;\frac12\right)$, 
\item $V=0$ in $\B\left(0;\frac12\right)^c$.\end{enumerate} We then set, for any $k\in \N$, $v_k:=V\left(\frac{\cdot}{r_k}\right)$, and note that 
\[ \int_\T |\n v_k|^2\leq \int_{\R^d}|\n V|^2.\]
Now, let $\Gamma_{0,k}:=\cup_{j=0}^k \gamma_k((0;1))$.
From \eqref{Eq:Testard} and the sign of $V$, one can easily find a smooth subpart $\Gamma_{1,k}$ of $\cup_{j>k}\gamma_k((0;1))$ such that, setting $\Gamma_k:=\Gamma_{0,k}\cup \Gamma_{1,k}$, there holds 
\[ \int_{\Gamma_k}\frac{v_k}{|\n \eta_{m^*}|}=0\] In particular, from \eqref{Eq:SOSD1}, 
\[ \sigma\int_{\Gamma_k}\frac{v_k^2}{|\n \eta_{m^*}|}\lesssim 1.\] Observing that $\Gamma_{0,k}\subset \Gamma_k$ and that $v_k^2\equiv 1$ on $\Gamma_{0,k}$, we deduce that 
\[ \int_{\Gamma_{0,k}}\frac1{|\n \eta_{m^*}|^2}\lesssim 1\]and the conclusion follows from letting $k\to \infty$.
\end{proof}

\begin{proof}[Proof of Theorem \ref{Th:ASCConstrained}]
We follow \cite[Lemma 3.2]{zbMATH05505659}: argue by contradiction and assume that $\n \eta_{m^*}(x_0)=0$. Recall the Zygmund class inequality \cite[Proposition 2.3.7]{zbMATH01201583}: as $\eta_{m^*}(x_0)=0\,, \n\eta_{m^*}(x_0)=0$,  
\begin{equation}\label{Eq:Zygmund0} |\n \eta_{m^*}(x)|\lesssim \Vert x-x_0\Vert \cdot \ln\left( \Vert x-x_0\Vert \right). \end{equation}Consequently, we obtain from Lemma \ref{Le:CKT} that 
\[ \sum_{k=0}^\infty \int_{\gamma_k((0;1))}\frac1{ \Vert x-x_0\Vert \cdot \ln\left( \Vert x-x_0\Vert \right)}<\infty.\]
However, given the fact that $r_k=\frac{r_0}{2^k}$ and the condition $\mathscr H^1 \left(\gamma_k([0;1)]\right)\gtrsim \frac1{2^k}$, this implies
\[ \sum_{k=0}^\infty\frac1k<\infty,\] a contradiction.

\end{proof}

\subsection{A comment on energy minimising solutions}\label{Se:LocalMinimality}
Consider, as in the introduction, the energy functional 
\begin{equation}\mathscr E:\mathcal M\times W^{1,2}(\T)\ni (m,\eta)\mapsto \frac12\int_{\T}|\n \eta|^2-\int_{\T}\left((f+g)m\eta-g\eta\right).\end{equation}Then it is fairly straightforward to deduce from Proposition \ref{Pr:LocalMinimality} that $m^*:=\mathds 1_{\{\eta_{m^*}>c_{m^*}\}}$ is a local minimiser of the functional 
\[ m\mapsto \mathscr E(m,\eta_{m^*})\] in the sense that 
\[ \exists r_0>0,\,, \forall m\in \mathcal M,\, \{m\neq m'\}\subset \B(x_0;r_0),\, \mathscr E(m^*,\eta_{m^*})\leq \mathscr E(m,\eta_{m^*})\] and, as already observed in the introduction, $\eta_{m^*}$ is a critical point (in the sense of shapes) of $\mathscr E(m^*,\cdot)$. One could hope that the second-order stability condition provided by Proposition \ref{Pr:LocalMinimalityShape} would yield that $(m^*,\eta_{m^*})$ is a local minimiser (or at least a minimiser for perturbations with small enough support). First observe that this would still not get rid of the difficulties coming from the volume constraint. Second, and this is more problematic, Proposition \ref{Pr:LocalMinimalityShape} only provides information at regular parts of the free boundary. However, satisfying second order optimality conditions on regular parts is not enough to derive regularity--this is already the case when considering minimal surfaces.
\section{ The case of penalised problems}\label{Se:UFBP2}

\subsection{Main result}
The main result is the following:
\begin{theorem}\label{Th:ReformPenalised}
Let $m^*$ solve \eqref{Eq:PvNonEnergetic2}. Let $\eta_{m^*}$ denote the associated switch function, that is, the unique solution of 
\begin{equation}\label{Eq:C2IMNonEnergetic}
-\Delta \eta_{m^*}+2\n\cdot\left(\eta_{m^*}\n \theta_{m^*}\right)-\partial_\theta Q(x,\Theta_{m^*})\eta_{m^*}=\partial_\theta j(x,\theta_{m^*})\text{ in }\T.
\end{equation}
Then:
\begin{enumerate}
\item \underline{Reformulation as a free boundary problem:} $m^*=\mathds 1_{E^*}$ for some $E^*$ and either
\[ E^*={\{\eta_{m^*}>c\}}\]
or 
\[ E^*={\{\eta_{m^*}\geq c\}}.\] In particular, $(\theta_{m^*},\eta_{m^*})$ solves the free boundary system 
\begin{equation}\label{Eq:FBNE2}
\begin{cases}
-\Delta \theta_{m^*}-|\n \theta_{m^*}|^2=\mathds 1_{\{\eta_{m^*}> c\}}+Q(x,\theta_{m^*})&\text{ in }\T,\, 
\\ -\Delta \eta_{m^*}+2\n\cdot\left(\eta_{m^*}\n \theta_{m^*}\right)-\partial_\theta Q(x,\theta_{m^*})\eta_{m^*}=\partial_\theta j(x,\theta_{m^*})&\text{ in }\T,
\end{cases}
\end{equation}Setting $\eta:=\eta_{m^*}-c$,  there exist $f,\, g$ such that
\[\forall p\in [1;+\infty),\, f,\, g\in W^{1,p}(\T),\, \min_\T\left(f+g\right)>0\] and 
\[ -\Delta \eta=(f+g)\mathds 1_{\{\eta>0\}}-g.\]
\item\underline{Second-order optimality conditions I:} with the same notation for $\eta$ there exist constants $\alpha>0$, $r_0>0$ such that the following holds: for any $m_0\in \mathcal M_0$ satisfying 
\[ \{m\neq m^*\}\subset \B(x_0;r_0)\] for some $x_0\in \T$, there holds
\begin{equation}\label{Eq:QBT2} 0\geq \int_{\T}\eta(m-m^*)+\alpha \Vert m-m^*\Vert_{W^{-1,2}(\T)}^2.\end{equation}
\item\underline{Second-order optimality conditions II:} there exists $r_0,\,\sigma>0$ such that, for any $\Gamma\subset \partial\{\eta>0\}$ of class $\mathscr C^{2,\gamma}$ (for any $\alpha \in (0;1)$), for any $v\in W^{1,2}(\T)$, \begin{equation}\label{Eq:SOSD2}\sigma\int_\Gamma\frac{v^2}{|\n \eta|}\leq \int_\T |\n v|^2+\int_\T v^2.\end{equation}
\end{enumerate}
\end{theorem}

\subsection{Reformulation as a free boundary problem}
Points $(2)$ and $(3)$ of Theorem \ref{Th:ReformPenalised} are proved exactly as Points $(1)$ and $(2)$ of Theorem \ref{Th:ReformConstrained}. Indeed, the proof of the latter follows from a first and second-order sensitivity analysis of $J$ and, introducing $I:m\mapsto \int_\T m$, we see that, for the penalised problem \eqref{Eq:PvNonEnergetic2}, the optimality conditions read: for any $m\in \mathcal M_0$,
\begin{enumerate}
\item $\dot J(m^*)[m-m^*]-c\dot I(m^*)[h]\leq 0$, which rewrites:
\[ \forall m\in \mathcal M_0,\, \int_\T (\eta_{m^*}-c)\left(m-m^*\right)\leq 0.\]  
\item For any $m,\, m'\in \mathcal M_0$, keeping in mind that $\ddot I(m)=0$, we still have (the proof is the same here)
\[ \left(\ddot J-c\ddot I\right)(m)[m'-m,m'-m]\gtrsim \Vert m-m'\Vert_{W^{-1,2}(\T)}^2.\]
\end{enumerate}
These information allow to conclude that $m^*=\mathds 1_{E^*}$ with $E^*=\{\eta_{m^*}>c\}$ or $\{\eta_{m^*}\geq c\}$ and, similarly, that there exists a constant $\alpha>0$ such that 
\[ \forall m\in \mathcal M_0,\, \{m\neq m^*\}\subset \B(x_0;r_0),\, 0\geq \int_{\T}\eta(m-m^*)+\alpha \Vert m-m^*\Vert_{W^{-1,2}(\T)}^2.\]
The rewriting of the equation under the form $ -\Delta \eta=(f+g)\mathds 1_{\{\eta>0\}}-g$ is identical.

\subsection{Second-order optimality conditions in the sense of shapes}
\subsubsection{The key lower estimate}
The main result is the following:
\begin{proposition}\label{Pr:LocalMinimalityShape2}
There exists a positive $\alpha,\, r_0>0$ such that the following holds: for any  $(F,\Gamma)$ such that $\Gamma$ is $\mathscr C^{2,\gamma}$ and $F$ is an open set such that  $F\cap  \{\partial \eta_{m^*}>c\}=\Gamma$,  for any smooth vector field $\Phi$ satisfying $\mathrm{supp}(\Phi)\subset \B(x_0;r_0)$ for some $x_0\in \T$, 
\[ \alpha \Vert \mathcal H_{\Phi,\Gamma}\Vert_{W^{-1,2}(\T)}^2\leq \int_\Gamma |\n \eta_{m^*}|\langle \Phi,\nu\rangle^2,\]
where we still use the notation 
\[ \mathcal H_{\Phi,\Gamma}:=\langle \Phi,\nu\rangle d\mathscr H^{d-1}\llcorner\Gamma.\]
\end{proposition}
The proof is identical (and simpler) than that of Proposition \ref{Pr:LocalMinimalityShape}.
\begin{proof}[Proof of Proposition \ref{Pr:LocalMinimalityShape2}]
We consider, for any $\delta>0$ small enough, 
\[ m_\delta:=m^*\circ (\mathrm{Id}+\delta \Phi)^{-1}
\] and $h_\delta:=m_\delta-m^*$.   We already observed in the proof of Proposition \ref{Pr:LocalMinimalityShape} that
\begin{equation}\label{Eq:ConvMeasure2}
\frac{h_\delta}{\delta}\underset{\delta \to 0}\rightarrow \mathcal H_{\Phi,\Gamma} \text{ weakly in $W^{-1,2}(\B(0;1))$ and in the sense of measures.}\end{equation}
Once more, it follows from the second order Hadamard formula  \cite[Proposition 5.4.18]{zbMATH06838450} and the fact that $\eta=0$ on $\Gamma$ that 
\begin{equation}\label{Eq:Hadamard22} \int_{\B(0;1)}\eta h_\delta=-\delta^2\int_\Gamma |\n \eta|\cdot \langle \Phi,\nu\rangle^2+\underset{\delta \to 0^+}o(\delta^2).\end{equation}
We thus derive from \eqref{Eq:QBT2} that 
\[ \delta ^2\left(\alpha \Vert \mathcal H_{\Phi,\Gamma}\Vert_{W^{-1,2}(\B(0;1))}^2- \int_\Gamma |\n \eta_{0}|\langle \Phi,\nu\rangle^2\right)+\underset{\delta\to 0^+}o(\delta^2)\leq 0\] and so
  \eqref{Eq:SOSD2} holds. 

\end{proof}

\subsubsection{Exploiting second-order optimality conditions: proof of Point $(3)$ of Theorem \ref{Th:ReformPenalised}}
\begin{proof}[Proof of Point $(3)$ of Theorem \ref{Th:ReformPenalised}]
The proof follows from a similar spectral argument. We let $\Gamma,\, \Phi$ be in the conditions of Proposition \ref{Pr:LocalMinimalityShape2}. Introduce the solution $w_\Phi$ to 
\begin{equation}\label{Eq:wPhi}
-\Delta w_\Phi+w_\Phi:=\mathcal H_{\Phi,\Gamma}\end{equation} so that 
\[ \Vert \mathcal H_{\Phi,\Gamma}\Vert_{W^{-1,2}(\T)}^2=\int_\T |\n w_\Phi|^2+\int_\T w_\Phi^2.\]
Introduce the eigenvalue
\begin{equation}\label{Eq:TauRayleigh} \tau_1(\Gamma):=\min_{w\in W^{1,2}(\T),\, w\not\equiv 0\text{ on }\Gamma }\frac{\int_\T |\n w|^2+\int_\T w^2}{\int_\Gamma \frac{w^2}{|\n\eta_{m^*}|}}\end{equation} and an associated eigenfunction 
$\omega_1$. The associated eigen-equation reads 
\begin{equation}\label{Eq:EigenEquation1}
\begin{cases}
-\Delta \omega_1+\omega_1=0&\text{ in }\T,\, 
\\ \llbracket \partial_\nu \omega_1\rrbracket=-\frac{\tau_1(\Gamma)}{|\n \eta_{m^*}|} \psi_1&\text{ on }\Gamma,\, 
\\\int_\Gamma \frac{\psi_1}{|\n \eta_{m^*}|}=0.
\end{cases}
\end{equation}
Take the vector field\[ \Phi:=\frac{\omega_1}{|\n \eta_{m^*}|}\nu\text{ on }\Gamma,\] so that (in a fashion similar to Proposition \ref{Pr:LocalMinimalityShape}) \[ \tau_1(\Gamma)\geq \alpha.\] The variational formulation \eqref{Eq:TauRayleigh} allows to conclude.
\end{proof}

\part{Blow-up analysis and non-degeneracy}\label{Se:BU}

\section*{Plan of the part}
We lay the groundwork for the regularity of the free boundary $\partial^eE^*$ for \eqref{Eq:PvNonEnergetic}--\eqref{Eq:PvNonEnergetic2} (where the measure theoretic boundary $\partial^eE^*$ is defined in \eqref{De:Boundary}).
In Section \ref{Se:Hodograph}, we show that the free boundary is regular around points where $\n \eta_{m^*}\neq 0$--this follows from a standard hodograph transform. This reduces the study of regularity to critical points, which we set to do in Section \ref{Se:Weiss}  with the Weiss functional \cite{zbMATH01161229}.
We  use the Weiss functional in the spirit of Chanillo, Kenig \& To \cite{zbMATH05505659} to show the $\mathscr C^{1,1}$ regularity of $\eta$ in the two-dimensional case--this is Section \ref{Se:2dReg}. Finally, in Section \ref{Se:NonDegeneracy}, we study the non-degeneracy of blow-up limits, establishing it at points of intermediate density for the constrained problem \eqref{Eq:PvNonEnergetic}, and at all points for the penalised problem \eqref{Eq:PvNonEnergetic2}. In these last two sections, the second-order optimality conditions provided by Points $(2)$ and $(3)$ of Theorems \ref{Th:ReformConstrained}--\ref{Th:ReformPenalised} are instrumental.

\section{Regularity at non-critical points}\label{Se:Hodograph}
In this section and in Section \ref{Se:Weiss}, we consider a solution $\eta$ to 
\begin{equation}\label{Eq:Criticality}
-\Delta \eta=(f+g)\mathds 1_{\underbrace{\{\eta>0\}\text{ or }\{\eta\geq 0\}}_{=E^*}}-g\end{equation}
where 
\begin{equation}\label{Eq:C3Regularity} \forall p\in [1;+\infty),\, f,\, g\in W^{1,p}(\T),\, \min_{\T}(f+g)>0.\end{equation}
 
We begin with a study of regularity around so-called ``regular points":
\begin{definition}\label{De:RegularPoint}
Let $x_0\in \partial^eE^*$. We say that $x_0$ is a regular free boundary point if 
\[ \n \eta(x_0)\neq 0.\]
\end{definition}

\begin{theorem}\label{Th:Hodograph}
Let $x_0$ be a regular free boundary point. Then, for any $\gamma\in (0;1)$, there exist
    \begin{itemize}
        \item an open set $V\subset\R^d,\, x_0\in V$,
        \item a $(d-1)$-dimensional open set $V'\subset \R^{d-1}$,
        \item a $\mathscr C^{2,\gamma}$ mapping $\varphi:V'\to\R$,
    \end{itemize}
such that, up to a rotation, 
\[
\{\eta>0\}\cap V=\{(x',x_d)\in V'\times\R, x_d>\varphi(x')\}, \]
\[
\{\eta=0\}\cap V=\{(x',\varphi(x')), x'\in V'\}.
\]
In particular, the free boundary $\partial^eE^*$ is $\mathscr C^{2,\gamma}$ around regular points. 
\end{theorem}
The proof is done through a standard hodograph transform and follows from an adaptation of \cite[Lemma 3]{zbMATH01588525}. For the sake of completeness, we give it in Appendix \ref{Ap:Hodograph}.

\section{The Weiss monotonicity formula \& existence of blow-up at critical points}\label{Se:Weiss}
Theorem \ref{Th:Hodograph} allows to focus on critical points of $\partial^eE^*$, that is, on points $x_0\in \partial^eE^*$ such that 
\[ \eta(x_0)=0,\, \n \eta(x_0)=0.\] Without loss of generality, we assume that 
\[x_0=0.\]
The natural scaling associated with \eqref{Eq:Criticality} is quadratic, leading to considering the rescaled functions
 \[ \eta_r:=\frac{\eta(r\cdot)}{r^2}\,, f_r:=f(r\cdot)\,, g_r:=g(r\cdot).\]

 We now use the  Weiss boundary adjusted energy introduced in \cite{zbMATH01161229}, and a central tool for unstable free boundary problems \cite{zbMATH06021980,Andersson2006CrossshapedAD,zbMATH06179931,zbMATH05311233,zbMATH05505659,zbMATH05129536,zbMATH06062397,zbMATH05204068,zbMATH06912136,zbMATH07826622}: for $r>0$ small enough, set
\begin{equation}\label{Eq:C3Weiss}
\Psi(r):=\int_{\B(0;1)} \vert \nabla \eta_r \vert^2 - 2 \int_{\B(0;1)} f_r (\eta_r)_++g_r (\eta_r)_- - 2 \int_{\partial \B(0;1)} \eta_r^2
\end{equation}

\subsection{Quasi-monotonicity of the Weiss functional}
We begin with the following quasi-monotonicity result:

\begin{theorem}\label{Pr:C3Weiss} 
There exist two constants $C\,, \beta>0$ that only depend on $f$ and $g$ such that 
\[W: r\mapsto \Psi(r)+Cr^\beta\] is non-decreasing. More precisely, we have, for some constant $C$, 
\begin{equation}\label{Eq:C3QuantifiedWeiss}
\forall r\geq 0\,, \forall 2r>s>r\,, W(r)-W(s)\geq C\int_{\partial \B(0;1)} \left(\eta_s-\eta_r\right)^2.
\end{equation}
\end{theorem}
This result (with varying $f$ and $g$) is similar to that of Shahgholian \cite{zbMATH05204068}. Our proof differs from his in that we rely on the Sobolev regularity of $f$ and $g$ directly, rather than on their H\"{o}lder continuity; this allows for a simpler proof, which we give for the sake of completeness in Appendix \ref{Ap:QuasiMonotonicity}.

\subsection{Existence of blow-up limits}\label{Se:Existence}
We can, without loss of generality, assume that 
\begin{equation}\label{Eq:MW6}f(0)>0.\end{equation} Indeed, should $f(0)\leq 0$, Assumption \eqref{Eq:C3Regularity} entails $g(0)>0$, which in turn allows to work on $-\eta$. Under Assumption \eqref{Eq:MW6}, we are in a situation similar to that of Monneau \& Weiss \cite{zbMATH05129536}, and we follow the analysis of \cite[Proposition 5.1]{zbMATH05129536} (see also \cite[Theorem 3.1]{zbMATH05311233}). 

From Theorem \ref{Pr:C3Weiss}, we know that $W$ has a limit $W(0^+)\in [-\infty;+\infty)$ at $r=0$, so we can define 
 \[\Psi(0^+):=\lim_{r\to 0^+}\Psi(r)\in[-\infty,+\infty).\]
 We define, for any $r>0$, 
\[ S(r):=\int_{\partial \B(0;1)} \eta_r^2.\]
The main result of this section is the following:

\begin{theorem}\label{Th:MW1} 
The following dichotomy holds:
\begin{enumerate}
\item If $\Psi(0^+)>-\infty$, the sequence $\{S(r)\}_{r\to0^+}$ is bounded. Furthermore, for any $\gamma\in (0;1)$,  $\{\eta_r\}_{r\to 0^+}$ is bounded in $\mathscr C^{1,\gamma}(\B(0;1))$. Finally, any  $\mathscr C^{1,\gamma}\,, W^{1,2}$ limit $\eta_0$ of any subsequence $\{\eta_{r_k}\}_{k\in \N}$, $r_k\underset{k\to \infty}\to 0^+$ is two-homogeneous and satisfies
\begin{equation}\label{Eq:MW1}-\Delta \eta_0=(f(0)+g(0))\underline{m}_0-g(0),\end{equation} where 
\[ \underline{m}_0\in\mathrm{argmax}\left( L^\infty(\B(0;1);[0;1])\ni m\mapsto \int_{\B(0;1)} m\eta_0\right).\]
\item If $\Psi(0^+)=-\infty$, then \[ S(r)\underset{r\to 0+}\rightarrow +\infty.\]  Setting
\[ \tilde \eta_r:=\frac{\eta_r}{\sqrt{S(r)}},\] the sequence $\left\{\tilde \eta_r\right\}_{r\to 0^+}$ is bounded in $\mathscr C^{1,\gamma}(\B(0;1))$ for all $\gamma\in(0,1)$, and any  $\mathscr C^{1,\gamma}$ limit $\tilde \eta_0$ of any subsequence $\left\{\tilde\eta_{r_k}\right\}_{k\in \N}$ where $r_k\underset{k\to\infty}\to0^+$ is two-homogeneous, harmonic and non identically zero. 
\end{enumerate}
\end{theorem}
We note that an immediate corollary of Theorem \ref{Th:MW1} is the following:
\begin{corollary}\label{Co:Te}
If, for a given sequence $\{r_k\}_{k\in \N}$ going to $0^+$, 
\[ \underset{k\to\infty}{\lim\sup}\left\Vert \eta_{r_k}\right\Vert_{L^\infty(\B(0;1))}=+\infty\] then 
\[ \Psi(0^+)=-\infty.\]
\end{corollary}
The proof is lengthy and technical, but mostly follows from standard arguments. We defer it to Appendix \ref{Ap:WeissBU}. Nevertheless, we isolate a part of the proof in the following remark for future reference:
\begin{remark}[Regarding the equation satisfied by blow-up limits]
Assume we have $\mathscr C^{1,\gamma}$ bounds on the sequence $\{\eta_r\}_{r\to 0^+}$ and let $\eta_0$ be a closure point. Rewrite the equation on $\eta_r$ as 
\[-\Delta \eta_r=(f_r+g_r)m_r-g_r\] where $m_r:=\mathds 1_{\{\eta_r>0\}}$, In particular, $m_r$ is a solution of 
\[\text{ Maximise }T(\cdot,\eta_r):L^\infty(\B(0;1));[0;1])\ni m\mapsto \int_{\B(0;1)} m\eta_r.\] Letting $\underline{m}_0$ denote a weak $L^\infty-*$ closure point of the sequence $\{m_r\}_{r\to 0^+}$, we deduce first that 
\begin{equation}\label{Eq:MaxT} \underline{m}_0\in \mathrm{arg max}\,T(\cdot,\eta_0)\end{equation} and, passing to the limit in the equation, that 
\[-\Delta \eta_0=(f(0)+g(0))\underline{m}_0-g(0).\]
We will show in Section \ref{Se:NonDegeneracy} that non-trivial blow-up limits $\eta_0$ satisfy 
\[-\Delta \eta_0=f(0)\mathds 1_{\{\eta_0>0\}}-g(0)\mathds 1_{\{\eta_0\leq 0\}},\] but this requires a  finer analysis.\end{remark}

\section{$\mathscr C^{1,1}$ regularity in two dimensions}\label{Se:2dReg}
The main result of this section is the following:
\begin{theorem}\label{Th:C11Regularity}
Let $d=2$. Let $m^*$ be a solution of either the volume constrained or volume penalised problem \eqref{Eq:PvNonEnergetic}--\eqref{Eq:PvNonEnergetic2}.  Then $\eta_{m^*}\in \mathscr C^{1,1}(\mathbb{T}^2)$.
\end{theorem}
As we explained in the introduction (see in particular Section \ref{Se:Bibliography}), such a regularity result necessarily relies on second-order optimality conditions, as it was shown by Andersson \& Weiss \cite{Andersson2006CrossshapedAD} that there are non $\mathscr C^{1,1}$ solutions to the criticality (or first-order) equation 
\[ -\Delta u=\mathds 1_{\{u>0\}}.\] It should also be noted that in higher dimensions, there are also non $\mathscr C^{1,1}$ solutions that are close to being energy minimisers \cite{zbMATH06021980,zbMATH06179931}.

\begin{proof}[Proof of Theorem \ref{Th:C11Regularity}] This proof is a simplified version of \cite[Corollary 4.2]{zbMATH05505659}. Recall that $\eta=\eta_{m^*}-c_{m^*}$ or $\eta_{m^*}-c$ depending on the problem we study. First observe that $\eta$ is locally $\mathscr C^{1,1}$ around regular points, see Theorem \ref{Th:Hodograph}. Second, $\eta$ is $\mathscr C^{2,\gamma}$ in $\{\eta\neq 0\}$ so that, if $\eta$ is not $\mathscr C^{1,1}$,  there exists a point $x_0\in \partial^eE^*$ with $\eta(x_0)=0,\, \n \eta(x_0)=0$, and \[ \underset{r\to 0^+}{\lim\sup}\sup_{\Vert x\Vert \leq 1}\frac{|\eta(rx)|}{r^2}=+\infty.\] In other words, using the notation $\eta_r:=\frac{\eta(r\cdot)}{r^2}$, Theorem \ref{Th:MW1} guarantees that the Weiss functional $\Psi$ satisfies 
\[ \Psi(0^+)=-\infty,\] and that any $\mathscr C^{1,\gamma}$ closure point of the sequence
\[ \tilde \eta_r:=\frac{\eta_r}{\Vert \eta_r\Vert_{L^2(\partial \B(0;1))}}\] is a non-trivial, 2-homogeneous harmonic function in $\mathbb B(0;1)$. In dimension 2, this means that, up to a rotation and a multiplicative constant, the only closure point of $\{\tilde \eta_r\}_{r\to 0^+}$ is 
\begin{equation}\label{Eq:ISle}  \tilde\eta_0:(x_1,x_2)\mapsto x_1x_2.\end{equation} 
Let us show that this implies that $x_0=0$ satisfies the assumptions of Theorem \ref{Th:ASCConstrained}, which would immediately contradict $\n \eta(0)=0$.  To this end, consider the sequence $\left\{r_k:=\frac1{2^k}\right\}_{k\in \N}$. Up to working on the sequence $\{\tilde\eta_{r_k}\circ R_k\}_{k\in \N}$ for a suitable sequence of rotations $\{R_k\}_{k\in \N}$, we assume without loss of generality that 
\begin{equation}\label{Eq:eep} \tilde\eta_{r_k}\underset{k\to \infty}\rightarrow \tilde \eta_0\text{ in }\mathscr C^{1,\gamma}.\end{equation}
\begin{remark}Of course, we say we can make this assumption without loss of generality, but this is because we only consider here the length of the nodal set $\{\eta=0\}$. In general, the question of \emph{a priori} uniqueness of blow-up limits remains open.\end{remark}
Now, consider $\Sigma:=\left(\frac12;\frac34\right)\times \left(-\delta_0;\delta_0\right)$ where $\delta_0$ is small enough. By \eqref{Eq:eep} and the fact that \[ \min_{\Sigma}|\n \tilde\eta_0|>0,\] we deduce successively that:
\begin{enumerate}
\item For any $x_1\in \left(\frac12;\frac34\right)$, $\tilde\eta_{r_k}(x_1,-\delta)\lesssim-\delta,\, \tilde\eta_{r_k}(x_1,+\delta)\gtrsim \delta$.
\item This, combined with $|\nabla_{x_2}\tilde\eta_0|>0$ in $\Sigma$, entails that, for any $k$ large enough, for any $x_1\in \left(\frac12;\frac34\right)$, there exists a unique $y_k(x_1)\in (-\delta;\delta)$ such that $\tilde\eta_{r_k}(x_1,y_k(x_1))=0$. Furthermore, $|\n \tilde\eta_{r_k}(x_1,y_k(x_1))|>0$.
\end{enumerate}
\begin{figure}
\includegraphics[width=0.5\textwidth]{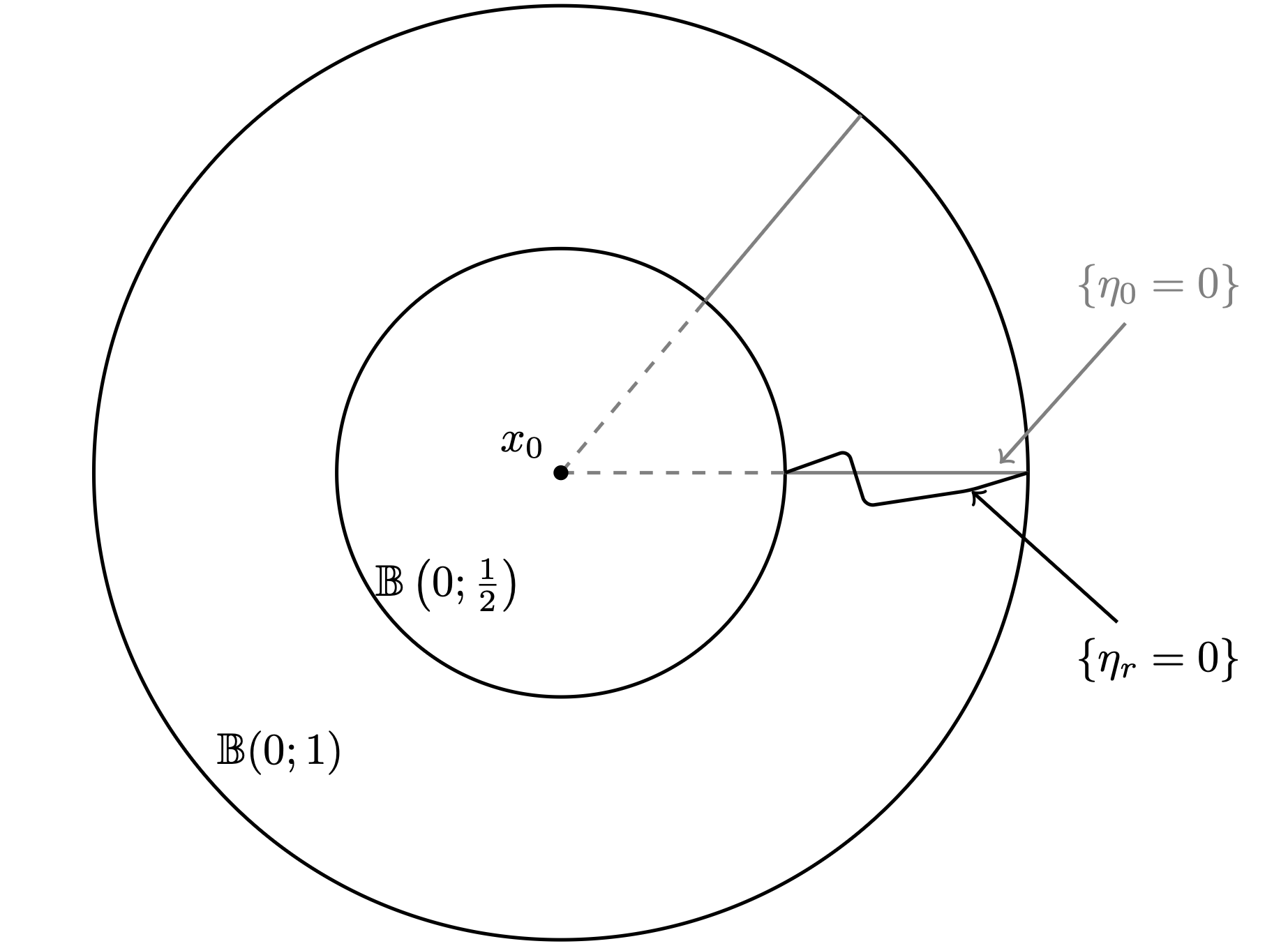}
\caption{Illustration of the construction of a ``long" part of the nodal set.}
\end{figure}
This implies that 
\begin{equation}\label{Eq:Deceived} \mathscr H^1\left(\{\eta_{r_k}=0\}\cap\{|\n \eta_{r_k}|\neq 0\}\cap \left(\B(0;1)\setminus \B\left(0;\frac12\right)\right)\right)\gtrsim 1.\end{equation} We then apply Theorem \ref{Th:ASCConstrained} to derive a contradiction.

\end{proof}

\section{(Non-)Degeneracy of blow-up limits and limit equations}\label{Se:NonDegeneracy}
We now investigate the non-degeneracy of blow-up limit.
This section is where the first distinction between penalised and constrained problems first materialises: in the case of volume constrained problems, we can only derive this non-degeneracy at points of intermediate density, while, for penalised problems, we obtain it for any free boundary point.

\subsection{The case of volume constrained problems}
We let, for any $x_0\in \T$ and any $r>0$ small enough, 
\begin{equation}\label{Eq:Density}
D_E(x_0;r):=\frac{|E\cap \B(x_0;r)|}{|\B(x_0;r)|}.
\end{equation}
We use the notations of Theorem \ref{Th:MW1}, recalling in particular that $\Psi$ is the Monneau-Weiss functional.
\begin{theorem}\label{Th:NonDegeneracyConstrained}
Let $m^*(=\mathds 1_{E^*}\text{ by Theorem \ref{Th:ReformConstrained}})$ be a solution of \eqref{Eq:PvNonEnergetic}. Let $x_0\in \T$ be a critical free boundary point. If 
\begin{equation}\label{Eq:NDAssumption}
0<\underset{r\to 0^+}{\lim \inf}D_E(x_0;r)\leq \underset{r\to 0^+}{\lim \sup}D_E(x_0;r)< 1
\end{equation}
then the blow-ups at $x_0$ are non-degenerate in the following sense: if $\Psi(0^+)>-\infty$, any closure point $\eta_0$ of 
$\{\eta_r\}_{r\to 0^+}$ provided by Theorem \ref{Th:MW1} is not identically zero and solves 
\[-\Delta\tilde\eta_0=(f(0)+g(0))\mathds 1_{\{\eta_0>0\}}-g(0).\]
\end{theorem}
\begin{remark}\label{Re:Reeb}Note that if $\Psi(0^+)=-\infty$, any closure point of $\{\tilde\eta_r\}_{r\to 0^+}$ is non-zero by Theorem \ref{Th:MW1}, whence we shall allow ourselves to write that ``all blow-ups are non-degenerate". \end{remark}
In fact,  Theorem \ref{Th:NonDegeneracyConstrained} is a direct corollary of the following optimal non-degeneracy estimate:
\begin{theorem}\label{Th:NDC}
With the same notations as in Theorems \ref{Th:MW1}, if $\Psi(0^+)>-\infty$, then, as $r\to 0^+$,
\begin{equation}\label{Eq:NonDegeneracyGeneral}
\left\Vert \eta_r\right\Vert_{L^2(\B(0;1))}\gtrsim D_E(x_0;r_j)^{\frac{d+4}{2d}}.
\end{equation}
\end{theorem}
Theorem \ref{Th:NDC} relies on the optimality conditions \eqref{Eq:QBT1}. We use the following lemma:
 \begin{lemma}\label{Le:Construction}
Assume without loss of generality that $x_0=0$ is a critical free boundary point. Let $\{r_j\}_{j\to 0^+}$ be a sequence converging to 0 and assume 
\begin{equation}\label{Eq:NoFullDensity}
\underset{j\to \infty}{\lim\sup} D_E(0;r_j)<1 \end{equation}. There exists a sequence of radially symmetric sets $\{F_j\}_{j\in \N},\, F_j\subset \B(0;1),\, |F_j|=D_E(0;r_j)$ that can be taken to be either centred balls or annuli   such that:
\begin{enumerate}
\item $m_j:=m^*\mathds 1_{\B(0;r_j)^c}+\mathds 1_{F_j}\left(\frac\cdot{r_j}\right)\in \mathcal M$, 
\item There holds 
\begin{equation}\label{Eq:CompW12} \Vert m^*-m_j\Vert_{W^{-1,2}(\T)}^2\gtrsim \left(r_j^dD_E(0;r_j)\right)^{\frac{d+2}d}.\end{equation} 
\end{enumerate}
\end{lemma}

\begin{proof}[Proof of Lemma \ref{Le:Construction}]
We take, for $F_j$, either a centred ball or a centred annulus. We use the notation (for $0<a\leq b$)
\[ \A(a;b):=\left\{a\leq \Vert \cdot\Vert\leq b\right\}.\]We introduce, at the blown-up scale, the sets 
\[ B_j:=\B\left(0;D_E(0;r_j)^{\frac1d}\right),\, A_j:=\A(\e_1(r_j),1)\] where $\e_1(r_j)$ is chosen to ensure that 
\[ |A_j|=D_E(0;r_j).\] Finally, we define 
\[ m_{j,B}:=m^*\mathds1_{\B(0;r_j)^c}+\mathds 1_{B_j}\left(\frac{\cdot}{r_j}\right),\, m_{j,A}:=m^*\mathds1_{\B(0;r_j)^c}+\mathds 1_{A_j}\left(\frac{\cdot}{r_j}\right).\] By construction, $m_{j,A}, m_{j,B}\in \mathcal M$. Furthermore, \eqref{Eq:NoFullDensity} yields, by explicit computations using the Fourier transform,
\begin{equation}\label{Eq:PetitClaim}
\Vert m_{j,A}-m_{j,B}\Vert_{W^{-1,2}(\T)}^2\underset{j\to \infty}\sim D_E(0;r_j)^{\frac{d+2}d}.
\end{equation} We thus deduce from the triangle inequality that 
\begin{equation}\label{Eq:Construction}
\underset{j\to \infty}{\lim\inf}\max\left(\Vert m^*-m_{A,j}\Vert_{W^{-1,2}(\T)}^2,\Vert m^*-m_{B,j}\Vert_{W^{-1,2}(\T)}^2\right)\gtrsim D_E(0;r_j)^{\frac{d+2}d}
\end{equation}and this concludes the proof.\end{proof}

\begin{proof}[Proof of Theorem \ref{Th:NDC}]
Let $\{r_j\}_{j\in \N}$ be a sequence converging to $0$ and let, for any $j\in \N$, $m_j:=m^*\mathds 1_{\B(0;r_j)^c}+\mathds 1_{F_j}\left(\frac{\cdot}{r_j}\right)$ be provided by Lemma \ref{Le:Construction}. We set $h_j:=m_j-m^*$. By the optimality conditions \eqref{Eq:QBT1}, we have 
\[-\int_\T \eta h_j\gtrsim \Vert h_j\Vert_{W^{-1,2}(\T)}^2\] for some constant independent of the sequence $\{r_j\}_{j\to \infty}$.
Recalling that $\eta_{r_j}=\frac{\eta(r_j\cdot)}{r_j^2}$ we obtain, after a change of variables and from \eqref{Eq:CompW12}
\[-r_j^{d+2}\int_{\B(0;1)} \eta_{r_j}h_j\left(\frac{\cdot}{r_j}\right)\gtrsim r_j^{d+2} D_E(0;r_j)^{\frac{d+2}2}.\]  As $h_j=m_j-m^*\in \{0;1\}$ a.e. the triangle inequality yields, 
\[ \left\Vert h_j\left(\frac{\cdot}{r_j}\right)\right\Vert_{L^2(\B(0;1)}\leq 2\sqrt{D_E(0;r_j)}.\] The Cauchy-Schwarz inequality gives
\[ D_E(0;r_j)^{\frac{d+2}{2}}\lesssim  \left\Vert \eta_r\right\Vert_{L^2(\B(0;1))}\sqrt{D_E(0;r_j)},\] and the non-degeneracy of $\eta_0$ follows.
\end{proof}

\begin{proof}[Proof of Theorem \ref{Th:NonDegeneracyConstrained}]The only point that remains to be studied is the limit equation. Let $\eta_0$ be the limit of some sequence $\{\eta_{r_j}\}_{j\in \N}$, where $r_j\underset{j\to \infty}\longrightarrow 0^+$ and, up to extracting a further subsequence, assume that 
\[ D_E(0;r_j)\underset{j\to\infty}\longrightarrow \underline{D}_0\in (0;1).\] Let
\[\underline{\mathcal M}:=\left\{m\in L^\infty(\B(0;1)):\, 0\leq m\leq 1\text{ a.e., }\int_{\B(0;1)}m=\underline{D}_0\right\}\] and, for any $r>0$, 
\[ \underline{\mathcal M}_r:=\left\{m\in L^\infty(\B(0;1)):\, 0\leq m\leq 1\text{ a.e., }\int_{\B(0;1)}m=D(0;r)\right\}.\]  Note that Assumption \eqref{Eq:NDAssumption} ensures that $\underline{\mathcal M}$ is infinite dimensional (in the sense that it generates, as a convex set, an infinite dimensional cone). For any $\underline m\in \underline{\mathcal M}$, fix a sequence $\{\underline{m}_j\}_{j\in \N}\in \prod_{j\in \N} \underline{\mathcal M}_{r_j}$ such that 
\[ \underline{m}_j\underset{j\to \infty}\longrightarrow \underline m\text{ weakly in $L^\infty-*$}\] and, finally, let $\underline{m}_0$ be a closure point (here again up to a further subsequence) of $\left\{m_r:=\mathds 1_{E^*\cap \B(0;r_j)}\left(\frac{\cdot}{r_j}\right)\right\}_{j\in \N}$. From \eqref{Eq:QBT1} applied to $h_j:=(\underline{m}_j-m_r)\left(\frac{\cdot}{r_j}\right)$ and Lemma \ref{Pr:SobolevComparison} we obtain, for any $j\in \N$, 
\[ 0\geq \int_{\B(0;1)} \eta_{r_j}h_j+\alpha \Vert h_j\Vert_{W^{-1,2}(\R^d)}^2.\] Passing to the limit, we deduce that for any $\underline{m}\in\underline{\mathcal M}$, there holds
\[ 0\geq \int_{\B(0;1)} \eta_0\left(\underline{m}-\underline{m}_0\right)+\alpha\left\Vert \underline{m}-\underline{m}_0\right\Vert_{W^{-1,2}(\R^d)}^2.\] We conclude, as in the proof of Theorem \ref{Th:ReformConstrained}, that this entails, first that $\underline{m}_0$ is bang-bang, say $\underline{m}_0=\mathds 1_{E_0}$, and second that $E_0$ is equal to either $\{\eta_0\geq 0\}$ or $\{\eta_0>0\}$. Now, recall that Assumption \eqref{Eq:MW6} is still in force, and so, from the fact that $\eta_0$ solves \eqref{Eq:MW1}, $|\{\eta_0=0\}\cap E_0|=0$. Thus, $E_0=\{\eta_0>0\}$ and the conclusion follows.
\end{proof}

\subsection{The case of penalised problems}

The situation is easier for penalised problems, as we have more freedom regarding the choice of competitors (and we do not need a positive density assumption).\begin{theorem}\label{Th:NonDegeneracyPenalised}
Let $m^*(=\mathds 1_{E^*}\text{ by Theorem \ref{Th:ReformPenalised}})$ be a solution of \eqref{Eq:PvNonEnergetic2}. Let $x_0\in \T$ be a critical free boundary point. The blow-ups at $x_0$ are non-degenerate in the following sense: if $\Psi(0^+)>-\infty$, any closure point $\eta_0$ of 
$\{\eta_r\}_{r\to 0^+}$ provided by Theorem \ref{Th:MW1} is not identically zero and solves 
\begin{equation}\label{Eq:LimitEquation}-\Delta\eta_0=(f(0)+g(0))\mathds 1_{E_0^*}-g(0)\end{equation} with 
\[ E_0^*=\{\eta_0>0\}\text{ or }E_0^*=\{\eta_0\geq 0\}.\]

\end{theorem}
Remark \ref{Re:Reeb} also applies here.

\begin{proof}[Proof of Theorem \ref{Th:NonDegeneracyPenalised}]
We use \eqref{Eq:QBT2} again. Argue by contradiction and assume that for some sequence $\{r_j\}_{j\to 0^+}$ the blow-up limit provided by Theorem \ref{Th:MW1} is degenerate and, up to working on $-\eta$, assume that \eqref{Eq:MW6} holds. It follows from Theorem \ref{Th:NonDegeneracyConstrained} that (taking a further subsequence if necessary)
\[ D_E(0;r_j)\underset{j\to \infty}\rightarrow 0.\] Take $m_j:=m^*\mathds 1_{\B(0;r_j)^c}+\mathds 1_{\B\left(0;\frac{r_j}2\right)}$ and $h_j:=m_j-m^*$. Then we have, on the one-hand
\[ -\int_\T \eta h_j=-r_j^{d+2}\int_{\B(0;1)}\eta_{r_j}h_j\left(\frac{\cdot}{r_j}\right)=\underset{j\to\infty}o \left(r_j^{d+2}\right)\] by the degeneracy assumption and, on the other hand
\[ \Vert h_j\Vert_{W^{-1,2}(\T)}^2\underset{j\to \infty}\sim r_j^{d+2}\] by construction as well. Passing to the limit in 
\[ -\int_\T \eta h_j\gtrsim  \Vert h_j\Vert_{W^{-1,2}(\T)}^2\] yields the contradiction: $\eta_0$ is not degenerate. Obtaining the limit equation is done exactly in the same way as in Theorem \ref{Th:NonDegeneracyConstrained}.
\end{proof}

\part{Regularity results for the free boundary}\label{Pa:Constrained}

\section*{Plan of the part}
We can finally prove Theorem \ref{Th:Main}.  The structure of the section is as follows:
\begin{enumerate}
\item First, we proceed to classifying the blow-up limits in the two-dimensional case, which we do in Section \ref{Se:Classification}. We then establish the regularity of the free boundary around points of intermediate density, see Section \ref{Se:IntermediateDensity}.
\item Building on this classification we study the regularity of the free boundary in the two-dimensional case. Here, although our approach shares similarities with the analysis of Chanillo, Kenig \& To \cite{zbMATH05505659} we need to substantially modify it in order to use points of intermediate density as our basis. This is done in Section \ref{Se:2DCKT}.
\item Regarding the higher dimensional case, which we recall we only study for penalised problems, we simply use the Federer dimension-reduction principle in the spirit of Weiss \cite{zbMATH01161229}. This is done in Section \ref{Se:Reduction}
\end{enumerate}
Throughout, the points of intermediate density play a crucial role. For further reference, we write their definition here:
\begin{definition}\label{De:IntermediateDensity}
Let $A\subset \T$ be such that $|A|,\, |\T\setminus A|>0$. We say that $x\in \T$ is a point of intermediate density for $A$ if 
\[ 0<\underset{r\to 0}{\lim\inf}D_A(x;r) \le \underset{r\to 0}{\lim\sup}D_A(x;r)<1\] where we recall that 
\[D_A(x;r)=\frac{|A\cap \B(x;r)|}{|\B(x;r)|}.\]
\end{definition}

\section{Classification of  blow-ups at intermediate density points in two dimensions}\label{Se:Classification}
The main result is the following:
\begin{proposition}\label{Pr:Classification}
Let $x_0\in \partial^eE^*$ be a point of intermediate density for $E^*$, in the sense of Definition \ref{De:IntermediateDensity}. Assume that $\n \eta(x_0)=0$.
Then there exists $N_0\in \N^*$ and a family $\{\eta_{0,1},\dots,\eta_{0,N_0}\}$($\subset \mathscr C^{1,\gamma}(\B(0;1))$ for any $\gamma \in (0;1)$) such that any blow-up limit $\eta_0$ satisfies, up to composition with a rotation, 
\[ \eta_0\in \{\eta_{0,1},\dots,\eta_{0,N_0}\}\] and each $\eta_{0,i}$ ($i=1,\dots, N_0$) satisfies:
\begin{enumerate}
\item $\eta_{0,i}$ changes sign in $\B(0;1)$, 
\item $\{\eta_{0,i}=0\}\cap \{ \n \eta_{0,i}=0\}=\{0\}$,
\item $|\n \eta_{0,i}(\cdot)|\gtrsim |\cdot|$.
\end{enumerate}
\end{proposition}
This result should be compared for instance to \cite[Section 5]{zbMATH07826622} where a similar type of classification is obtained (but for a different equation, and where $\eta_0$ is superharmonic).
\begin{proof}[Proof of Proposition \ref{Pr:Classification}]
As always, we assume that $x_0=0$ and, up to working on $-\eta$, we assume that \eqref{Eq:MW6} holds ($f(0)>0$).  We split the proof in several steps.

\emph{Preliminary step: every blow-up limit changes sign. } We first prove that any blow-up limit $\eta_0$ satisfies
\begin{equation}\label{Eq:SignChange}
|\{\eta_0> 0\}|\cdot |\{\eta_0<0\}|>0.
\end{equation}Argue by contradiction and assume that a given blow-up limit has constant sign, say $\eta_0\geq 0$ or $\eta_0\leq 0$. By the intermediate density assumption, we deduce that 
\begin{equation}
|\{\eta_0=0\}|>0
\end{equation} and, with the notations of \eqref{Eq:LimitEquation}, that 
\[E_0^*=\{\eta_0>0\}.\]%Indeed, should that not be the case, we would deduce that either $D_{E^*}(0;r)\underset{r\to 0^+}\rightarrow 0$ or $D_{E^*}(0;r)\underset{r\to 0^+}\rightarrow 1$, in contradiction with the intermediate density assumption. 
We deduce from \eqref{Eq:LimitEquation} that $g(0)=0$, whence 
\[-\Delta \eta_0\geq 0\text{ in }\B(0;1).\]In particular, $\eta_0$ is superharmonic and, from the maximum principle, there holds $\eta_0\equiv 0$, in contradiction with Theorem \ref{Th:NonDegeneracyConstrained}.

Now recall that $\eta_0$ is 2-homogeneous, so that we can write 
\[\eta_0(r,\theta)=r^2\p(\theta),\] where $\p$ solves
\begin{equation}\label{Eq:phi}\begin{cases} -\p''-4\p=f(0)\mathds 1_{\{\p>0\}}-g(0)\mathds 1_{\{\p\leq 0\}}\text{ in }(0;2\pi),\\ \p(0)=\p(2\pi).\end{cases}\end{equation}

We consider $\mathcal I:=\left\{I_s\right\}_{s\in S}$ the collection of disjoint connected sign components of $\p$ (meaning that each $\p$ has constant sign on each $I_s$, and that $\p=0$ on $\partial I_s$). Our goal is to show that $S$ is bounded independently of $\eta_0$. 

Observe that \eqref{Eq:phi} has explicit solutions: if $I_s=(\theta_{0,s};\theta_{1,s}) \in \mathcal I$ is such that $\p>0$ in $I_s$ (in which case we say that $I_s$ is a \emph{positivity component of $\p$}), then there exists a constant $B_s$ such that 
\begin{equation}\label{Eq:ExplicitBlowup+} \p(\theta)=-\frac{f(0)}2 \sin^2(\theta-\theta_{0,s})+B_s \sin\left(2(\theta-\theta_{0,s})\right),\end{equation} while if $\p\leq 0$ in $I_s$ (in which case we say that $I_s$ is a \emph{negativity component of $\p$}) there exists a constant $B_s$ such that 
\begin{equation}\label{Eq:ExplicitBlowup-} \p(\theta)=\frac{g(0)}2 \sin^2(\theta-\theta_{0,s})+B_s \sin\left(2(\theta-\theta_{0,s})\right),\end{equation}

We now  distinguish between three cases, depending on the number of connected components $I_s$ satisfying $|I_s|=\pi$.

\emph{Case 1: $\#\{s \in S,\, |I_s|=\pi\}=0$.} We first show that in that case, for any $s\in S$, 
\begin{equation}\label{Eq:NonZeroDerivativeIs}\p'(\theta_{0,s})\cdot\p'(\theta_{1,s})\neq 0.\end{equation} Without loss of generality, assume that $\p'(\theta_{0,s})=0$. It follows from the explicit expressions \eqref{Eq:ExplicitBlowup+}-\eqref{Eq:ExplicitBlowup-} that $B_s=0$.  If $I_s$ is a positivity component of $\p$, the contradiction follows from the fact that $f(0)>0$. If $I_s$ is a negativity component of $\p$, we deduce that $g(0)>0$ (from Step 1) and, from $\sin^2(\theta_{1,s}-\theta_{0,s})=0$, that either $\theta_{1,s}-\theta_{0,s}=\pi$ or $\theta_{1,s}-\theta_{0,s}=2\pi$. The former is prohibited by $|I_{0,s}|\neq \pi$, and the latter by \eqref{Eq:SignChange}.

We now distinguish two subcases of Case 1:
\begin{enumerate}
\item \underline{If $\sin(2|I_s|)=0$ for some $s\in S$: }without loss of generality, assume that $\sin(2|I_0|)=0$. Let us first show that in this case 
\begin{equation}\label{Eq:LengthI0}
|I_0|=\frac{\pi}2.
\end{equation}To prove \eqref{Eq:LengthI0} first note that the assumption of Case 1 implies
\[ |I_0|=\frac{\pi}2\text{ or }|I_0|=\frac{3\pi}2.\] Now,  it follows from \eqref{Eq:ExplicitBlowup+}--\eqref{Eq:ExplicitBlowup-} and $f(0)>0$ that  $I_0$ is a negativity component of $\p$ and that $g(0)=0$. In this case we obtain
\[\p(\theta)=B_s \sin(2(\theta-\theta_{0,0}))\text{ in }I_0.\] Consequently, if $|I_0|=\frac{3\pi}2$, $\p$ changes sign in $I_0$, in contradiction with the definition of $I_0$. This establishes \eqref{Eq:LengthI0}.Observe now that for any negativity component $I_s$ of $\p$ there holds 
\begin{equation}\label{Eq:LengthIs}|I_s|=\frac{\pi}2.\end{equation} Indeed, since we already know that $g(0)=0$, \eqref{Eq:LengthIs} follows from the explicit expression \eqref{Eq:ExplicitBlowup-}. 

We now study $I_1$, the first positivity component of $\p$. To alleviate notations we let $I_0=(\theta_0;\theta_1), \, I_1=(\theta_1;\theta_2)$. As $\p'(\theta_1)=-2B_0$ from \eqref{Eq:ExplicitBlowup-} and $|I_0|=\frac{\pi}2$, we know that $I_1$ is well defined. We now show that 
\begin{equation}\label{Eq:Equality02}
|I_1|<\frac{\pi}2,\, B_2=B_0.\end{equation} Indeed, should \eqref{Eq:Equality02} hold, it would follow that 
\begin{equation} B_1=B_3\end{equation} and that there are exactly three negativity components $I_0,\, I_2,\, I_4$, all of size $\frac{\pi}2$, and three positivity components $I_1,\, I_3,\, I_5$, all of size $\frac{\pi}6$, which  suffices for our needs. To prove \eqref{Eq:Equality02} first note that, as $\p$ is $\mathscr C^{1}$, there holds, taking the derivatives at $\theta_1$, 
\[-2B_0=2B_1.
\]Furthermore, as $f(0)>0$, it follows from \eqref{Eq:ExplicitBlowup+} and $f(0)>0$ that $|I_1|<\frac{\pi}2$. Iterating and using the continuity of $\p'$ at $\theta_2$, we deduce that $B_0=B_2$, thereby establishing \eqref{Eq:Equality02}. Finally, from $\p(\theta_1)=0$ we deduce that 
\[ B_1=\frac{f(0)\sin^2\left(\frac{\pi}6\right)}{2\sin\left(\frac\pi3\right)}\] and $B_0=-B_1$. In particular, in this subcase, there is exactly one possible blow-up profile up to a rotation.
\item\underline{If $\sin(2|I_s|)\neq 0$ for any $s\in S$: }  without loss of generality, assume that $I_0=(\theta_{0,0};\theta_{1,0})=(\theta_0;\theta_1)$ for the sake of notational convenience is a negativity component for $\p$. By \eqref{Eq:ExplicitBlowup-} and $\sin(2|I_0|)\neq 0$ we obtain $g(0)\neq 0$ so that 
\[ B_0=\frac{g(0)\sin^2(\theta_{1}-\theta_{0})}{2\sin(2(\theta_{1}-\theta_{0}))}=\frac{g(0)}4\tan\left(\theta_{1}-\theta_{0}\right).\] From \eqref{Eq:NonZeroDerivativeIs}, there exists a positivity component $I_1=(\theta_{0,1};\theta_{1,1})=(\theta_1;\theta_2)$. From \eqref{Eq:ExplicitBlowup+} we similarly deduce that $f(0)\neq 0$ and that 
\[ B_1=-\frac{f(0)}4\tan\left(\theta_{2}-\theta_{1}\right).\] Matching the derivative of $\p$ at $\theta_{1}$, we obtain 
\[ \frac{\tan\left(\theta_{2}-\theta_{1}\right)}{\tan\left(\theta_{1}-\theta_{0}\right)}=-\frac{g(0)}{f(0)}.\] If $\theta_{2}=2\pi$, we conclude that there are exactly two phases. Else, iterating the procedure we construct a sequence of alternating sign components $I_k=(\theta_{0,k};\theta_{1,k})$ with $\theta_{1,k}=\theta_{0,k+1}$ and such that
\[
\begin{cases}
 \frac{\tan\left(\theta_{1,k}-\theta_{0,k}\right)}{\tan\left(\theta_{1,k-1}-\theta_{0,k-1}\right)}=-\frac{g(0)}{f(0)}\text{ if $k$ is odd }
 \\ \frac{\tan\left(\theta_{1,k}-\theta_{0,k}\right)}{\tan\left(\theta_{1,k-1}-\theta_{0,k-1}\right)}=-\frac{f(0)}{g(0)}\text{ if $k$ is even }.
\end{cases}
\]We thus deduce that for any odd $k$ there holds $|I_k|=|I_1|$ and that, for any even $k$, $|I_k|=|I_0|$.  We refer to Fig. \ref{Fig:BU}. Consequently, there are only finitely many sign components. Finally, observe that the number $N$ of sign components is necessarily even, and that 
\begin{equation}\label{Eq:Ops} 2\pi=\frac{N}2\left(|I_0|+|I_1|\right)=\frac{N}2 \left(|I_0|+\tan^{-1}\left(-\frac{g(0)}{f(0)}|I_0|\right)\right).\end{equation} Consequently, $|I_0|$ can only take four different values, and we thus have, for each finite number $N$ of connected components, at most four different solutions.

We now have to prove that there are at most $N_0$ sign components for some $N_0\in \N$. This however follows once more from \eqref{Eq:ExplicitBlowup+}-\eqref{Eq:ExplicitBlowup-} and the previous analysis: indeed, observe that, from the previous argument, if a solution has exactly $N$ sign components, then 
\[ B_0=\frac{g(0)}{4}\tan(\theta_{1,0}-\theta_{0,0}),\, B_1=-\frac{g(0)}{4}\tan(\theta_{1,0}-\theta_{0,0}).\] If there were a sequence $\{\p_N\}_{n\in \N}$ of admissible solutions such that $\p_N$ has at least $N$ connected components, it follows that, letting $I_{0,N}$ denote the length of (one) negativity component of $\p_N$ and $I_{1,N}$ the length of (one) positivity component of $\p_N$, \eqref{Eq:Ops} entails $|I_{0,N}|+|I_{1,N}|\underset{N\to \infty}\rightarrow 0$, and \eqref{Eq:ExplicitBlowup+}--\eqref{Eq:ExplicitBlowup-} implies 
\[ \Vert \p_N\Vert_{L^\infty(\B(0;1))}\underset{N\to \infty}\rightarrow 0.\] By a diagonal argument, this implies that $\bar \eta\equiv 0$ is a blow-up at $x_0=0$, in contradiction with Theorem \ref{Th:NonDegeneracyConstrained}.
\end{enumerate}
This concludes the analysis of the first case.
\begin{figure}\begin{tikzpicture}[rotate=90, scale=1.5]

    % Coordinates
    \coordinate (O) at (0,0);
    
    % Angle definitions: Large is double the Small
    \def\smallangle{22}
    \def\largeangle{44}
    \def\startangle{42}

    % Ray positions
    \coordinate (R1) at (\startangle:2.5);                            
    \coordinate (R2) at ({\startangle-\largeangle}:2.5);             
    \coordinate (R3) at ({\startangle-\largeangle-\smallangle}:2.5);  
    \coordinate (R4) at ({\startangle-2*\largeangle-\smallangle}:2.5);
    \coordinate (R5) at ({\startangle-2*\largeangle-2*\smallangle}:2.5);

    % Draw outer boundary and rays
    \draw[thick] (O) circle (2.5cm);
    \foreach \p in {R1,R2,R3,R4,R5} {\draw[thick] (O) -- (\p);}

    % --- Angle Marks with different radii ---
    
    % Sector 1 (Large, eta > 0): Double Arc at LARGER radius
    \draw ($(O)!0.8cm!(R2)$) arc ({\startangle-\largeangle}:\startangle:0.8cm);
    \draw ($(O)!0.88cm!(R2)$) arc ({\startangle-\largeangle}:\startangle:0.88cm);

    % Sector 2 (Small, eta < 0): Single Arc at SMALLER radius
    \draw ($(O)!0.55cm!(R3)$) arc ({\startangle-\largeangle-\smallangle}:{\startangle-\largeangle}:0.55cm);

    % Sector 3 (Large, eta > 0): Double Arc at LARGER radius
    \draw ($(O)!0.8cm!(R4)$) arc ({\startangle-2*\largeangle-\smallangle}:{\startangle-\largeangle-\smallangle}:0.8cm);
    \draw ($(O)!0.88cm!(R4)$) arc ({\startangle-2*\largeangle-\smallangle}:{\startangle-\largeangle-\smallangle}:0.88cm);

    % Sector 4 (Small, eta < 0): Single Arc at SMALLER radius
    \draw ($(O)!0.55cm!(R5)$) arc ({\startangle-2*\largeangle-2*\smallangle}:{\startangle-2*\largeangle-\smallangle}:0.55cm);

    % Labels
    \filldraw (O) circle (1.2pt) node[ below left] {$x_0$};
    
    \node at ({\startangle-\largeangle/2}:1.7) {$\eta > 0$};
    \node at ({\startangle-\largeangle-\smallangle/2}:1.7) {$\eta < 0$};
    \node at ({\startangle-\largeangle-\smallangle-\largeangle/2}:1.7) {$\eta > 0$};
    \node at ({\startangle-2*\largeangle-\smallangle-\smallangle/2}:1.7) {$\eta < 0$};

\end{tikzpicture}

\caption{Illustration of the repeating patterns.}\label{Fig:BU}
\end{figure}
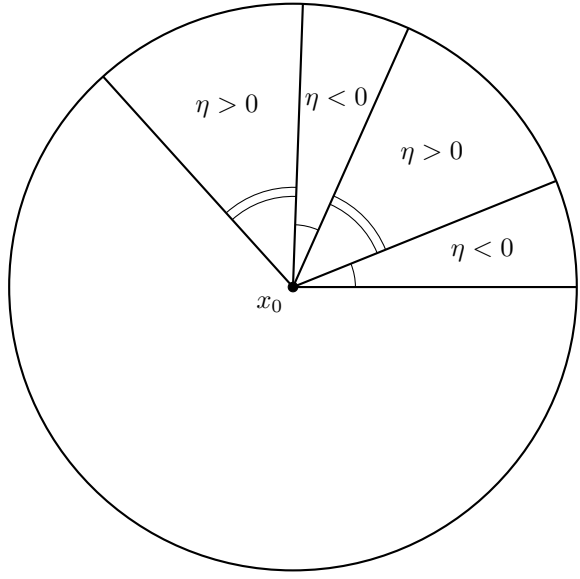

\emph{Case 2: $\#\{s \in S,\, |I_s|=\pi\}=1$.} We show that this case can not occur. Assume without loss of generality that $|I_0|=\pi$. From \eqref{Eq:ExplicitBlowup+}--\eqref{Eq:ExplicitBlowup-} we obtain 
\[ \p'(0)=2B_0=\p'(\pi)\] whence, as $I_0$ is a sign component of $\p$, we deduce that 
\begin{equation}\label{EqK} \p'(0)=\p'(\pi)=0.\end{equation} In particular, $B_0=0$. As $I_0$ is the only sign component of length $\pi$, we can use \eqref{Eq:SignChange} and apply the analysis of Case 1 to $\mathbb S(0;1)\setminus I_0$ to obtain a finite  sequence of sign components $\{I_k\}_{1\leq k\leq N}$ such that $I_k=(\theta_{0,k};\theta_{1,k})$ with 
\[\pi=\theta_{0,1}<\theta_{1,1}<\dots<\theta_{1,N}=2\pi,\] and the analysis of Case 1 entails 
\[ \p'\neq 0\text{ on }\cup_{k=1}^N \partial I_k.\] In particular, $\p'(0)\neq 0$, in contradiction with \eqref{EqK}.

\emph{Case 3:  $\#\{s \in S,\, |I_s|=\pi\}=2 $.} We show that this case can not  occur. Assume without loss of generality that $I_0=(0;\pi)$ and  $I_1=(\pi;2\pi)$, and that $I_0$ is a negativity component of $\p$, while $I_1$ is a positivity component of $\p$. In the same way that we derived \eqref{EqK} we have, in this case,
\begin{equation}\label{Eq:Pasta}
\p'(0)=\p'(\pi)=2B_0=0.
\end{equation}
It follows that $B_1=0$ as well, whence we deduce that 
\[ \p(\theta)=\frac{g(0)}2 \sin^2(\theta)\mathds 1_{(0;\pi)}-\frac{f(0)}2\sin^2(\theta)\mathds 1_{(\pi;2\pi)}.
\] This is in contradiction with the fact that $I_1$ is a positivity component of $\p$, and concludes the analysis of this case.
\end{proof}

\section{Regularity of the free boundary around points of intermediate density}\label{Se:IntermediateDensity}

\subsection{Density of intermediate density points}
As we explained, we use intermediate density points in the sense of Definition \ref{De:IntermediateDensity} as the basis for our regularity analysis. The measure theoretical boundary $\partial^e A$ is defined analogously to \eqref{De:Boundary}. The following proposition is  due to Pegon \cite[Lemma 4.4.2]{pegon:tel-01661457}:
\begin{proposition}\label{Pr:DensityIntermediateDensity}
Let $A\subset \T$ be such that $|A|,|\T\setminus A|>0$. Let $x_0\in \partial^eA$. For any $\e>0$, there exists $x_1\in \B(x_0;\e)$  of intermediate density for $A$ in the sense of Definition \ref{De:IntermediateDensity}.
\end{proposition}We give the proof for the sake of completeness.
\begin{proof}[Proof of Proposition \ref{Pr:DensityIntermediateDensity}] We  reproduce the main arguments of  \cite[Proof of Lemma 4.4.2]{pegon:tel-01661457}, and we recall that one of the ideas of \cite{pegon:tel-01661457} is to compute the density of points by considering squares rather than balls. Namely, for any $x\in \T,\, r>0$, we let $Q(x;r)$ be the cube centred at $x$ with side length $r$, and we set, for any measurable set $\O$, 
\[ D_{\O,\square}(x;r):=\frac{|Q(x;r)\cap \O|}{|Q(x;r)|}.\]
Let $\e>0$. As $|A\cap \B(x_0;\e)|\cdot|(\T\setminus A)\cap \B(x_0;\e)|>0$ we let $y_0,\, y_1\in \B(x_0;\e)$ be Lebesgue points for $A$ and $A^c$ respectively. Let $\e_1>0$ be such that 
\[ \B(y_0;\e_1),\, \B(y_1,\e_1)\subset \B(x_0;\e),\, D_{A,\square}(y_0;\e_1)>\frac12,\, D_{A^c,\square}(y_1;\e_1)>\frac12.\] The goal is to construct a Cauchy sequence $\{x_k\}_{k\in \N}\subset \B(x_0;\e_1)$ such that 
\begin{equation}\label{Eq:Densityk}
\forall k \in \N,\, D_{A,\square}\left(x_k;2^{-k}\e_1\right),D_{A^c,\square}\left(x_k;2^{-k}\e_1\right)=\frac12.\end{equation} This sequence is constructed inductively:
by the intermediate value theorem, there exists $t_1\in (0;1)$ such that 
\[ D_{A,\square}\left((1-t_1)y_0+t_1y_1;\e_1\right)=\frac12.\] Set $x_1:=(1-t_1)y_0+t_1y_1$. To construct $x_2$,  divide $Q(x_1;\e_1)$ into $2^d$ cubes of side length $\e_2:=\frac{\e_1}2$ with centers $(y_{1,1},\dots,y_{1,2^d})$. If for some $j$ there holds $D_{A,\square}(y_j;\e_2)=\frac12$, we set $x_2=y_j$. Else, as 
\[ D_{A,\square}(x_0;\e_1)=\frac12=\frac{|Q(0;\e_2)|}{|Q(x_0;\e_1)|}\sum_{j=1}^{2^d} D_{A,\square}(y_j;\e_2)\] there exist two cubes $Q(y_i;\e_2),\, Q(y_j;\e_2)$ such that $D_{A,\square}(y_i;\e_2)>\frac12$, $D_{A^c,\square}(y_j;\e_2)>\frac12$ and we repeat the initial argument, yielding a satisfactory $x_2$. This gives a sequence $\{x_k\}_{k\in \N}$ in the conditions of \eqref{Eq:Densityk}, which thus converges to some $x_\infty$. Observe that $x_\infty \in \cap_{k\in \N} Q(x_k;\e_0 2^{-k})$. Now, let $r\in (0;\e_0)$ and $k\in \N$ such that $r\in (2\sqrt{d}\e2^{-(k+1)};2\sqrt{d}\e2^{-k})$. As $x_\infty \in\cap_{k\in \N} Q(x_k;\e_0 2^{-k})$, $Q(x_{k+1};\e_02^{-(k+1)})\subset \B(x_\infty;r)$ so that 
\[D_A(x_\infty;r)\geq \frac12\frac{|Q(x_{k+1};\e_02^{-(k+1)})|}{|\B(x_\infty;r)|}\gtrsim 1,\] and the same holds for $A^c$, thereby concluding the proof. \end{proof}

\subsection{Regularity at intermediate density points}\label{section:RegularityAtIntermediateDensityPoints2D}
We are now ready to prove the following regularity result around points of intermediate density:
\begin{proposition}\label{Pr:RegularityIntermediateDensity}
Let $d = 2$ and $x_0$ be a point of intermediate density for $E^*$ in the sense of Definition \ref{De:IntermediateDensity}. Locally around $x_0$, $\nabla \eta(x_0) \neq 0$ and $\partial^eE^*$ is a $\mathscr C^{2,\gamma}$ curve.
\end{proposition}

\begin{proof}[Proof of Proposition \ref{Pr:RegularityIntermediateDensity}] We show that at a point of intermediate density, we have $\n \eta(x_0)\neq 0$. To do this, we follow the proof of Theorem \ref{Th:C11Regularity}: from Proposition \ref{Pr:Classification} we only have a finite number of admissible blow-ups (up to rotations) and, furthermore, with the notations of Proposition \ref{Pr:Classification},
\[ \min_{i\in \{1,\dots,N_0\}} \min_{\B\left(0;1\right)\setminus \B\left(0;\frac12\right)} |\n \eta_{0,i}|>0.\] From these two information we obtain in the same way that \eqref{Eq:Deceived} that 
\[ \mathscr H^1\left(\{\eta=0\}\cap\{|\n \eta|\neq 0\}\right)\gtrsim \frac1{2^k}.\] The conclusion follows by Theorem \ref{Th:ASCConstrained}.
\end{proof}

\section{Regularity in the two-dimensional case}\label{Se:2DCKT}
\subsection{Main result}
This section is dedicated to the proof of Theorem \ref{Th:Main} (1). The proof is based on the following ingredients:
\begin{enumerate}
    \item Non-degeneracy of the gradient of $\eta$ at intermediate density points (see Proposition \ref{Pr:RegularityIntermediateDensity} above).

    \item The smoothness of the connected components of the free boundary $\partial E^{\ast}$ where at least one point satisfies $\nabla \eta \neq 0$ (see Lemma \ref{Le:NonDegenerateConnectedComponentsAreSmooth} below).

    \item The impossibility of the accumulation of smooth connected components to singular points (see Proposition \ref{Pr:NonAccumulationToSingularPoints} below).

    \item An inductive argument, to iteratively count the connected components of the measure theoretic free boundary $\partial^e E^{\ast}$ (defined in \eqref{De:Boundary}).
    
    \item Finally, in Section \ref{Se:Why2D}, we explain heuristically why the arguments presented here can not be adapted to the higher dimensional setting.
\end{enumerate}
Although some points are similar to the analysis of Chanillo, Kenig \& To \cite{zbMATH05505659}, a significant difference is that we do not have non-degeneracy at every free boundary points, meaning that we have to rely on the non-degeneracy at intermediate density points. More precisely, we transfer the regularity from points with a nontrivial density condition (namely a geometric condition at the blown-up limit), to all points in the measure theoretic free boundary (where the condition holds at the original scale).

We begin with a lemma, to be compared with \cite[Lemma 5.6]{zbMATH05505659}, which allows to pass from the local analysis of the free boundary near density points from Section \ref{section:RegularityAtIntermediateDensityPoints2D}, to a global one.
    \begin{lemma}[]\label{Le:NonDegenerateConnectedComponentsAreSmooth}
        Let $d=2$ and $x_0 \in \partial^e E^{\ast}$ be such that $\n \eta(x_0)\neq 0$. Then, there exists a unique curve $\Gamma \subset \T$ such that:
        \begin{itemize}
            \item[(i)] $x_0 \in \Gamma$ and $\Gamma \subseteq \partial E^{\ast}$,
        
            \item[(ii)] $\Gamma$ is simple, closed and $\mathscr C^{2,\gamma}$ for any $\gamma\in (0;1)$,

            \item[(iii)] $\nabla \eta\neq 0$ on  $\Gamma$. In particular, $\Gamma \subseteq \partial^e E^{\ast}$
        \end{itemize}
    \end{lemma}
    \begin{proof}[Proof of Lemma \ref{Le:NonDegenerateConnectedComponentsAreSmooth}]
        From Theorem \ref{Th:Hodograph}, we know that, for some $\e>0$, there exists a unique $\mathscr C^{2,\gamma}$ curve passing through $x_0$ such that $\partial^e E^*\cap \B(x_0;\e)=\Gamma$. Let $s$ be the arc length parametrisation of $\Gamma$ starting from $x_0$. Since $\Gamma$ can be extended uniquely  as long as  $\left|\nabla \eta\left( \Gamma(s) \right)\right|>0$, it suffices to show that 
        \begin{itemize}
            \item[(a)] $ L := \sup\left\{s > 0 :\,\left| \nabla \eta\left( \Gamma(s) \right)\right|>0 \text{ and } \Gamma(s) \neq x_0 \right\} < +\infty$,

            \item[(b)] $\Gamma(L) = x_0$.
        \end{itemize}
        We first show that $L<\infty$. Assume by contradiction that $L=+\infty$. Let $\omega_\Gamma$ be  the set of closure points of $\{\Gamma(s)\}_{s\to \infty}$. $\omega_\Gamma$ is compact. We claim that, for any $\bar x \in \omega_\Gamma$, there exists $\bar \e>0$ such that $\mathscr H^1(\B(\bar x;\bar \e)\cap \Gamma)<+\infty$. Indeed, should this not be the case, there exists $\bar x$ such that $\bar x$ is in the closure of $\Gamma$ and $ \mathscr H^1(\B(\bar x;\bar \e)\cap \Gamma)=+\infty$ for any $\e>0$. In particular, as $\Gamma$ is continuous, Theorem \ref{Th:ASCConstrained} applies and $\n \eta(\bar x)\neq 0$, whence $\Gamma$ has finite length locally around $\bar x$, a contradiction. As $\omega_\Gamma$ is compact, we can cover $\omega_\Gamma$ with finitely many such balls $\B(\bar x;\bar \e)$. Consequently, $\Gamma$ has finite length, in contradiction with $L=+\infty$. Thus, $L<+\infty$.  Now if $\Gamma(L)\neq x_0$, then there exists $\bar x$ in the closure of $\Gamma$ such that, by the definition of $L$, $\n \eta(\bar x)=0$. However, as $\Gamma$ is $\mathscr C^{2,\gamma}$, Theorem \ref{Th:ASCConstrained} implies $\n \eta(\bar x)\neq 0$, a contradiction. The fact that $\n \eta\neq 0$ on $\Gamma$ similarly follows from Theorem \ref{Th:ASCConstrained}.
         \end{proof}

    Now our aim is to bound the number of regular curves $\Gamma$ appearing in Lemma \ref{Le:NonDegenerateConnectedComponentsAreSmooth}. To this aim, we introduce the following notion:
    \begin{definition}\label{def:OuterInnerBoundaryOfOptimalSet}
        Let $d = 2$, $E^{\ast}$ the optimal set and $\eta_{m^{\ast}}$ the associated solution. We say that a $C^1$ simple and closed curve $\Gamma$ is a nondegenerate boundary if
        \begin{equation*}
            \Gamma \subset \partial^e E^{\ast} \quad \text{and} \quad \nabla \eta \neq 0 \text{ on } \Gamma.
        \end{equation*}
        Moreover, we say that $\Gamma$ is
        \begin{itemize} 
        
            \item an outer boundary for $E^{\ast}$ if 
        \begin{equation*}
            \partial_{\nu_{\Gamma}} \eta_{m^{\ast}} < 0,  \text{ on } \Gamma,
        \end{equation*}

        \item an inner boundary for $E^{\ast}$ if 
        \begin{equation*}
            \partial_{\nu_{\Gamma}} \eta_{m^{\ast}} > 0,  \text{ on } \Gamma.
        \end{equation*}
        \end{itemize}

        where $\text{int}(\Gamma)$ denotes the interior of the curve $\Gamma$ in the sense of Jordan and $\nu_{\Gamma}$ the outer normal to $\text{int}(\Gamma)$.

          \end{definition}
    In the next proposition, to be compared with \cite[Lemmas 6.2, 6.3]{zbMATH05505659}, we show that only a finite number of (nondegenerate) outer and inner boundaries can occur in $\Omega$. 
    \begin{proposition}\label{Pr:NonAccumulationToSingularPoints}
       Let $d = 2$ and consider a sequence $\{ \Gamma_i\}_{i \in \N}$ of nondegenerate outer and inner boundaries, in the sense of Definition \ref{def:OuterInnerBoundaryOfOptimalSet}. Then, there exists $\N \ni N_0 = N_0(\eta)$ such that
       \begin{equation*}
           \#\{ \Gamma_i\}_{i \in \N} \le N_0.
       \end{equation*}
    \end{proposition}
    \begin{remark}
        The bound $N_0$ on the number of nondegenerate inner/outer boundaries is independent of 
        \begin{equation*}
            \inf_{i \in \N} \inf_{x \in \Gamma_i} \left\vert \nabla \eta(x) \right\vert.
        \end{equation*}
        More precisely, Proposition \ref{Pr:NonAccumulationToSingularPoints} remains valid even if
        \begin{equation*}
            \inf_{i \in \N} \inf_{x \in \Gamma_i} \left\vert \nabla \eta(x) \right\vert = 0.
        \end{equation*}
    \end{remark}
    \begin{proof}[Proof of Proposition \ref{Pr:NonAccumulationToSingularPoints}]
        Without loss of generality, and up to following the same argument on  $\left( E^{\ast} \right)^c$ and $-\eta$, we can  assume that $\{ \Gamma_i\}_{i \in \N}$ are outer boundaries of $E^*$. For any $i \in \N$, let $x_0 \in \text{int}(\Gamma_i)$ be such that
        \begin{equation*}
            \eta(x_0) = \max_{x \in \text{int}(\Gamma_i) \cup \Gamma_i} \eta(x).
        \end{equation*}
        The point $x_0$ satisfies the above property $x_0 \in \text{int}(\Gamma_i)$, since by assumption $\nabla \eta \neq 0$ on $\Gamma_i$. Denoting
        \begin{equation*}
            D_i := \sup_{x \in \Gamma_i} \left\vert x - x_0 \right\vert,
        \end{equation*}
        it follows from Theorem \ref{Th:C11Regularity} that
        \begin{align*}
            \begin{split}
                \left\vert \nabla \eta(x) \right\vert & = \left\vert \nabla \eta(x) - \nabla \eta(x_0) \right\vert \\
                & \le \| \nabla \eta \|_{\mathscr C^{1,1}(\mathbb{T}^2)} \vert x \vert \\
                & \le \| \nabla \eta \|_{\mathscr C^{1,1}(\mathbb{T}^2)} D_i.
            \end{split}
        \end{align*}
        Furthermore, since we are working in two dimensions, we have 
        \[ |\Gamma_i|\gtrsim D_i.\]
Consequently,        %
        \begin{align*}
            \int_{\Gamma_i} \frac{1}{\left\vert \nabla \eta \right\vert} & \ge \frac{1}{D_i \| \nabla \eta \|_{\mathscr C^{1,1}(\mathbb{T}^2)}}  \left\vert \Gamma_i \right\vert \\
            & \ge \frac{C(d)}{\| \nabla \eta \|_{\mathscr C^{1,1}(\mathbb{T}^2)}}.
        \end{align*}
        The contradiction follows from the second order optimality conditions (Lemma \ref{Le:CKT}): suppose by contradiction that $\#\{ \Gamma_i\}_{i \in \N} = +\infty$. Let $\{ x_i \}_{i \in \N}$ a sequence of points such that $x_i \in \Gamma_i$ for all $i \in \N$ and let $x_{\infty}$ be a closure point of $\{ x_i \}_{i \in \N}$. Clearly $x_{\infty} \in \partial E^{\ast}$. Moreover, since the curves $\Gamma_i$ are mutually disjoint and for all $r > 0$ there exists $i = i(r) \in \N$ such that $\Gamma_i \cap \B(x_{\infty};r) \neq 0$ and $i(r) \to +\infty$ as $r \to 0^+$, we obtain from Lemma \ref{Le:NonDegenerateConnectedComponentsAreSmooth} that $\nabla \eta(x_{\infty}) = 0$. Let us now show the existence of a vanishing sequence of radii $\{ r_k \}_{k \in \N}$ such that (up to a non relabeled subsequence) $\Gamma_k \subset \B(x_{\infty};r_k) \setminus \B(x_{\infty};r_{k+1})$, for all $k \in \N$. As a consequence, the thesis will follow directly from Lemma \ref{Le:CKT}. Since by Lemma \ref{Le:NonDegenerateConnectedComponentsAreSmooth} we have $x_{\infty} \notin \cup_i \Gamma_i$, it follows that for any $i\in \N$ there holds $\mathrm{dist}\left( x_{\infty}, \Gamma_i \right) > 0$. In particular, it suffices to prove that 
        \begin{equation}\label{Eq:HpBoundNumberConnectedComponents}
            \text{For any $r > 0$ there exists $i \in \N$ such that $\Gamma_i \subset \B(x_{\infty};r)$.}
        \end{equation}
        If this is the case, we can choose the radii $\{ r_k \}_{k \in \N}$ iteratively as follows:
        \begin{itemize}
            \item[1.] $r_0 = \mathrm{dist}\left(x_{\infty}, \Gamma_0  \right) + \mathrm{diam}\left( \Gamma_0 \right)$,

            \item[2.] $r_1 = \mathrm{dist}\left(x_{\infty}, \Gamma_0  \right)$,

            \item[3.] for all $k \ge 1$, $r_{k+1} = \mathrm{dist}\left(x_{\infty}, \Gamma_k  \right)$, where $\Gamma_k$ satisfies $\Gamma_k \subset \B(x_{\infty};r_k)$.
        \end{itemize}
        We are only left to prove \eqref{Eq:HpBoundNumberConnectedComponents}. Suppose by contradiction that \eqref{Eq:HpBoundNumberConnectedComponents} is false. Hence, there exists $r>0$ such that $\Gamma_i \cap \partial \B(x_{\infty};r) \neq 0$ for $i \in \N$ sufficiently large. For all $j \in \N$ sufficiently large (up to a non relabeled subsequence) consider the curves $\gamma_j = \Gamma_j \cap \left( \B\left(x_{\infty}, 2^{-j} \right) \setminus \B\left(x_{\infty}, 2^{-(j+1)} \right) \right)$. Since the curves $\Gamma_i$ are continuous and $x_{\infty} \notin \cup_i \Gamma_i$ is an accumulation point for $\{x_i\}_{i\in \N}$, each curve $\gamma_j$ admits a connected component $\widetilde{\gamma}_j$ with endpoints belonging to $\partial \B\left(x_{\infty}, 2^{-j} \right)$ and $\partial \B\left(x_{\infty}, 2^{-(j+1)} \right)$. In particular, $\mathscr H^1(\B(\bar x;\bar \e)\cap \widetilde{\gamma}_j) \gtrsim 2^{-j}$. Theorem \ref{Th:ASCConstrained} is now in contradiction with $\nabla \eta(x_{\infty}) = 0$. This concludes the proof.
    \end{proof}
We can now prove the main theorem.

    \begin{proof}[Proof of Theorem \ref{Th:Main} (1)]
        We proceed inductively: consider the sets $\mathcal{G} = \T$ and define $\mathcal{C} = \emptyset$. If $\partial^e E^{\ast} \neq \emptyset$, then by Proposition \ref{Pr:DensityIntermediateDensity} there exists a point of intermediate density $x_0$ of $E^{\ast}$ in $\mathcal{G}$. By Proposition \ref{Pr:RegularityIntermediateDensity} $\nabla \eta(x_0) \neq 0$ so that, thanks to Lemma \ref{Le:NonDegenerateConnectedComponentsAreSmooth}, there exists a (unique) $\mathscr C^{2,\gamma}$, simple and closed curve $\Gamma_{x_0}$ (with $\nabla \eta \neq 0$ on $\Gamma_{x_0}$) such that $x_0 \in \Gamma_{x_0}$. We set $\mathcal{C} \rightarrow \mathcal{C} \cup \Gamma_{x_0}$ and $\mathcal{G} \rightarrow \mathcal{G} \setminus \Gamma_{x_0}$. By Proposition \ref{Pr:NonAccumulationToSingularPoints} we can iterate the above argument at most $N_0$ times. We are only left to prove that after $N_0$ steps, $\partial^e E^{\ast} \cap \mathcal{G} = \emptyset$ or, equivalently, that $\partial^eE^*=\mathcal C$. By contradiction, assume that there exists $x \in \partial^e E^{\ast}$ such that $x \in \mathcal{G}$. By the definition of measure theoretic free boundary \eqref{De:Boundary} and since  $\mathcal C$ is a finite union of $N_0$ closed curves, there exists $r>0$ such that:
        \begin{itemize}
            \item[(a)] $|E^*\cap \B(x;r)| > 0$ and  $|\left( \mathbb{T}^2 \setminus E^* \right) \B(x;r)|>0$,

            \item[(b)] $\B(x;r) \subset \mathcal{G}$.
        \end{itemize}
        Hence, we can apply once again Proposition \ref{Pr:DensityIntermediateDensity}, Proposition \ref{Pr:RegularityIntermediateDensity} and Lemma \ref{Le:NonDegenerateConnectedComponentsAreSmooth}, deducing the existence of a further regular curve $\Gamma_x$ such that $\Gamma_x \notin \mathcal{C}$, contradicting the fact that the iterative procedure concludes after $N_0$ steps.    \end{proof}

Of course, Theorem \ref{Th:ReformPenalised} follows, in the two dimensional case, from Theorem \ref{Th:ReformConstrained}.

\subsection{Why is 2d so particular?}\label{Se:Why2D}
We briefly comment on the main difference between the two and $d$ dimensional settings, $d>2$, namely, that Proposition \ref{Pr:NonAccumulationToSingularPoints} has no chance of following from second-order optimality conditions. This can be seen from the point of view of torsion functions. Assume for the sake of simplicity that the second-order optimality conditions entail 
\[ \int_{\{\eta=0\}}\frac{1}{|\n \eta|}<+\infty\] and that 
\[ f \equiv 1,\, g=0.\] Assume further that, around a critical free boundary point $x_0$ $E^*$ contains an accumulation of small connected components (see Fig. \ref{Fig:Ill2d}) which, for the sake of simplicity, we take to be balls, say $\sqcup_{k\in \N} \B(y_k;\delta_k)\subset E^*\cap B(x_0;1)$, where each $\B(y_k;\delta_k)$ is a connected component of $E^*$.  Given the equation satisfied by $\eta$, $\eta|_{‘\B(y_k;\delta_k)}$ is the torsion function of each $\B(y_k;\delta_k)$. Letting $w_\delta$ denote the torsion function on $\B(0;\delta)$, we obtain 
\[ \sum_{k=0}^\infty \int_{\partial \B(0;\delta)} \frac{1}{|\n w_\delta|}<+\infty.\] However, $w_\delta$ has an explicit expression leading (up to dimensional constant $C(d)$) to 
\[  \int_{\partial \B(0;\delta)} \frac{1}{|\n w_\delta|}\sim C(d)\delta^{d-2},\] whence we deduce 
\[ \sum_{k=0}^\infty \delta_k^{d-2}<\infty.\] In two dimensions, this would allow to conclude that there exists only a finite number of connected components. In higher dimensions, this does not allow to conclude.

\begin{figure}
\begin{tikzpicture}[rotate=90]

    % Define a style for the blobs with a dashed pattern fill
    \tikzset{
        blob/.style={
            thick, 
            draw=black, 
            pattern={Lines[angle=45, distance=2pt, line width=0.5pt]},
            pattern color=black!70
        }
    }

    % Central point x_0
    \filldraw (0,0) circle (1.5pt) node[left=3pt] {$x_0$};

    % Small blobs near x_0 (Left cluster)
    \draw[blob] plot [smooth cycle, tension=1] coordinates {(-0.5,0.5) (-0.3,0.7) (-0.6,0.8)};
    \draw[blob] plot [smooth cycle, tension=1] coordinates {(-0.8,-0.2) (-0.6, -0.4) (-0.9,-0.5)};
    \draw[blob] plot [smooth cycle, tension=0.8] coordinates {(-0.2,-0.8) (0.1, -1) (-0.1,-1.2)};

    % Mid-range blobs
    \draw[blob] plot [smooth cycle, tension=0.7] coordinates {(1, 0.5) (1.4, 0.8) (1.2, 1.1) (0.8, 0.9)};
    \draw[blob] plot [smooth cycle, tension=0.8] coordinates {(0.8, -0.5) (1.3, -0.6) (1.1, -0.9)};
    \draw[blob] plot [smooth cycle, tension=0.9] coordinates {(1.5, -1.5) (1.8, -1.8) (1.4, -2)};

    % Larger blobs on the right
    % Top right blob
    \draw[blob] plot [smooth cycle, tension=0.7] coordinates {(2.5, 1.5) (3.2, 1.8) (3, 2.5) (2.3, 2.2)};
    
    % Middle right blob
    \draw[blob] plot [smooth cycle, tension=0.7] coordinates {(2.2, 0) (2.8, 0.2) (2.6, -0.3)};
    
    % Large bottom right blob with label C_k
    \draw[blob] plot [smooth cycle, tension=0.6] coordinates {(2.5, -1.5) (3.5, -1.2) (3.8, -2.2) (2.8, -2.5) (2.3, -2)};
    \node[fill=white, inner sep=1.5pt, rounded corners=1pt, scale=0.7] at (3.1, -1.8) {$\{\eta>0\}$};

    % A few extra stray small ones for texture
    \draw[blob] plot [smooth cycle, tension=1] coordinates {(-1.2, 1) (-1, 1.2) (-1.3, 1.3)};
    \draw[blob] plot [smooth cycle, tension=1] coordinates {(1.5, 2.8) (1.8, 3) (1.6, 3.2)};

\end{tikzpicture}\caption{In two dimensions, such an accumulation of smaller and smaller components can not happen.}\label{Fig:Ill2d}\end{figure}
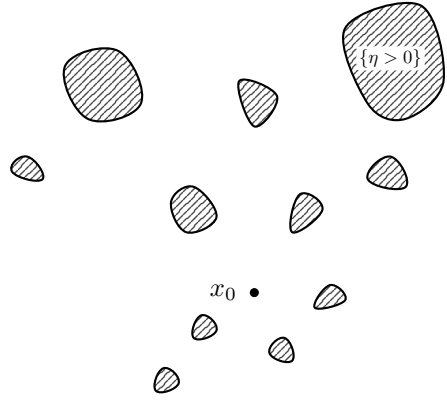

\section{Regularity in higher dimensions for penalised problems}\label{Se:Reduction}
This section is devoted to the proof of the only part of Theorem \ref{Th:Main} that remains to be dealt with, the higher dimensional regularity for penalised problems. We follow the approach pioneered by Federer \cite{zbMATH03309325}.
\begin{proof}[Proof of Theorem \ref{Th:Main}, point (2)-(b)]
We follow the Federer dimension reduction principle, which requires a bit more information than previously derived regarding blow-ups.  For the time being, let $x_0$ be a critical free boundary point in $\T$, $d\geq 3$, and let $\eta_0$ be a blow-up at $x_0$. If the Weiss functional $\Psi$ satisfies $\Psi(0^+)=-\infty$, the result follows from the dimension of critical points for homogeneous harmonic functions.  Else, we know from Theorem \ref{Th:NonDegeneracyPenalised}  and its proof that the following holds: letting 
\[ \mathcal M_0(\B(0;1)):=\left\{m\in L^\infty(\B(0;1)):\, 0\leq m\leq 1\text{ a.e.}\right\}\] and $m^*:=\mathds 1_{E_0^*}$ we have, for any $m\in \mathcal M_0(\B(0.1))$, 
\begin{equation}\label{Eq:DRShape}\int_{\B(0;1)}\eta_0(m-m^*)+\alpha \Vert m-m^*\Vert_{W^{-1,2}(\B(0;1))}^2\leq 0.\end{equation}
Finally, $\eta_0$ is 2-homogeneous from Theorem \ref{Th:MW1}. It follows from the same proof as Proposition \ref{Pr:LocalMinimalityShape2} that $\eta_0$ satisfies strong second-order shape optimality conditions in the sense that \eqref{Eq:SOSD1} is satisfied.

We can now proceed with dimension reduction: arguing by contradiction and assuming that $\mathscr H^{d-2+\e}(\{\n \eta=0\}\cap \{\eta=0\})\in (0;+\infty]$ for some $\e>0$. We let $x_0\in \{\eta=0\}\cap \{\n \eta=0\}$ be such that 
\begin{equation}\label{Eq:d-2}\underset{r\to 0^+}{\lim\sup}\frac{ \mathscr H^{d-2+\e}\left(\{\n \eta=0\}\cap \{\eta=0\}\cap \B(x_0;r)\right)}{r^{d-2+\e}}>0\end{equation} and we fix a sequence $\{r_k\}_{k\in \N}$ converging to 0 such that this $\lim\sup$ becomes a limit. We let $\eta_0$ be a blow-up limit along this sequence (which is non-degenerate), and $x_{00}\in \B(0;1)\setminus\{0\}$ be such that 
\begin{equation}\label{Eq:d-22}\underset{r\to 0^+}{\lim\sup}\frac{ \mathscr H^{d-2+\e}\left(\{\n \eta_0=0\}\cap \{\eta_0=0\}\cap \B(x_0;r)\right)}{r^{d-2+\e}}>0.\end{equation}
Observe that $\eta_1:\mathbb S(0;1)\ni x\mapsto \eta_0(x)$ solves
\[-\Delta_{\mathbb S(0;1)} \eta_1=2d\eta_1+(f(0)+g(0))\mathds 1_{\{\eta_1>0\}}-g(0)=(f_1+g_1)\mathds 1_{\{\eta_1>0\}}-g_1,\]
where $(f_1,g_1)=(f(0)+2d\eta_1,g(0)-2d\eta_1)$ still satisfy the assumption \eqref{Eq:C3Regularity}. Furthermore, it is clear (taking $h$ homogeneous) that $\eta_1$ satisfies \eqref{Eq:QBT2}. Finally, from \eqref{Eq:d-22}, there exists a point $x_1\in \mathbb S(0;1)$ such that 
\begin{equation}\label{Eq:d-222}\underset{r\to 0^+}{\lim\sup}\frac{ \mathscr H^{d-2+\e}\left(\{\n \eta_1=0\}\cap \{\eta_1=0\}\cap \B_{\mathbb S(0;1)}(x_0;r)\right)}{r^{d-2+\e}}>0\end{equation} where $\mathbb B_{\mathbb S(0;1)}$ denotes geodesic balls on $\mathbb S(0;1)$. 
Now, introducing $\tilde \eta_1:=\eta_1\circ \psi_1,\, \tilde f_1:=f_1\circ\psi_1,\, \tilde g_1:=g\circ\psi_1$ where $\psi_1:\B(0;1)\to \mathbb S^{d-1}$ is a local analytic chart, we see that $\tilde \eta_1$ is a solution of 
\[-\Delta \tilde\eta_1+g(\psi_1,\n \eta_1,\eta_1\circ\psi,\n\eta_1\circ\psi)=(\tilde f_1+\tilde g_1)\mathds 1_{\{\tilde\eta_1>0\}}-\tilde g_1,\] which rewrites 
\[-\Delta \tilde\eta_1=(f_1'+g_1')\mathds 1_{\{\tilde\eta_1>0\}}-g_1',\] where $(\tilde\eta_1,\tilde g_1)$ still satisfy \eqref{Eq:C3Regularity}. Besides, $\tilde \eta_1$ still satisfies \eqref{Eq:QBT2}. Iterating this procedure $(d-2)$ times, we would obtain a solution $\hat\eta$ of 
\[-\Delta\hat \eta=(\hat f+\hat g) \mathds 1_{\{\hat \eta>0\}}-\hat g\text{ in }\R^2\] that satisfies \eqref{Eq:QBT2}, where $(\hat f,\hat g)$ satisfy \eqref{Eq:C3Regularity}, with $\mathscr H^\e(\{\n \hat \eta=0\}\cap \{\hat \eta=0\})>0$. The conclusion follows from the 2 dimensional case.

    \end{proof}

\bibliographystyle{abbrv}
\bibliography{BiblioCaffarelli.bib}

\appendix
\part*{Appendices}
\section{Proof of some technical results}\label{Ap:Technical}
\subsection{Proof of Lemma \ref{Le:Differentiability}}\label{Ap:Differentiability}

\begin{proof}[Proof of Lemma \ref{Le:Differentiability}]
We simply give the main elements: consider $m\in \mathcal M$ and $h\in L^\infty(\O)$; define, for any $t\in (-1;1)$ small enough, $m_t:=m+th\,, \theta_t:=\theta_{m_t}$. Finally, set $z_t:=\frac{\theta_t-\theta_0}t$. Direct computations show that there exists a function $\xi_t$ such that $\xi_t\in [\theta_t;\theta_0]$ and
\begin{equation}\label{Eq:C2WFZt} -\Delta z_t-2\langle \n z_t,\n\theta_t+\n\theta_0\rangle=h+Q'(x,\xi_t)z_t.\end{equation}
Note that by elliptic regularity, $\theta_t\underset{t\to 0}\rightarrow \theta_0$ strongly in $W^{1,2}(\T)$. We can thus assume from \eqref{Eq:AssQ} 
that for any $|t|$ small enough $Q'(x,\xi_t)\leq -\delta$ for some fixed $\delta$ whence the maximum principle gives 
\[ \Vert z_t\Vert_{L^\infty(\T)}\leq \frac1\delta \Vert h\Vert_{L^\infty(\T)}.\] This implies (using $z_t$ as a test function) that 
\[ \underset{t\to 0}{\lim\sup}\Vert z_t\Vert_{W^{1,2}(\T)}<+\infty.\] Finally, let $z_0$ be a closure point of $\{z_t\}_{t\to 0}$. Passing to the limit in \eqref{Eq:C2WFZt} shows that $z_0$ solves 
\begin{equation}\label{Eq:C2WFZ0}
-\Delta z_0-2\langle \n z_0,\n \theta_0\rangle=h+\partial_\theta Q(x,\theta_0)z_0.\end{equation} It remains to show that \eqref{Eq:C2WFZ0} has a unique solution. If it has two distinct solutions, it follows that the equation 
\[ -\Delta z-2\langle \n z,\n \theta_0\rangle-\partial_\theta Q(x,\theta_0)z=0\] has a non-trivial solution. Once again by \eqref{Eq:AssQ}, looking at the equation around points of maximum and minimum, this implies $z=0$. Thus $z_t\underset{t\to 0}\rightarrow z_0$ weakly in $W^{1,2}(\T)$, strongly in $L^2(\T)$, thereby concluding the proof for the first-order differentiability. The second-order differentiability follows along the same lines.
\end{proof}

\subsection{Proof of Estimate \eqref{Eq:WK2Est}}\label{Ap:Fredholm}
We prove a more general result and consider an elliptic operator of the form\[ \mathcal L:=-\Delta+\langle b,\n\cdot\rangle +c\] with $b\,, \n\cdot b\in L^\infty(\T)$ and $c\in L^\infty(\T)$. We assume that the principal eigenvalue $\Lambda(\mathcal L)$ satisfies 
\begin{equation}\label{Eq:SpectralGap} \Lambda(\mathcal L)>0.\end{equation}

We have the following proposition:
\begin{proposition}\label{Th:Solvability}
For any $f\in W^{-2,2}(\T)$, there exists a unique $u_f\in L^2(\T)$ such that 
\begin{equation}\label{Eq:C2Sept}\mathcal L u=f\text{ in }W^{-2,2}(\T)\end{equation} in the sense that 
\[ \forall v\in W^{2,2}(\T)\,, \int_{\T} u \mathcal L^* v=\int_{\T}fv\] where $\mathcal L^*$ is the adjoint of $\mathcal L$. Furthermore, there exists a constant $C$ such that
\[ \Vert u\Vert_{L^2(\T)}\leq C\Vert f\Vert_{W^{-2,2}(\T)}.\]
\end{proposition}

\begin{proof}[Proof of Proposition \ref{Th:Solvability}]
We begin with the case $f\in L^2(\T)$.  In that case, the Fredholm alternative yields the existence and uniqueness of a solution $u_f$ to \eqref{Eq:C2Sept}. Furthermore, by elliptic estimates, if $f_k\underset{k\to \infty}\rightarrow f$ in $L^2(\T)$, it follows that $u_{f_k}\underset{k\to\infty}\rightarrow u_{f_\infty}$. Consequently, the linear map $S:f\mapsto u_f$ is continuous in $L^2(\T)$, whence the existence of a constant $C$ such that 
\begin{equation}\label{Eq:L2Est} \forall f \in L^2(\T)\,, \Vert u_f\Vert_{L^2(\T)}\leq C\Vert f\Vert_{L^2(\T)}.\end{equation}  Now, let $f\in L^2(\T)$ and $u_f$ solve \eqref{Eq:C2Sept}. Let $\p_f$ solve 
\[ \mathcal L^*\p_f=u_f.\] By elliptic regularity, 
\[ \Vert \p_f\Vert_{W^{2,2}(\T)}\leq C\Vert u_f\Vert_{L^2(\T)}\] for some (other) constant $C$. Thus
\begin{align*}
\int_{\T} u_f^2&=\int_{\T} (\mathcal L u_f)\p_f
\\&=\int_{\T} f\p_f
\\&\leq \Vert f\Vert_{W^{-2,2}(\T)} \cdot\Vert \p_f\Vert_{W^{2,2}(\T)}
\\&\lesssim \Vert f\Vert_{L^2(\T)}\cdot \Vert u_f\Vert_{L^2(\T)}
\end{align*} which gives the existence of a constant $C$ such that for any $f\in L^2(\T)$
\[ \Vert u_f\Vert_{L^2(\T)}\leq C\Vert f\Vert_{W^{-2,2}(\T)}.\] However, $L^2(\T)$ is dense in $W^{-2,2}(\T)$

We can then conclude by approximating any $f\in W^{-2,2}(\T)$ by a sequence of $L^2$ functions.

\end{proof}

The newt proposition is the following:
\begin{proposition}\label{Th:NegativeSobolev}
Let $\mathcal L$ satisfy \eqref{Eq:SpectralGap}. Then there exist two constants $c\,, C$ such that, for any $f\in L^2(\T)$, if we let $u_f$ solve
\[ \mathcal Lu_f=f\text{ in }\T\] in a weak $W^{1,2}(\T)$ sense, there holds 
\begin{equation}\label{Eq:L2Est1}
c \Vert u_f\Vert_{L^2(\T)}\leq \Vert f\Vert_{W^{-2,2}(\T)}\leq C\Vert u_f\Vert_{L^2(\T)}\end{equation} and 
\begin{equation}\label{Eq:SobEst2}
c \Vert  u_f\Vert_{W^{1,2}(\T)}\leq \Vert f\Vert_{W^{-1,2}(\T)}\leq C\Vert u_f\Vert_{W^{1,2}(\T)}.\end{equation}\end{proposition}
\begin{proof}[Proof of Proposition \ref{Th:NegativeSobolev}]
We already established that there holds $\Vert u_f\Vert_{L^2(\T)}\lesssim \Vert f\Vert_{W^{-2,2}(\T)}$, see Eq. \eqref{Eq:L2Est}. To prove the converse estimate, let $v\in W^{2,2}(\T)$. Then
\begin{align*}
\left|\int_{\T} fv\right|&=\left|\int_{\T} v\mathcal L u_f\right|
\\&=\left|\int_{\T} (\mathcal L^* v)u_f\right|
\\&\leq \Vert \mathcal L^* v\Vert_{L^2(\T)}\cdot\Vert u_f\Vert_{L^2(\T)}.
\end{align*}
However, as 
\[ \mathcal L^*v=-\Delta v-v\n\cdot b-\langle \n v,b\rangle+cv\] we deduce that 
\[ \left|\int_{\T}fv\right|\lesssim \Vert u_f\Vert_{L^2(\T)}\Vert v\Vert_{W^{2,2}(\T)}\] and \eqref{Eq:L2Est1} follows.

To prove \eqref{Eq:SobEst2} we begin by observing that, letting $v$ solve 
\[ \mathcal L^*v=-\n\cdot(\n u)\] we have by similar duality arguments
\[ \Vert v\Vert_{W^{1,2}(\T)}\lesssim \Vert \n u\Vert_{L^2(\T)}.\]We thus obtain, 
\[ \int_{\T}|\n u_f|^2=\int_{\T} fv\lesssim \Vert f\Vert_{W^{-1,2}(\T)}\Vert \n u\Vert_{L^2(\T)}\] and so, in combination with \eqref{Eq:L2Est}, this gives 
\[ \Vert u_f\Vert_{W^{1,2}(\T)}\lesssim \Vert f\Vert_{W^{-1,2}(\T)}.\] Besides,  using the fact that $c\,, b\,, \n\cdot b\in L^\infty(\T)$, we deduce that for any $v\in W^{1,2}(\T)$ there holds
\begin{align*}
\left|\int_{\T} fv\right|&=\left|\int_{\T}( \mathcal L u_f)v\right|
\\&\lesssim \Vert u_f\Vert_{W^{1,2}(\T)}\cdot\Vert v\Vert_{W^{1,2}(\T)},
\end{align*}
thereby establishing \eqref{Eq:SobEst2}.

\end{proof}

\subsection{Proof of Theorem \ref{Th:Hodograph}}\label{Ap:Hodograph}
The key to Theorem \ref{Th:Hodograph} is the following result, from which the conclusion follows immediately.
\begin{proposition}\label{prop:hod_transform}
    Let $\gamma\in(0,1)$, and $h:(t,x)\in\R\times\R^n\mapsto h(t,x)$ be locally bounded such that $h(t,\cdot)\in \mathscr C^{0,\alpha}_{\text{loc}}(\R^n)$ for any $t\in\R$. Assume that $\eta\in \mathscr C^{0}(\R^n)\cap W^{1,2}_{\text{loc}}(\R^d)$ satisfies 
    \begin{equation}
        \label{eq:eta_hod}
    -\Delta \eta(x)=h(\eta(x),x) \text{ in }\B(0;1)
    \end{equation}
    in the weak sense. 
    
    Let  $x_0\in \B(0;1)$ be such that $\nabla\eta(x_0)\neq0$. Then the level set $\{\eta>\eta(x_0)=:c\}$ is $C^{2,\gamma}$ around $x_0$, in the sense that there exists 
    \begin{itemize}
        \item an open set $V\subset\R^d$,
        \item a $(d-1)$-dimensional open set $V'$,
        \item a $\mathscr  C^{2,\gamma}$ mapping $\varphi:V'\to\R$,
    \end{itemize}
such that up to rotation 
\[
\{\eta>c\}\cap V=\{(x',x_d)\in V'\times\R, x_d>\varphi(x')\}, \]
\[
\{\eta=c\}\cap V=\{(x',\varphi(x')), x'\in V'\}.
\]
\end{proposition}

This type of result is classicaly shown by hodograph transform, which roughly amounts to using $(x,\eta(x))$ as new system of coordinates. For the presentation of the proof we follow closely \cite[Lemma 3]{zbMATH01588525}. 

\begin{proof}[Proof of Proposition \ref{prop:hod_transform}]
    As $h$ is locally bounded, standard elliptic estimates  ensure that $\eta\in W^{2,p}(\B(0;1))\cap C^{1,\gamma}(\B(0;1))$ for any $p\in(1,\infty),$ $\gamma\in(0,1)$. By the implicit function theorem, the assumption $\nabla\eta(x_0)\neq0$  ensures that $\{\eta>c\}$ is a $\mathscr C^{2,\gamma}$ hypersurface around $x_0$ for any $\gamma\in (0;1)$. 
    Up to rotating we can assume that 
    \[
    \begin{cases}
    \partial_{x_i}\eta(x_0)=0, \text{ for } i=1,\ldots, d-1,\\
        \partial_{x_d}\eta(x_0)\neq0.
    \end{cases} \]
    We consider
    \[
    \Phi:(x',x_d)\in B\mapsto (x',\eta(x',x_d)),
    \]
    where $B\subset \B(0;1)$ is a ball centred at $x_0$ taken sufficiently small to ensure  $\Phi$ is a $\mathscr C^{1,\gamma}$-diffeomorphism, and write
    \[\Phi^{-1}:(y',y_d)\in \Phi(B)\mapsto (y',F(y',y_d)),
    \]
    where $F\in  \mathscr C^{1,\gamma}(\Phi(B),\R)$. In particular,
       \[
\{\eta>c\}\cap V=\{(x',x_d)\in V'\times\R, x_d>\varphi(x')\}, \]
\[
\{\eta=c\}\cap V=\{(x',\varphi(x')), x'\in V'\}.
\]
with $V:=\Phi(B)$, $V'$ is the $(d-1)$-dimensional projection onto the $(d-1)$ first coordinates and $\varphi:=F(\cdot,c)\in  \mathscr  C^{1,\gamma}(V',\R)$. It only remains to improve the regularity of $F(\cdot,c)$ to $ \mathscr C^{2,\gamma}$.

    By definition of $F$, one has 
    \[
    F(x',\eta(x',x_d))=x_d \text{ on } B,
    \]
which yields, letting $y=\Phi(x)$,
\[
\begin{cases}
    \partial_{y_i}F(y)+\partial_{y_d}F(y)\partial_{x_i}\eta(x)=0, \text{ for }i=1,\ldots, d-1,\\
    \partial_{y_d}F(y)\partial_{x_d}\eta(x)=1.
\end{cases}
\]
Thus $\partial_{y_d}F\neq0$ on $\Phi(B)$, and we can rewrite these identities as
\[
\begin{cases}
    \partial_{x_i}\eta=-\frac{\partial_{y_i}F}{\partial_{y_d}F}, \text{ for }i=1,\ldots, d-1,\\
    \partial_{x_d}\eta=\frac{1}{\partial_{y_d}F}
\end{cases}
\]
on $B$. We deduce 
\begin{align*}
\partial_{x_dx_d}\eta(x)&=\partial_{x_d}\left(\frac{1}{\partial_{y_d}F(y)}\right), \\
&=\partial_{y_d}\left(\frac{1}{\partial_{y_d}F}\right)(y)\partial_{x_d}\eta(x),\\
&=-\frac{\partial_{y_dy_d}F(y)}{(\partial_{y_d}F)^3(y)},
\end{align*}
and likewise 
\[
\partial_{x_ix_i}\eta(x)=\left(\partial_{y_i}-\left(\frac{\partial_{y_i}F}{\partial_{y_d}F}\right)(y)\partial_{y_d}\right)
\left(\frac{-\partial_{y_i}F}{\partial_{y_d}F}\right)(y).\]
Denoting by $a:=\frac{\partial_{y_i}F}{\partial_{y_d}F}$ on $\Phi(B)$, one has
\begin{align*}
    \partial_{x_ix_i}\eta=\frac{-\partial_{y_iy_i}F}{\partial_{y_d}F}+\frac{\partial_{y_i}F}{(\partial_{y_d}F)^2}\partial_{y_iy_d}F-a\left(-\frac{\partial_{y_dy_i}F}{\partial_{y_d}F}+\frac{\partial_{y_i}F}{(\partial_{y_d}F)^2}\partial_{y_dy_d}F\right),
\end{align*}
We can hence rewrite  \eqref{eq:eta_hod} as
\[
-\Delta\eta=LF,
\]
where $L$ is a quasi-linear second-order elliptic operator. In addition, denoting $y_0:=\Phi(x_0)$ we have
\[
LF(y_0)=-\frac{1}{\partial_{y_d}F(y_0)}\left(\sum_{i=1}^{d-1}\partial_{y_iy_i}F+\frac{1}{(\partial_{y_d}F)^2}\partial_{y_dy_d}F\right)(y_0),
\]
since $\partial_{y_i}F(y_0)=0$. Thus, we can rewrite  \eqref{eq:eta_hod} as
\[
LF(y)=h(y_d,y',F(y',y_d)) \text{ on } \Phi(B).
\]
Now since $h(t,\cdot)$ is $\mathscr C^{0,\alpha}$ for each $t\in\R$, by standard elliptic estimates we obtain that $F$ is $ \mathscr C^{2,\gamma}$ in $y'=(y_1,\ldots,y_{d-1})$. This concludes the proof.

\end{proof}

%\subsection{Proof of the density result (Proposition \ref{Pr:LocalMinimalityShape})}\label{Ap:Density}

\subsection{Proof of Theorem \ref{Pr:C3Weiss}}\label{Ap:QuasiMonotonicity}

\begin{proof}[Proof of Theorem \ref{Pr:C3Weiss}]
The function $\eta_r$ solves
    \[
    -\Delta \eta_r = f_r \mathds 1_{\{\eta_r> 0\}} - g_r 1_{\{\eta_r\leq 0\}}.
    \]
Define, for any $t>0$, 
\[
E(t):=\int_{\B(0;1)}|\nabla \eta_t|^2, \ G(t):=\int_{\partial \B(0;1)}\eta_t^2.
\]
Following \cite{zbMATH05204068}, we decompose $\Psi(s)-\Psi(r)$ as follows: for any $0<r<s$,
\[
\Psi(s)-\Psi(r)=I_1+I_2
\]
where we have set
\[
I_1=\int_r^s\left(\frac{dE}{dt}-2\frac{dG}{dt}-\left(2\int_{\B(0;1)}\frac{\partial\eta_t}{\partial t}\left(f_t \mathds 1_{\{\eta_t> 0\}} - g_t \mathds 1_{\{\eta_t\leq0\}}\right)\right)dt\right),
\] and
\begin{multline*}
I_2:=2\int_r^s\left(\int_{\B(0;1)}\frac{\partial\eta_t}{\partial t}\left(f_t \mathds 1_{\{\eta_t> 0\}} - g_t \mathds 1_{\{\eta_t\leq0\}}\right)\right)ds\\+2\int_{\B(0;1)}\left(f_r (\eta_r)_++g_r (\eta_r)_-\right)-\left(f_s (\eta_s)_++g_s (\eta_s)_-\right).\end{multline*}
We estimate $I_1,\, I_2$ separately, 
\begin{enumerate}
\item\underline{Estimating $I_1$:} our goal is to show the following estimate:
\begin{equation}\label{Eq:MonotI1}
I_1=\int_r^sdt\int_{\partial \B(0;1)}\frac1t\left(\frac{\partial\eta_t}{\partial\nu}-2\eta_t\right)^2=\int_r^s dt\int_{\partial \B(0;1)}t\left(\frac{\partial\eta_t}{\partial t}\right)^2.
\end{equation}
 From \eqref{Eq:Criticality}, we obtain\begin{align*}
\frac{dE}{dt}&=2\int_{\B(0;1)} \left\langle\n\frac{\partial\eta_t}{\partial t},\nabla \eta_t\right\rangle
\\&= -2\int_{\B(0;1)}\frac{\partial\eta_t}{\partial t}\Delta \eta_t+2\int_{\partial \B(0;1)}\frac{\partial\eta_t}{\partial t}\cdot\frac{\partial\eta_t}{\partial \nu}\\
&= 2\int_{\B(0;1)}\frac{\partial\eta_t}{\partial t}\left(f_t \mathds 1_{\{\eta_t>0\}} - g_t \mathds 1_{\{\eta_t\leq0\}}\right)
\\&+2\int_{\partial \B(0;1)}\frac{\partial\eta_t}{\partial t}\cdot\frac{\partial\eta_t}{\partial \nu}.\end{align*}
Likewise, 
\begin{align*}
-2\frac{dG}{dt}&=-4\int_{\partial\B(0;1)}\eta_t\frac{\partial \eta_t}{\partial t}.
\end{align*}
Consequently,
\begin{align*}
I_1=2\int_r^sdt\int_{\partial \B(0;1)}\frac{\partial \eta_t}{\partial t}\left(\frac{\partial \eta_t}{\partial\nu}-2\eta_t\right).
\end{align*} But we now observe that 
\begin{align}\label{Eq:JuanTabo}
\frac{\partial\eta_t}{\partial t}&=\frac{\langle\n \eta(t\cdot),\cdot\rangle}{t}-\frac{2}{t^3}\eta(t\cdot)
\\&=\frac1t\left(\langle\n\eta_t(\cdot),\cdot\rangle-\frac2t\eta_t \right) 
\end{align}
so that, as $\partial_\nu\eta_t(x)=\langle \n\eta_t(x),x\rangle$ on $\partial \B(0;1)$, we finally get \eqref{Eq:MonotI1}:
\begin{equation*}
I_1=\int_r^sdt\int_{\partial \B(0;1)}\frac1t\left(\frac{\partial\eta_t}{\partial\nu}-2\eta_t\right)^2=\int_r^s dt\int_{\partial \B(0;1)}t\left(\frac{\partial\eta_t}{\partial t}\right)^2.
\end{equation*}
\item\underline{Estimating $I_2$:} our goal is to show 
the existence of $C\,, \beta>0$ that only depend on the $W^{1,p}$ norms of $f$ and $g$ (for $p$ large enough) such that\begin{equation}\label{Eq:MonotI2}\left| I_2\right|\leq C\left(s^\beta-r^\beta\right).\end{equation}
To alleviate notations, we split $I_2$ in two parts by setting 
\[ I_2=I_2(f)+I_2(g)\] with 
\[I_2(f)=2\int_r^s\left( \int_{\B(0;1)}\frac{\partial\eta_t}{\partial t}f_t \mathds 1_{\{\eta_t> 0\}}\right)dt +2\int_{\B(0;1)}f_r (\eta_r)_+-f_s (\eta_s)_+
\]
\[ I_2(g)=-2\int_r^s\left( \int_{\B(0;1)}\frac{\partial\eta_t}{\partial t}g_t \mathds 1_{\{\eta_t\leq0\}}\right)dt +2\int_{\B(0;1)}g_r (\eta_r)_--g_s (\eta_s)_-,
\]
and we will show the existence of $C,\, \beta>0$ such that
\begin{equation}\label{Eq:MonotI2F} |I_2(f)|\leq C\left(s^\beta-r^\beta\right),\end{equation} the same proof yielding a similar bound for $I_2(g)$. Observe that 
\[ \frac{\partial\eta_{t,+}}{\partial t}=\frac{\partial\eta_t}{\partial t}\mathds 1_{\{\eta_t> 0\}}.\]Letting $\sigma=\frac{s}r$, we obtain
\begin{align*}
\frac12I_2(f)&=\int_r^s\left( \int_{\B(0;1)}\frac{\partial\eta_t}{\partial t}f_t \mathds 1_{\{\eta_t> 0\}}\right)dt +\int_{\B(0;1)}f_r (\eta_r)_+-f_s (\eta_s)_+
\\&=\int_r^s\left( \int_{\B(0;1)}f_t\frac{\partial\eta_{t,+}}{\partial t}\right)dt+\int_{\B(0;1)}f_r (\eta_r)_+-f_s (\eta_s)_+
\\&=-\int_r^s \left( \int_{\B(0;1)} \frac{\partial f_t}{\partial t} \eta_{t,+}\right)dt
\\&=-\int_s^r \left(\int_{\B(0;1)} \langle \n f(tx),x\rangle \eta_{t,+}\right)dt.
\end{align*}
Recall the Zygmund class inequality \cite[Proposition 2.3.7]{zbMATH01201583}: as $\eta(0)=0\,, \n\eta(0)=0$, 
\begin{equation}\label{Eq:Zygmund}\forall x \in \B(0;1)\,, |\eta(tx)|\leq C t^2\Vert x\Vert^2|\log(t\Vert x\Vert)|.\end{equation} As, for any $\gamma<1$ we have \[ |\log(t\Vert x\Vert)|\lesssim t^{-\gamma}\Vert x\Vert^{-\gamma},\] we deduce
\begin{equation}\label{Eq:Zygmund2} | \eta_t(x)|\lesssim \Vert x\Vert^{2-\gamma}t^{-\gamma}\lesssim t^{-\gamma}\text{ in }\B(0;1).\end{equation} We will fix the parameter $\gamma$ later on. From the H\"{o}lder inequality, we obtain, for any $p>1$ (whose Lebesgue conjugate exponent is denoted $p'$)
\begin{align*}
|I_2(f)|&\leq\int_r^s \left(\int_{\B(0;1)}t^{-\gamma}|\n f(tx)|dx\right)dt
\\&\lesssim \int_s^r t^{-\gamma}t^{-d}\int_{\B(0;t)}|\n f(\xi)|d\xi dt
\\&\lesssim \int_r^s t^{-\gamma-d}|\B(0;t)|^{\frac1{p'}}\cdot\Vert\n f\Vert_{L^p(\B(0;1))}dt
\\&\lesssim \Vert \n f\Vert_{L^p(\B(0;1))}\int_r^s t^{-\gamma-d+\frac{d}{p'}}dt
\\&= \Vert \n f\Vert_{L^p(\B(0;1))}\left(s^{1-\gamma+d\left(\frac1{p'}-1\right)}-r^{1-\gamma+d\left(\frac1{p'}-1\right)}\right)
\\&= \Vert \n f\Vert_{L^p(\B(0;1))}\left(s^{1-\gamma-\frac{d}{p}}-r^{1-\gamma-\frac{d}p}\right).
\end{align*}
Taking  $p$ large enough to ensure that $\beta=1-\gamma-\frac{d}{p}$ satisfies $\beta>0$ yields \eqref{Eq:MonotI2F} and applying the same chain of inequalities to bound $I_2(g)$  finally leads to \eqref{Eq:MonotI2}.
\end{enumerate}
We can now conclude:
\[ \Psi(s)-\Psi(r)\geq \int_r^s \left(\int_{\partial \B(0;1)}t\left(\frac{\partial\eta_t}{\partial t}\right)^2\right)dt-C(s^\beta-r^\beta).\] Thus, $W:r\mapsto \Psi(r)+Cr^\beta$ is non-decreasing. Finally, to obtain \eqref{Eq:C3QuantifiedWeiss}, let $r>0$ and $s\in (r;2r)$. Then 
\begin{align*}
 \int_{\partial \B(0;1)}(\eta_s-\eta_r)^2&=\int_{\partial\B(0;1)}\left(\int_r^s \frac{\partial \psi}{\partial t}dt\right)^2
 \\&\leq (s-r)  \int_r^s\int_{\partial \B(0;1)}\left(\frac{\partial\eta_t}{\partial t}\right)^2
 \\&\leq   r\int_r^s\int_{\partial \B(0;1)}\left(\frac{\partial\eta_t}{\partial t}\right)^2
 \\&\leq   \int_r^st\int_{\partial \B(0;1)}\left(\frac{\partial\eta_t}{\partial t}\right)^2
 \\&\leq I_1.
\end{align*}
This concludes the proof.
\end{proof}

\section{Proof of Theorem \ref{Th:MW1}}\label{Ap:WeissBU}

\subsubsection{Preliminary results}
The proof of Theorem \ref{Th:MW1} is lengthy, and we rely on the two following technical lemmas. The first one gathers some results regarding sub and super-harmonicity of auxiliary functions related to the problem;  this is akin to \cite[Lemma 3.6]{zbMATH05311233}. We use the notations of Theorem \ref{Th:MW1}.
\begin{lemma}\label{lemma:tech_sub_suph}
The following holds:
    \begin{itemize}
    \item
    There exists 
    \[M:=M(\|f\|_{L^\infty(\B(0;1))},\|g\|_{L^\infty(\B(0;1))})>0\]  such that, for any $r\in(0;1)$,
    \[
    \eta_{r,+}+M|\cdot|^2 \text{ and } \eta_{r,-}+M|\cdot|^2
    \]
    are sub-harmonic.
    \item Let $a,b\in\R$ be such that $a+b>0$. Then there exists 
    \[M:=M(a,b,\|f\|_{L^\infty(\B(0;1))},\|g\|_{L^\infty(\B(0;1))})>0\] such that, for any $r\in(0;1)$,    \[
    a\eta_{r,+}+b\eta_{r,-}+M|\cdot|^2
    \]
    is sub-harmonic.
    \end{itemize}
\end{lemma}
\begin{proof}[Proof of Lemma \ref{lemma:tech_sub_suph}]
   Recall the Kato inequality \cite{zbMATH03389535}: if $w\in H^2(\B(0;1))$ then $\Delta |w|\geq \text{sign}(w)\Delta w$ in the sense of distributions, that is
    \[
    \forall \Psi\in C^{\infty}_{c}(\B(0;1)),\, \Psi\geq0,\ -\int_{\B(0;1)}\nabla |w|\cdot\nabla\Psi\geq \int_{\B(0;1)}(\text{sign}(w)\Delta w)\Psi.
    \]
Using this inequality and the fact that 
\[
\eta_{r,+}=\frac{|\eta_r|+\eta_r}{2}
\]
we get, in the sense of distributions
\begin{align*}
    -\Delta \eta_{r,+}&\leq \frac{1}{2}\left(1+\text{sign}(\eta_r)\right)\left(f_r\mathds 1_{\{\eta_r> 0\}}-g_r\mathds 1_{\{\eta_r\leq 0\}}\right)\\
    &=f_r\mathds 1_{\{\eta_r> 0\}}
    \\&\leq \|f\|_{L^\infty(\B_r)}.
\end{align*}
Likewise, from 
\[
\eta_{r,-}=\frac{|\eta_r|-\eta_r}{2}
\]
we obtain
\[
-\Delta \eta_{r,-}\leq \|g\|_{L^\infty(\B_r)}
\]
and so we can fix a constant $M$ such that $\eta_{r,\pm}+M\| \cdot\|^2$ is sub-harmonic.

The last point of the lemma is proved the same way: observe that 
\[
a\eta_{r,+}+b\eta_{r,-}=\frac12\left((a+b)|\eta_r|+(a-b)\eta_r\right)
\]
so that, if $a+b>0$,
\begin{align*}
   - \Delta(a\eta_{r,+}+b\eta_{r,-})&\leq \frac12\left((a+b)\text{sign}(\eta_r)+(a-b)\right)\left(f_r\mathds 1_{\{\eta_r> 0\}}-g_r\mathds 1_{\{\eta_r\leq 0\}}\right)\\
    &=a f_r\mathds 1_{\{\eta_r> 0\}}+b g_r\mathds 1_{\{\eta_r< 0\}}.
\end{align*}
The conclusion follows.\end{proof}

\begin{lemma}\label{lemma:tech_b_up}
Under Assumption \eqref{Eq:MW6}, there exists $C>0$ such that
\begin{equation}\label{eq:ur-_ur+}
\forall r>0, \ \int_{\partial \B(0;1)}\eta_{r,-}^2\leq C\left(1+\int_{\partial \B(0;1)}\eta_{r,+}^2\right)
\end{equation}
for each $r>0$.
\end{lemma}
\begin{proof}[Proof of Lemma \ref{lemma:tech_b_up}]
We argue by contradiction and assume that there exists some sequence $r_k\underset{n\to \infty}\rightarrow 0$ such that 
\begin{equation}\label{eq:contrad1}
\forall n\in\N,\ \int_{\partial \B(0;1)}\eta_{r_k,-}^2\geq k\left(1+\int_{\partial \B(0;1)}\eta_{r_k,+}^2\right).
\end{equation}
Let us show that this implies
\begin{equation}\label{eq:contrad2}
\int_{\B(0;1)}\eta_{r_k}^2\lesssim \int_{\partial \B(0;1)}\eta_{r_k,-}^2
\end{equation}

\begin{proof}[Proof of \eqref{eq:contrad2}]
If \eqref{eq:contrad2} does not hold then, up to taking another subsequence (still labeled $\{r_k\}_{k\in \N}$)
\begin{equation}\label{Eq:AM}
\int_{\B(0;1)}\eta_{r_k}^2\geq k\int_{\partial \B(0;1)}\eta_{r_k,-}^2
\end{equation}
Define, for any $k\in \N$, 
\[w_k:=\frac{\eta_{r_k}}{\left(\int_{\B(0;1)}\eta_{r_k}^2\right)^{1/2}}\] which satisfies (by \eqref{eq:contrad1}--\eqref{Eq:AM})
\begin{equation}\label{eq:contrad3}\int_{\partial \B(0;1)}w_k^2\underset{n\to\infty}\rightarrow 0.\end{equation}
From Theorem \ref{Pr:C3Weiss}
\begin{multline*}
\int_{\B(0;1)}|\nabla w_k|^2\leq \frac{1}{\int_{\B(0;1)}\eta_{r_k}^2}W(1)+\frac{2}{\int_{\B(0;1)}\eta_{r_k}^2}\left(\int_{\B(0;1)}f_{r_k}\eta_{r_k,+}+g_{r_k}\eta_{r_k,-}\right)\\+2\int_{\partial \B(0;1)}w_k^2+C\frac{(1-r_k^\gamma)}{\int_{\B(0;1)}\eta_{r_k}^2}.
\end{multline*}
As  \eqref{eq:contrad1}--\eqref{Eq:AM} imply
\[ \int_{\B(0;1)}\eta_{r_k}^2\underset{n\to \infty}\rightarrow +\infty\]
 we deduce that $\{w_k\}_{n\in \N}$ is bounded in $W^{1,2}(\B(0;1))$, so that (up to a further subsequence) it converges weakly in $W^{1,2}(\B(0;1))$, strongly in $L^2(\B(0;1))$ and $L^2(\partial \B(0;1))$ to some $w_\infty\in W^{1,2}(\B(0;1))$, which is constant by passing to the weak limit in the above estimate. But then, the fact that $\int_{\partial \B(0;1)}w_\infty^2=0$ yields a contradiction with $\int_{\B(0;1)}w_\infty^2=1$.
\end{proof} Now, consider
\[
v_k:=\frac{\eta_{r_k}}{\left(\int_{\partial \B(0;1)} \eta_{r_k,-}^2\right)^{1/2}}.
\]
From \eqref{eq:contrad1}-\eqref{eq:contrad2} $\{v_k\}_{k\in \N}$ is bounded in $L^2(\B(0;1))$ and $L^2(\partial \B(0;1))$. From Theorem \ref{Pr:C3Weiss}, we deduce 
\begin{multline*}
\int_{\B(0;1)}|\nabla v_k|^2\leq \frac{W(1)}{\int_{\partial \B(0;1)}\eta_{r_k,-}^2}\\+\frac{2}{\int_{\partial \B(0;1)}\eta_{r_k,-}^2}\left(\int_{\B(0;1)}f_{r_k}\eta_{r_k,+}+g_{r_k}\eta_{r_k,-}\right)+2\int_{\partial \B(0;1)}v_k^2+C\frac{(1-r_k^\gamma)}{\int_{\partial \B(0;1)}\eta_{r_k}^2}.
\end{multline*}
Now, as $\int_{\partial \B(0;1)}\eta_{r_k,-}^2\to+\infty$, \eqref{eq:contrad1}--\eqref{eq:contrad2} imply the boundedness of $\{v_k\}_{k\in \N}$ in $W^{1,2}(\B(0;1))$. As $\{\Delta v_k\}_{n\in \N}$ is uniformly bounded in $L^\infty$  we deduce that, up to a subsequence, $\{v_k\}_{k\in \N}$ converges in $C^{1,\gamma}_{\text{loc}}(\B(0;1))\cap L^2(\partial \B(0;1))$ to $v_\infty$ satisfying 
\[
\begin{cases}
    -\Delta v_\infty=0, \text{ in }B_1,\\
    v_\infty(0)=0.
    %v_\infty\leq 0 \text{ in }\partial \B(0;1)
\end{cases}
\]
Furthermore, as
\[\int_{\partial \B(0;1)}v_{\infty,+}^2=\lim_{n\to+\infty}\frac{\int_{\partial \B(0;1)}\eta_{r_k,+}^2}{\int_{\partial \B(0;1)} \eta_{r_k,-}^2}=0, \text{ from \eqref{eq:contrad1}}\] we obtain 
\[ v_\infty\leq 0\text{ on }\partial \B(0;1).\] 
The strong maximum principle implies $v_\infty\equiv 0$. On the other hand, $\int_{\partial \B(0;1)}v_{\infty,-}^2=\lim \int_{\partial \B(0;1)}v_{n,-}^2=1$, thus yielding a final contradiction. This implies \eqref{eq:ur-_ur+}. 
\end{proof}

\subsubsection{Proof of Theorem \ref{Th:MW1}}We can now move on to the proof of Theorem \ref{Th:MW1}.

\begin{proof}[Proof of Theorem \ref{Th:MW1}]
We mostly follow the approach of \cite{zbMATH05129536}. As a preliminary step let us show that
\begin{equation}\label{Eq:ExpressionDI}
\frac12\frac{d}{dr}\int_{\partial \B(0;1)} \eta_r^2=\frac1r\left(\Psi(r)-\int_{\B(0;1)} f_r \eta_{r,+}-\int_{\B(0;1)} g_r \eta_{r,-}\right).
\end{equation}
Indeed, 
\begin{align*}
\frac12\frac{d}{dr}\int_{\partial \B(0;1)} \eta_r^2&=\int_{\partial \B(0;1)} \eta_r\frac{\partial \eta_r}{\partial r}
\\&=\frac1r\int_{\partial \B(0;1)} \eta_r\langle\n \eta_r,x\rangle-\frac2r\int_{\partial \B(0;1)}\eta_r^2
\\&=\frac1r\int_{ \B(0;1)}\left(\eta_r\Delta \eta_r+|\n \eta_r|^2\right)-\frac2r\int_{\partial \B(0;1)}\eta_r^2 
\\&=\frac1r\int_{\B(0;1)}\left(|\n \eta_r|^2-(f_r\eta_{r,+}+g_r\eta_{r,-})\right)-\frac2r\int_{\partial \B(0;1)}\eta_r^2
\\&=\frac1r\left(\Psi(r)+\int_{\B(0;1)} \left(f_r \eta_{r,+}+ g_r \eta_{r,-}\right)\right).\end{align*}We now distinguish between the cases $\Psi(0^+)>-\infty$ and $\Psi(0^+)=-\infty$.

\textbf{Case $\Psi(0^+)>-\infty$.} 

\textit{Step 1.} We first prove that there exists positive constant $\zeta, r_0, C>0$ only depending on $f(0), g(0), \|f\|_{C^{0,\gamma}(\B(0;1))}$ and $\|g\|_{C^{0,\gamma}(\B(0;1))}$ such that for any $r\in(0,r_0)$
\begin{equation}\label{eq:fg_below}\int_{\B(0;1)} \left(f_r \eta_{r,+}+ g_r \eta_{r,-}\right)\geq \zeta\int_{\B(0;1)}\eta_{r,+}-C.\end{equation}
By continuity of $f$ and $g$, and recalling that $\min (f+g)>0$, we can find  $a>0,b\in\R$ and a radius $r_0>0$ such that
\[
\forall r\in(0,r_0), \ f\geq  a, g\geq b \text{ in } \B_r,\,, a+b>0.
\]
We thus have
\begin{align*}
    \int_{\B(0;1)} \left(f_r \eta_{r,+}+ g_r \eta_{r,-}\right)&\geq \int_{\B(0;1)}\left(a \eta_{r,+}+b\eta_{r,-}\right)
\end{align*}
Let now $\zeta>0$ be such that $(a-c)+b>0$. We fix $M>0$ provided by Lemma \ref{lemma:tech_sub_suph} such that
\[
\overline{\eta_r}:=(a-\zeta) \eta_{r,+}+ b\eta_{r,-}+M|\cdot|^2 \text{ is sub-harmonic in }\B(0;1).
\]
Using the mean value property for sub-harmonic function, we deduce that
\[
\fint_{\B(0;1)} \left((a+\zeta) \eta_{r,+}+b \eta_{r,-}+M|\cdot|^2\right)\geq \overline\eta_r(0)=0.
\]
Consequently,
\[
\int_{\B(0;1)}\left(a \eta_{r,+}+b\eta_{r,-}\right)\geq \zeta\int_{\B(0;1)}\eta_{r,+} -M\fint_{\B(0;1)}|x|^2
\]
thus giving \eqref{eq:fg_below}.

\textit{Step 2.} We now show that, for any  $M'>0$ there exists $M''>0$ such that 
\begin{equation}\label{Eq:MW3}\int_{\partial \B(0;1)}\eta_r^2\geq M''\Rightarrow \int_{\B(0;1)} \eta_{r,+}\geq M'.\end{equation}
From the mean value property for the sub-harmonic function $\eta_{r,+}+M|\cdot|^2$ (for some $M>0$ given by Lemma \ref{lemma:tech_sub_suph}) we have
\[
\forall x\in \B_{3/4}\setminus \B_{1/2},\ \left(\eta_{r,+}(x)+M|x|^2\right)\leq\fint_{\B(0;1)} \left(\eta_{r,+}+M|\cdot|^2\right)
\]
so that it holds 
\[
\left(\int_{\B_{3/4}\setminus \B_{1/2}}\left(\eta_{r,+}+M|\cdot|^2\right)^2\right)^{1/2}\leq C_d\fint_{\B(0;1)}\left(\eta_{r,+}+M|\cdot|^2\right)
\]
for some dimensional constant $C_d>0$, providing 
\begin{align}
\nonumber
    \left(\int_{\B(0;1)}\eta_{r,+}\right)^{2}&\geq C'\int_{\B_{3/4}\setminus \B_{1/2}}\eta_{r,+}^2-C''\\
    &\geq C'''\int_{1/2}^{3/4}\left(\int_{\partial\B(0;1)}\eta_{rs,+}^2\right)ds-C''
    \label{eq:subh_ur+_est}
\end{align}
for constants $C',C'',C'''>0$. Now, for any $s\in\left(\frac12;\frac34\right)$,
\begin{align*}
    \int_{\partial\B(0;1)}\eta_{rs,+}^2&=\int_{\partial\B(0;1)}\eta_{rs}^2-\eta_{rs,-}^2\\
    &\geq \int_{\partial\B(0;1)}\eta_{rs}^2-C\int_{\partial\B(0;1)}\eta_{rs,+}^2-C
\end{align*}
where the constant $C>0$ is given by Lemma \ref{lemma:tech_b_up}, thus yielding
\begin{equation}
    \label{eq:urs_parB1_est}
    \int_{\partial\B(0;1)}\eta_{rs,+}^2\geq \frac{1}{C+1}\int_{\partial\B(0;1)}\eta_{rs}^2-\frac{C}{C+1}.
\end{equation}
for the same $s$. On the other hand, from the monotonicity formula \eqref{Eq:C3QuantifiedWeiss} we know that there exists $C,\, \beta>0$ such that for any $r$ small enough and any $s\in (1/2,3/4)$, 
\[C(r^{\beta}-(rs)^{\beta})+\Psi(r)-\Psi(rs)\geq C\int_{\partial \B(0;1)}(\eta_r-\eta_{rs})^2.\] 
Since $\Psi$ is uniformly continuous (as $\Psi(0^+)>-\infty$), we get that for any $r$ small enough,
\[
\int_{\partial \B(0;1)}(\eta_r-\eta_{rs})^2\leq1.
\]
Combining this estimate with the elementary inequality
\[
\frac{1}{2}\int_{\partial\B(0;1)}\eta_r^2\leq \int_{\partial\B(0;1)}(\eta_r-\eta_{rs})^2+\int_{\partial\B(0;1)}\eta_{rs}^2
\]
we obtain that if $r$ is small enough then
%\begin{equation}\label{Eq:Metz}
%\int_{\partial \B(0;1)}\eta_r^2\geq M''\Rightarrow \int_{\B(0;1)} \eta_r^2\geq M'.
%\end{equation}(\`a r\'ediger proprement, en gros on regarde simplement les moennes sur des sphères). 
\begin{equation}
    \label{eq:ur_urs}
    \forall s\in \left(\frac12;\frac34\right), \
    \int_{\partial\B(0;1)}\eta_{rs}^2\geq \frac{1}{2}\int_{\partial\B(0;1)}\eta_r^2-1
\end{equation} 
for some constant $C>0$. From \eqref{eq:ur_urs}, \eqref{eq:urs_parB1_est} and \eqref{eq:subh_ur+_est} we deduce that there exists $c,C>0$ such that for any $r$ small enough\[
\left(\int_{\B(0;1)}\eta_{r,+}\right)^{2}\geq c\int_{\partial\B(0;1)}\eta_r^2-C,
\]
which proves \eqref{Eq:MW3}.

\textit{Step 3.} We now show that
\begin{equation}
\label{Eq:Metz5}
\underset{r\to 0^+}{\lim\sup} \int_{\partial \B(0;1)}\eta_r^2<+\infty.\end{equation}

To this end, we first observe the following fact: from \eqref{Eq:ExpressionDI} and $\Psi(0^+)>-\infty$, we obtain, for two constants $A\,, B$, 
\[ \frac{d}{dr}\int_{\partial \B(0;1)}\eta_r^2\geq \frac1r\left(-B+A\int_{\B(0;1)}\eta_{r,+}\right).\] From \eqref{Eq:MW3},  we deduce that there exists $M''>0$ such that 
\begin{equation}\label{Eq:JJ5} \int_{\partial \B_1}\eta_r^2\geq M''\Rightarrow \frac{d}{dr}\int_{\partial \B_1}\eta_r^2> 0.\end{equation}  Assume there exist $r_1<r_0$ such that 
\[ S(r_1)>S(r_0)\geq M''.\]From \eqref{Eq:JJ5}, there exists $\e>0$ such that $S'>0$ in $(r_1;r_1+\e)$. By an easy reasoning, we deduce that 
\[ \sup\left\{r\in (r_1;r_0)\,, S'>0\text{ in }(r_1;r)\right\}=r_0\] and thus $S(r_1)<S(r_0)$. We obtain a contradiction, whence \eqref{Eq:Metz5} follows.

\textit{Step 4.}
We now prove
\begin{equation}\label{Eq:Metz7}
\underset{r\to 0^+}{\lim\sup}\Vert \eta_r\Vert_{\mathscr C^{1,\gamma}(\overline{\B}(0;1))}<\infty\text{ for any $\gamma\in (0;1)$}.\end{equation}
We begin with a weaker Sobolev estimate, namely, \begin{equation}\label{Eq:Metz6}
\underset{r\to 0}{\lim\sup}\int_{\B(0;1)}| \n \eta_r|^2<+\infty.
\end{equation} 
To prove \eqref{Eq:Metz6}, observe that on the one hand
\begin{equation}\label{Eq:MetzPoincare} \forall v \in W^{1,2}(\B(0;1))\,\Vert v\Vert_{L^2(\B(0;1))}\leq C\left(\Vert \n v\Vert_{L^2(\B(0;1))}+\Vert v\Vert_{L^2(\partial \B(0;1))}\right),\end{equation} for some constant $C>0$. On the other hand, from Theorem \ref{Pr:C3Weiss}, we have
\begin{multline}\label{Eq:NanxyPoincare}
    \forall r\in(0,1),\ \int_{\B(0;1)}|\n \eta_r|^2\leq \Psi(1)+2\max(\Vert f\Vert_\infty,\Vert g\Vert_\infty)\Vert \eta_r\Vert_{L^2(\B(0;1))}\\
    +2\int_{\partial \B(0;1)}\eta_r^2
    +C(1-r^\alpha)
    \end{multline} 
    where the constant $C>0$ is given by Theorem \ref{Pr:C3Weiss}.
    Combining \eqref{Eq:Metz5}--\eqref{Eq:MetzPoincare}--\eqref{Eq:NanxyPoincare} gives \eqref{Eq:Metz6}.
    
Now, from standard elliptic regularity estimates \cite[Theorem 12.2.2]{zbMATH06062283} and a bootstrap argument, we deduce that
\begin{equation}\label{Eq:Nancy1}
\forall p\in [1;+\infty),\,\underset{r\to 0}{\lim\sup}\Vert \eta_r\Vert_{W^{2,p}\left(\B\left(0;\frac34\right)\right)}<+\infty
\end{equation}
    whence 
    \begin{equation}\label{Eq:Nancy2}
    \forall \gamma\in(0;1)\,, \underset{r\to 0}{\lim\sup}\Vert \eta_r\Vert_{\mathscr C^{1,\gamma}\left(\overline \B\left(0;\frac12\right)\right)}<\infty.
    \end{equation}
    Using the fact that 
    \(\eta_r=\frac12\eta_{2r}\left(\frac\cdot2\right)\), we obtain  
    \[ \forall \gamma \in (0;1),\, \underset{r\to 0^+}{\lim\sup}\Vert \eta_{r}\Vert_{\mathscr C^{1,\gamma}(\overline\B(0;1))}\lesssim \underset{r\to 0}{\lim\sup}\Vert \eta_r\Vert_{\mathscr C^{1,\gamma}\left(\overline \B\left(0;\frac12\right)\right)}<\infty. \]
    In particular, \eqref{Eq:Metz7} is proved and there exists $\eta_0\in \cap_{\gamma\in (0;1)}\mathscr C^{1,\gamma}(\overline \B(0;1))$ such that, along a subsequence \( r_k\underset{k\to \infty}\rightarrow 0\), we have
    \begin{equation}\label{Eq:Nancy2}\forall \gamma\in (0;1),\, \eta_{r_k}\underset{k\to \infty}\rightarrow \eta_0\text{ in }\mathscr C^{1,\gamma}\left(\overline \B(0;1)\right).\end{equation} In particular, 
    \begin{equation}\label{Eq:ZeroCriticalPsi}
    \n \eta_0(0)=0.\end{equation}

\textit{Step 5.} To identify the equation satisfied by any blow-up limit $\eta_0$, we first rewrite the equation on $\eta_r$ as 
\[-\Delta \eta_r=(f_r+g_r)m_r-g_r\] where $m_r:=\mathds 1_{\{\eta_r>0\}}$. In particular, $m_r$ is a solution of 
\[\text{ Maximise }T(\cdot,\eta_r):L^\infty(\B(0;1));[0;1])\ni m\mapsto \int_{\B(0;1)} m\eta_r.\] Letting $\underline{m}_0$ denote a weak $L^\infty-*$ closure point of the sequence $\{m_r\}_{r\to 0^+}$, we deduce first that 
\begin{equation}\label{Eq:MaxT} \underline{m}_0\in \mathrm{arg max}\,T(\cdot,\eta_0)\end{equation} and, passing to the limit in the equation, that 
\[-\Delta \eta_0=(f(0)+g(0))\underline{m}_0-g(0).\]

\textit{Step 6.}  We now prove that blow-up limits are $2$-homogeneous. To this end, fix a sequence $\{r_k\}_{k\in \N}$ and $\eta_0$ such that \eqref{Eq:Nancy2} holds. For any $s>0$, we have 
\[ \eta_{sr_k}=\frac{\eta_{ r_k}(s\cdot)}{s^2}\underset{k\to +\infty}\to\frac1{s^2}\eta_0(s\cdot)=:\eta_{0,s}\text{ in }\mathscr C^{1,\gamma}\left(\overline\B(0;1)\right) \text{ for any } \gamma\in(0,1).\]
Introduce, for any $v\in W^{1,2}(\B(0;1))$, the notation $v_s:=\frac{v(s\cdot)}{s^2}$ and set
\begin{multline} \forall r\,, s>0\,, \forall v \in W^{1,2}(\B(0;1)),\,\\ \mathcal W(r,s;v):=\int_{\B(0;1)}|\n v_s|^2-2\int_{\B(0;1)}\left(f_rv_{s,+}+g_r v_{s,-}\right)-2\int_{\partial \B(0;1)}v_s^2.\end{multline}
On the one-hand, by continuity of $W$ at $0^+$ we have 
\begin{equation}\label{Eq:Nancy3}
\mathcal W(sr_k,sr_k;\eta)=W(sr_k)\underset{k\to \infty}\rightarrow W(0^+)=\Psi(0^+).\end{equation}
On the other hand, we also have that
\begin{align*}
\mathcal W(sr_k,sr_k;\eta)&=\int_{\B(0;1)}\left|\n (\eta_s)_{r_k}\right|^2-2\int_{\B(0;1)}\left((f_s)_{r_k} (\eta_s)_{r_k,+}+(g_s)_{r_k} (\eta_s)_{r_k,)}\right)
\\&-2\int_{\partial \B(0;1)}(\eta_s)_{r_k}^2
\\&\underset{k\to \infty}\rightarrow \int_{\B(0;1)}\left|\n\eta_{0,s}\right|^2-2\int_{\B(0;1)}\left(f(0)\eta_{0,s,+}+g(0)\eta_{0,s,-}\right)
\\&-2\int_{\partial \B(0;1)}\eta_{0,s}^2.
\end{align*} 
Introducing with a slight abuse the notation 
\[ \mathcal W(0,s;v):=\int_{\B(0;1)}|\n v_s|^2-2\int_{\B(0;1)}\left(f(0)v_{s,+}+g(0) v_{s,-}\right)-2\int_{\partial \B(0;1)}v_s^2\] we thus conclude
\begin{equation}\label{Eq:C3ScaleInvariance}
\forall s,s'>0\,, \mathcal W(0,s;\eta_0)=\Psi(0^+)=\mathcal W(0,s';\eta_0).
\end{equation}
In other words, the ``limiting Weiss functional" $\mathcal W_0(\cdot;\eta_0):=\mathcal W(0,\cdot;\eta_0)$ is constant in $s>0$. We now deduce that this implies:
\begin{equation}\label{Eq:Nancy4}
\eta_0\text{ is $2$-homogeneous: }\eta_0(r\cdot)=r^2\eta_0(\cdot).\end{equation}
Indeed, we can directly adapt the proof of Theorem \ref{Pr:C3Weiss} with $f$ and $g$ constant, up to the following minor adaptation: first, $\int_{\B(0;1)}\underline{m}_0\eta_0=\int_{\B(0;1)}\eta_{0,+}$. Second, from \eqref{Eq:MaxT}, we can write\[ \underline{m}_0=\mathds 1_{\{\eta_0>0\}}+\overline{m_1},\, \mathrm{supp}(\overline{m_1})\subset \{\eta_0=0\},\] which gives, using \eqref{Eq:JuanTabo},
\[ \int_{\B(0;1)} \frac{\partial \eta_{0,s}}{\partial s}\cdot \left(\underline{m}_0-\mathds 1_{\{\eta_{0,s}>0\}}\right)=\int_{\B(0;1)}\overline{m_1}\frac{\partial \eta_{0,s}}{\partial s}=0.\]  

This gives the existence of a constant $C>0$ such that
\begin{multline} \forall r,\, s>0\,, \mathcal W_0(s;\eta_0)-\mathcal W_0(r;\eta_0)\geq \int_r^s t\int_{\partial \B(0;1)}\left(\frac{\partial \eta_{0,s}}{\partial t}\right)^2dt\\=C\int_r^s t\int_{\partial \B(0;1)}\left|\langle \n \eta_{0,t},\cdot\rangle-\frac2t\eta_{0,t}\right|^2dt.\end{multline} We thus deduce from the fact that $\mathcal W_0(\cdot,\eta_0)$ is constant that 
\[ \forall s>0\,, \langle \n \eta_{0,s},\cdot\rangle=\frac2s\eta_{0,s}\text{ on }\partial \B(0;1).\]Then again, this implies that 
\[ \forall s>0\,, \langle \n \eta_0,\cdot\rangle=\frac{2}r\eta_{0}\text{ on }\partial \B(0;s),\] which gives \eqref{Eq:Nancy4}  and thus concludes the proof of Theorem \ref{Th:MW1} when $\Psi(0^+)>-\infty$.

\textbf{Case $\Psi(0^+)=-\infty$.} 

Let us first prove that 
\begin{equation}\label{Eq:Nancy5}
\Psi(0^+)=-\infty\Rightarrow \left(S(r)=\int_{\partial \B(0;1)}\eta_r^2\underset{r\to 0^+}\rightarrow +\infty\right).
\end{equation}
To this end, first observe that, as $\Psi(0^+)=-\infty$, we have 
\begin{equation}\label{Eq:Nancy6}-\int_{\B(0;1)}\left(f_r\eta_{r,+}+g_r\eta_{r,-}\right)+\int_{\partial\B(0;1)}\eta_r^2\to+\infty.
\end{equation}
On the other hand, since $\eta_{r,+}+M|\cdot|^2$ and $\eta_{r,-}+M|\cdot|^2$ are sub-harmonic and non-negative for some $M>0$ (see Lemma \ref{lemma:tech_sub_suph}), there exists $C_n,C_n'>0$ such that for all $r\in(0,1)$
\[
\int_{\B(0;1)}\eta_{r,+}\leq C_n\left(1+\int_{\partial\B(0;1)}\eta_{r,+}\right)\leq C_n'\left(1+\left(\int_{\partial\B(0;1)}\eta_{r,+}^2\right)^{1/2}\right)
\]
and likewise
\[
\int_{\B(0;1)}\eta_{r,-}\leq C_n'\left(1+\left(\int_{\partial\B(0;1)}\eta_{r,-}^2\right)^{1/2}\right)
.\]
As $f\,, g \in L^\infty(\B(0;1))$, \eqref{Eq:Nancy5} follows.

 Now, from the monotonicity of the Weiss functional $W$ (see Theorem \ref{Pr:C3Weiss}) for any $r\in(0;r_0)$ we have
\[\Psi(r)\leq \Psi(r_0)+C(r_0^\beta-r^\beta),\] so that 
\begin{equation}\label{eq:dege_Monot}\int_{\B(0;1)}|\n \tilde \eta_r|^2\leq \frac{1}{S(r)}\left(\Psi(r_0)+Cr_0^\alpha\right)+\frac2{\sqrt{S(r)}}\int_{\B(0;1)}(f_r\tilde \eta_{r,+}+g_r\tilde \eta_{r,-})+2\int_{\partial \B(0;1)}\tilde \eta_r^2.\end{equation}
Recalling the uniform $L^\infty$ bounds for $f$ and $g$, and that by definition $\|\tilde \eta_r\|_{L^2(\partial \B(0;1))}=1$, using the Poincaré inequality \eqref{Eq:MetzPoincare} in this estimate ensures that 
\[
\lim\sup_{r\to 0}\int_{\B(0;1)}| \n \eta_r|^2<+\infty.
\]
The same inequality \eqref{Eq:MetzPoincare} then provides a uniform $L^2(\B(0;1))$ bound on $\tilde \eta_r$, so that overall we obtain that $\tilde \eta_r$ is bounded in $H^1(\B(0;1))$. Passing to the weak limit in $H^1(\B(0;1))$ strong in $L^2(\partial\B(0;1))$ in \eqref{eq:dege_Monot}, we obtain that any limit $\tilde \eta_0$ of any subsequence is harmonic and satisfies
\[ \int_{\B(0;1)}|\n \tilde \eta_0|^2\leq 2\int_{\partial \B(0;1)}\tilde \eta_0^2\,, \int_{\partial \B(0;1)}\tilde\eta_0^2=1.\] In particular, $\tilde\eta_0\not\equiv 0$.
Furthermore, adapting the argument given in the case $\Psi(0^+)>-\infty$, we deduce that, letting $\tilde\eta_{0,s}:=\frac{\tilde\eta_0(s\cdot)}{s^2}$, we obtain 
\[ \forall s>0\,,  \int_{\B(0;1)}|\n \tilde \eta_{0,s}|^2\leq 2\int_{\partial \B(0;1)}\tilde \eta_{0,s}^2\,, \int_{\partial \B(0;1)}\tilde\eta_{0,s}^2=1\] and that 
\begin{equation}\label{Eq:Fre} \forall r\,, s>0\,,  \int_{\B(0;1)}|\n \tilde \eta_{0,s}|^2- 2\int_{\partial \B(0;1)}\tilde \eta_{0,s}^2=  \int_{\B(0;1)}|\n \tilde \eta_{0,r}|^2- 2\int_{\partial \B(0;1)}\tilde \eta_{0,r}^2.\end{equation}  Now, by the same arguments that led to Theorem \ref{Pr:C3Weiss}, the function $\mathcal W_0$ defined by 
\[ s>0\mapsto \mathcal W_0(s):=\int_{\B(0;1)}|\n\tilde\eta_{0,s}|^2-2\int_{\partial \B(0;1)}\tilde\eta_{0,s}^2\] satisfies
\begin{equation}\label{Eq:CdP1}
\forall s>0\,, \mathcal W_0'(s)=\frac2s\int_{\partial \B(0;1)}\left(\frac{\partial\tilde\eta_{0,s}}{\partial\nu}-2\tilde\eta_{0,s}\right)^2.
\end{equation}
From \eqref{Eq:Fre}, we deduce that for any $s>0$ we have
\[\frac{\partial\tilde\eta_{0,s}}{\partial\nu}-2\tilde\eta_{0,s}=0\text{ on }\partial \B(0;1),\] which implies the $2$-homogeneity of the limit.

\end{proof}

\end{document}